\documentclass[12pt]{amsart}

\usepackage{subfiles}
\usepackage{amssymb, mathtools}
\usepackage{mathabx}
\usepackage{mathrsfs}
\usepackage{hyperref}
\usepackage[capitalize,noabbrev]{cleveref}
\usepackage{tikz-cd}
\usepackage{tikzsymbols}

\creflabelformat{enumi}{#2(#1)#3}
\crefname{enumi}{}{}
\Crefname{enumi}{}{}

\numberwithin{equation}{section}

\newcommand{\idf}{B}

\newcommand{\abeins}{\sigma}

\newcommand{\lact}{\smalltriangleright}
\newcommand{\blact}{\blacktriangleright}

\newcommand{\op}{\mathrm{op}}
\newcommand{\supp}{\operatorname{supp}}

\newcommand{\monic}{\hookrightarrow}
\newcommand{\epic}{\twoheadrightarrow}
\newcommand{\ev}{\operatorname{ev}}
\newcommand{\evl}{\ev}
\newcommand{\colim}{\operatorname{colim}}
\newcommand{\im}{\operatorname{im}}
\newcommand{\id}{\operatorname{id}}
\newcommand{\End}{\operatorname{End}}

\newcommand{\hhom}{\operatorname{hom}}
\newcommand{\Hom}{\operatorname{Hom}}
\newcommand{\Prim}{\operatorname{Prim}}
\newcommand{\EndA}{\End(A)}
\newcommand{\Der}{\operatorname{Der}}
\newcommand{\Derk}{\Der}
\newcommand{\DerA}{\Der(A)}

\newcommand{\A}{\mathcal{A}}
\newcommand{\B}{\mathcal{B}}
\newcommand{\C}{\mathcal{C}}
\newcommand{\DD}{\mathcal{D}}
\newcommand{\DDA}{\DD_A}
\newcommand{\DK}{\D_K}
\newcommand{\D}{\mathscr{D}}
\newcommand{\J}{\mathscr{J}}
\newcommand{\DA}{\D_A}

\newcommand{\Dcusp}{\D_{k[t^2,t^3]}}

\newcommand{\Dpoly}{\D_{k[t]}}

\newcommand{\DB}{\D_B}
\newcommand{\U}{\mathscr{U}}
\newcommand{\UA}{\U_A}

\newcommand{\UK}{\U_K}
\newcommand{\N}{\mathbb{N}}
\newcommand{\Z}{\mathbb{Z}}

\newcommand{\ZZ}{\mathcal{Z}}
\newcommand{\dd}{\mathrm{d}}

\newcommand{\OmA}{\Omega^{1}_{A}}

\newcommand{\homk}{\mathrm{Hom}}
\newcommand{\otimesk}{\otimes}
\newcommand{\coring}[2]{\(#1|#2\)}

\renewcommand{\emptyset}{\varnothing}

\renewcommand{\P}{\mathscr{P}}
\newcommand{\PA}{\P_A}

\newcommand{\vspan}{\mathrm{span}}
\newcommand{\gr}{{\operatorname{gr}}}
\newcommand{\Sym}{\operatorname{Sym}}
\newcommand{\SymA}{\Sym_{A}}
\newcommand{\SymnA}{\Sym^n_{A}}

\newcommand{\Gmod}{G\mbox{-}\mathbf{Mod}}
\newcommand{\Hmod}{H\mbox{-}\mathbf{Mod}}

\newcommand{\Dmod}{\DA\mbox{-}\mathbf{Mod}}

\renewcommand{\Cat}{\mathbf{C}}
\newcommand{\DCat}{\mathbf{D}}
\newcommand{\Amod}{A\mbox{-}\mathbf{Mod}}
\newcommand{\AmodA}{A\mbox{-}\mathbf{Mod}\mbox{-}A}
\newcommand{\Bmod}{B\mbox{-}\mathbf{Mod}}
\newcommand{\jmod}{\mbox{-}\mathbf{Mod}}

\theoremstyle{plain}
\newtheorem*{theoremm}{Theorem}
\newtheorem{theorem}{Theorem}[section]
\newtheorem{lemma}[theorem]{Lemma}
\newtheorem{proposition}[theorem]{Proposition}
\newtheorem{corollary}[theorem]{Corollary}
\newtheorem{conjecture}[theorem]{Conjecture}
\newtheorem*{question}{Question}

\theoremstyle{definition}

\newtheorem{definition}[theorem]{Definition}

\theoremstyle{remark}

\newtheorem{example}[theorem]{Example}
\newtheorem{remark}[theorem]{Remark}

\begin{document}
\title[Differential operators
on cusps]{The ring of
differential operators on 
a monomial curve is a 
Hopf algebroid}
\author{Ulrich Krähmer}
\author{Myriam Mahaman}

\begin{abstract}
This article considers cuspidal curves
whose coordinate rings are numerical
semigroup algebras.
Using a general result about descent
of Hopf algebroid structures, their rings of differential
operators 
are shown to be 
cocommutative and conilpotent
left Hopf algebroids. If the
semigroups are symmetric so that the
curves are Gorenstein, they are full
Hopf algebroids (admit an anti\-pode). 
\end{abstract}


\maketitle

\tableofcontents

\section{Introduction}
\subsection{The smooth case}
Let $k$ be a field of characteristic
0, $A=k[X]$ be
the coordinate ring of a
smooth 
affine variety $X$,
and $\DA$ be the algebra of
$k$-linear differential
operators on $A$.
As will be recalled in
Sections~2 and 3, $\DA$ 
is then generated as a $k$-algebra
by the subalgebra $A \subseteq 
\DA$ of multiplication
operators and the Lie algebra 
$\DerA \subseteq \DA$ of 
vector fields on
$X$ (derivations
on $A$). This yields  
an isomorphism of $A$-rings
$\DA \cong \UA$ with the
universal enveloping
algebra of the 
Lie--Rinehart algebra 
$(A,\DerA)$. 

In particular, $\DA$ is 
a left Hopf algebroid
over $A$ (see Section~4), which means that 
if $M,N$ are  
$\DA$-modules, then both
$M \otimes _A N$ and 
$ \mathrm{Hom} _A(M,N)$ carry
$\DA$-module
structures rendering  
$\Dmod$ a closed
monoidal category. 
Hence a natural question arises:

\begin{question}
For which $k$-algebras $A$ is $\DA$ a 
left Hopf algebroid?
\end{question}

This was one
of Sweedler's motivations for
introducing Hopf algebroids
(called by him
$\times_A$-Hopf algebras) 
in
\cite{sweedlerGroupsSimpleAlgebras1974}.
Theorem 8.7  therein
introduces the class of algebras 
\emph{with almost finite projective differentials}
for which
the answer is positive.
This class includes smooth and
\emph{purely inseparable} (see \cref{weirdex}) algebras.

\subsection{Monomial curves}
In the present article, we use a
descent theory approach to 
answer the question for
the coordinate rings 
$A$ of cuspidal curves given by
numerical semigroups:

\begin{theorem}\label{main1}
Let $k$ be a field of
characteristic 0, 
$\A \subseteq (\mathbb{N},+)$ be
a numerical semigroup, that
is, a monoid generated by 
$p_1,\ldots,p_n \in \mathbb{N}$ 
with 
$\mathrm{gcd}(p_1,\ldots,p_n)=1$, 
and let
$A$ be its monoid algebra over $k$. 
Then $\DA$ is canonically a
cocommutative and conilpotent left Hopf
algebroid over $A$. If
$\A$ is symmetric, it is a full Hopf
algebroid.
\end{theorem} 

We refer to Section~7 for more
definitions and references
concerning numerical
semigroups. 

\subsection{The basic example}
\label{sec:basicex}
When $p_1=2,p_2=3$, then
$\Dcusp$ is  
generated by the coordinate
functions $t^2,t^3$
and the differential operators 
$$
	D_0 \coloneqq t \partial ,\quad 
   L_{-2} \coloneqq \partial^2
- \frac{2}{t} \partial,\quad 
    L_{-3} \coloneqq \partial^3 - 
	\frac{3}{t} \partial^2 + \frac{3}{t^2}
  \partial
$$
of orders $1,2$ and $3$,
respectively,
where $ \partial
\coloneqq \frac{\dd }{\dd t}$
\cite{smithDifferentialOperatorsAffine1988}.
For this example, we compute the
Hopf algebroid structure explicitly:
 
\begin{theorem}
	\label{cusptheorem}
Let $k$ be a field of
characteristic 0 and
$A:=k[t^2,t^3]$. Then 
the $k$-algebra $\DA$
carries the structure of a
cocommutative, conilpotent,
and involutive 
Hopf algebroid over $k[t^2,t^3]$
whose counit, coproduct and
antipode are determined by
\begin{align*}
	\varepsilon (D_0) 
&=
\varepsilon (L_{-2}) = 
	\varepsilon (L_{-3}) = 0,\\	
		  \Delta \left( D_{0} \right)
&=
			D_{0} \otimes_{A} 1 + 1 \otimes_{A} D_{0},\\
		  \Delta \left( L_{-2} \right)
&=
			L_{-2} \otimes_{A} 1
			+ 2 D_{0} \otimes_{A} (D_{0}-1) L_{-2}
			- 2 L_{1} \otimes_{A} L_{-3}
			+ 1 \otimes_{A}
L_{-2},\\
		  \Delta \left( L_{-3} \right)
		  &=
		  L_{-3} \otimes_{A} 1
			+ 3 L_{-2} \otimes_{A} L_{-1}
			- 3 L_{-1} \otimes_{A} L_{-2}
			\\
&\quad\quad + 6 D_{0} \otimes_{A} (D_{0} - 1)L_{-3}
			- 6 L_{1} \otimes_{A} L_{-2}^{2}
			+ 1 \otimes_{A}
L_{-3},\\
			S(D_{0}) &= - D_{0} +
1,\quad
			  S(L_{-2}) =
L_{-2},\quad 
			  S(L_{-3}) =
-L_{-3},
		  \end{align*}
where
$L_{1} := t^{3}(D_{0}-1)L_{-2} 
- t^{4}L_{-3}$ and 
$L_{-1} := t^{2}(D_{0}-1) L_{-3} - 
t^{3}L_{-2}^{2}$. 
\end{theorem}

Smith and Stafford, and,
independently, Muhasky
have shown that $\Dcusp$ is
Morita-equivalent to the ring
$\Dpoly$ of differential
operators on the
normalisation $k[t]$
of $k[t^2,t^3]$
\cite{smithDifferentialOperatorsAffine1988,muhaskyDifferentialOperatorRing1988}, and Ben-Zvi and Nevins
have extended this result to suitable
quotients of Cohen--Macaulay
varieties
\cite[Theorem~1.2]{ben-zviCuspsModules2004}. As $k[t]$
is smooth, the Weyl algebra 
$\D_{k[t]}$ is a Hopf
algebroid,
and with the above Hopf algebroid
structure on $\Dcusp$, the  
Morita-equivalence becomes 
strict monoidal, see 
Theorem~\ref{thmsemigroupalg}.

\subsection{Motivation}
We find the above results
interesting for several 
reasons:

First, being a Hopf algebroid has
implications for the homological
properties of $\DA$ and yields
additional algebraic structures 
on $ \mathrm{Ext} $-,
\mbox{\(\mathrm{Tor}\)-,} and $ \mathrm{Cotor} $-groups
over $\DA$.
For example, the derived version 
of the internal
hom adjunction
$ \mathrm{Hom}_{\DA}(M,N) \cong 
\mathrm{Hom} _{\DA}(A, \mathrm{Hom}
_{A}(M,N))$ is a spectral sequence
which in good cases allows one to
compute 
$ \mathrm{Ext} _{\DA}(M,N)$ using 
a projective resolution of the 
$\DA$-module $A$. See for example
\cite{kowalzigDualityProductsAlgebraic2010,
kowalzigCyclicStructuresAlgebraic2011,
kowalzigBatalinVilkoviskyStructures2014}
for more information. 

Second, Moerdijk and Mr\v{c}un have
obtained a Hopf algebroid version  
of the Cartier--Milnor--Moore theorem 
\cite[Theorem~3.1]{moerdijkUniversalEnvelopingAlgebra2010}. 
The classical theorem
asserts that 
over a field of characteristic zero,
the cocommutative
and conilpotent Hopf algebras
are precisely the
universal enveloping algebras
of Lie algebras
(see e.g.~\cite[Theorem~3.5]{larsonCocommutativeHopfAlgebras1967}
and \cite[Theorem
V.4]{sweedlerCocommutativeHopfAlgebras1965}, 
neither Cartier nor Kostant had
published their original
proofs). The
Hopf algebroid version 
establishes a correspondence
between $A$-projective
Lie--Rinehart algebras and
cocommutative, conilpotent
Hopf algebroids satisfying some
projectivity assumption. As 
Theorem~\ref{main1}
demonstrates, a
classification beyond the
projective case could be a
non-trivial but interesting endeavour.

Third, the underlying coalgebras
over $A$ are neat examples that
demonstrate some subtleties 
of the theory
of coalgebras over commutative rings
that are not fields:
at first, the formulas 
in Theorem~\ref{cusptheorem} 
do not seem to 
support the claim that $
\Dcusp$
is cocommutative and
conilpotent. On closer
inspection, one convinces
oneself of the cocommutativity 
which is veiled by the
relations in the
$k[t^2,t^3]$-module $\Dcusp$. 
Concerning conilpotency,
one observes 
a phenomenon that cannot occur for
coaugmented coalgebras over
fields: 
$\Dcusp$ is
an example of a coaugmented
coalgebra whose primitive 
filtration is not a coalgebra
filtration (see
\cref{examplenotfilt}). 

Next, the question whether
$\DA$ admits an antipode
aligns nicely with the
theory of symmetric numerical
semigroups. Their
monoid algebras are Gorenstein
\cite{kunzValuesemigroupOnedimensionalGorenstein1970},
and this yields an 
isomorphism of $A$-rings 
$S \colon \DA \rightarrow 
\D^\op_A$. See 
\cite{quinlan-gallegoSymmetryRingsDifferential2021a}
for a recent elementary proof of
this fact
due to 
Quinlan-Gallego and 
\cite[Corollary 4.9]{ 
yekutieliResiduesDifferentialOperators1998}
for a more general (and less
elementary) proof due to Yekutieli. 

Last but not least, we present
Theorem~\ref{main1} as an
application of a general theory of
descent of Hopf algebroid
structures that might be of interest
in itself and can be applied more
widely and in particular to
the study of differential
operators in characteristic
$p$. We briefly sketch this in 
the introduction for
experts on Hopf algebroids,
see Section~5 for the details.

\subsection{Descent for Hopf
algebroids}
Let $k$ be any commutative ring,
and assume 
we are given a bialgebroid 
$C$ over a commutative 
$k$-algebra $K$ and 
two subalgebras $A,B \subseteq K$. 
The unit object in the 
monoidal category of $C$-modules is
$K$, with $c \in C$ acting on 
$b \in K$ as $ \varepsilon (b \blact
c) = \varepsilon (cb)$ where 
$ \varepsilon $ is the counit
of $C$. 

Now define 
\begin{align*}
	C(B,A)
&:= \{c \in C \mid 
\varepsilon (cb) \in A \,\forall b \in
B\},\\
	R(B,A)
&:= \{ r \colon  
C(B,A) \rightarrow 
A,c \mapsto 
\textstyle\sum_i  \varepsilon (a_icb_i)
\mid a_i \in A,b_i \in B\}.
\end{align*}
Our main result provides criteria 
for coalgebra, bi- and Hopf algebroid
structures to descend to 
$C(B,A)$ if $C \cong K \otimes _A 
C(B,A)$. We are interested
especially in the case that $A
\subseteq K$ is a localisation 
(in our application, $K$ corresponds
to the smooth locus
of a singular variety with
coordinate ring $A$), so $K$ is
typically not faithfully flat over
$A$ (which is usually assumed
in results involving descent). 
However, we prove that it
suffices to verify that $C(B,A)$ is 
$R(B,A)$-locally
projective (see
Section~\ref{locprocsec} for this
notion):

\begin{theorem}\label{descenttheorem}
If 
\(C(B,A)\) is 
\(R(B,A)\)-locally
projective and 
\[ 
	\mu \colon  
	K \otimes _A C(B,A) 
	\rightarrow C,\quad
	x \otimes _A c 
	\mapsto xc 
\]
is surjective,  
then we have:
\begin{enumerate}
\item 
$(\Delta , \varepsilon )$
uniquely restricts to a
\coring{B}{A}-coring structure on  
\(C(B,A)\).
\item 
If $B=A$, this turns
\(C(A,A)\) into a bialgebroid over
\(A\).
\item If $C$ is in addition 
a left Hopf
algebroid, then $C(A,A)$ is a left
Hopf algebroid. 
\end{enumerate}
\end{theorem}

\subsection{Structure of the paper}
The paper is organised as
follows: 
Sections~2-4 contain
preliminary material
on differential operators and
Lie--Rinehart algebras, on smooth
and étale algebras, 
respectively on bi- and 
Hopf algebroids.
Experts on these topics can safely
skip these sections, they do not
contain novel material but are there
in order to keep the paper
self-contained. 
Section~5 
contains the proof of the 
above Theorem~\ref{descenttheorem}. 
For this, we first recall the
theory of $A$-modules $M$ that are
locally projective relative to a
subset $R \subseteq M^*=
\mathrm{Hom} _A(M,A)$. 
Section~6 applies this theory 
to rings of differential operators
in general, and Section 7 focuses on
the monomial curves from
Theorem~\ref{main1} and the specific
example $k[t^2,t^2]$ from 
Theorem~\ref{cusptheorem}.  

As this is a long paper, we will
include a more detailed summary of
each section in its preamble.  

\subsection{Conventions}
From now on, 
$k$ is a fixed commutative ground
ring. Unless specified explicitly
otherwise, ``algebra'' means
``unital associative algebra
over $k$'', that
is, ``morphism $k \rightarrow Z(H)$
of rings'', where $Z(H)$ is the centre
of the ring $H$.

The main results remain true 
if one replaces $k$-modules
by other closed symmetric
monoidal categories 
with sufficiently
good properties, so  
we 
suppress subscripts $k$ and 
write 
$ \mathrm{Hom} (V,W)$ instead of 
$ \mathrm{Hom} _k(V,W)$ and
similarly 
$\otimesk, \Der(A,M), 
\DA,\OmA,\ldots$ 
and only
specify the ground ring relative to
which we work when it is different
from $k$.
The symbols $\epic$ and
$\monic$ are used to denote
surjective and injective maps
(rather than epi- respectively
monomorphisms).

Furthermore, $A,B$ and $K$ are
reserved to denote commutative
algebras coming with algebra
morphisms $A \rightarrow K,B
\rightarrow K$. A guiding
example that one can always keep in
mind when reading the paper
is
$$
	A = k [t^2,t^3] 
	\subset 
	B = k [t] 
	\subset 
	K = k [t,t^{-1}],
$$
$k$ being any field of
characteristic 0.

Last but not least, 
for any $A$-module $M$, 
we denote by $M^*:=\mathrm{Hom}
_A(M,A)$ its dual. 

\subsection*{Acknowledgements}
It is a pleasure to thank
Paolo Saracco for pointing out
to use the reference
\cite{cohnRemarkBirkhoffWittTheorem1963},
as well as Toby Stafford and Paul
Smith for explaining to us some
details of their work on
differential operators on singular
curves. 
Ulrich Krähmer is supported
by the DFG grant
``Cocommutative comonoids''
(KR 5036/2-1).
Myriam Mahaman is supported 
by the COST Grant HORIZON-MSCA-2022-SE-01-01 CaLIGOLA.

\section{Differential
operators}\label{differentialoperatorssec}
Throughout the paper, $k$ is a
commutative ring, ``algebra''
means ``unital associative algebra
over $k$'', and $A$ is a
commutative algebra. 
This section contains
background material on the 
$A$-ring 
$\DA$ of 
differential
operators on $A$, 
and on its relation
to the universal enveloping
algebra $\UA$ of the Lie--Rinehart
algebra $\DerA$ of 
derivations of $A$.

\subsection{Filtered and graded
(co)algebras}
There will be several filtered
and graded objects in this
paper, so we begin by fixing
our notation and conventions
concerning these. 

\begin{definition}
A \emph{filtered object} $V$ in a
category $\Cat$ is a chain of
monomorphisms $f_i \colon V_i
\rightarrow V_{i+1}$, $i \in
\mathbb{Z} $. 
The graded object 
$ \gr(V)$ \emph{associated
to} a filtered object $V$ is
given by $\gr(V)_i :=
V_i/V_{i-1}= \mathrm{coker}\,
f_{i-1}$. 
\end{definition}

Here and elsewhere, we tacitly 
assume that all occurring (co)limits
such as (co)products, (co)kernels
etc.~exist. 

More generally, one may replace 
$ \mathbb{Z} $ by any
\emph{directed
set}, that is, a preorder $(\Lambda,
\le )$ such that any two elements 
$ \lambda , \mu \in \Lambda $ have
an \emph{upper bound} 
$ \nu \in \Lambda $ with 
$ \lambda \le \nu,\mu \le \nu $,
and consider functors $V \colon 
\Lambda \rightarrow \Cat$; these
will be called 
\emph{directed systems} in 
$\Cat$ while functors 
$\Lambda^\op \rightarrow \Cat$
will be called \emph{inverse
systems}. If $ \Lambda $ is 
discrete ($
\lambda \le \mu \Leftrightarrow
\lambda = \mu $), we call 
$ V$ a \emph{$ \Lambda $-graded
object} in $\Cat$. As is customary,
we will often use the same symbol
$V$ to denote the colimit of a
directed system $V$. 

\begin{example}
If $V$ is a $ \mathbb{Z} $-graded 
respectively a filtered 
$k$-module, then we denote  
$\bigoplus_i V_i$ respectively 
$\bigcup_i V_i$ also by $V$. 
\end{example}

With natural transformations as
morphisms, the directed respectively
inverse systems in $\Cat$ (and in
particular filtered and graded
objects) form again categories
and $\gr$ is a functor from filtered
to $ \mathbb{Z} $-graded objects.
If $ \Lambda $ and 
$\Cat$ are both monoidal categories, 
the category of 
directed systems is also
monoidal with the \emph{Day
convolution} as monoidal structure,
see e.g.~\cite[
Section 6.2]{loregianCoEndCalculus2021}.

\begin{example}
The tensor product 
of $ \mathbb{Z}
$-graded $k$-modules 
$V,W$ has  
$$
	(V \otimes W)_l = 
	\bigoplus_{m + n = l}
	V_m \otimes W_n.
$$
In the tensor product of filtered
$k$-modules,
$(V \otimes W)_l$ is the image 
of the canonical morphism
$$ 
	\bigoplus_{m + n \le l} V_m
	\otimes W_n
	\rightarrow 
	V \otimes W
$$ 
induced by
the tensor products of the 
embeddings $V_m \rightarrow V,
W_n \rightarrow W$.
\end{example}

Thus we may consider \emph{graded}
respectively \emph{filtered
(co)algebras}, that is, 
(co)monoids in the monoidal 
categories of graded respectively
filtered $k$-modules, as well as the categories of graded
respectively filtered (co)modules
over such (co)algebras. We will
also iterate the process and
consider for example filtered
coalgebras $D$ in the category of
graded modules over
a (commutative) graded algebra $A$.

\begin{remark}\label{internalhomrem}
The category of $ \mathbb{Z}
$-graded $k$-modules is closed
monoidal: for any graded
modules $V,W$, their 
internal hom
is the graded module with $[V,W]_d := 
\{
\varphi \in \homk(V,W) \mid 
\varphi (V_i) \subseteq 
W_{i+d} \,\forall i \in
\mathbb{Z} \}$. 
\end{remark}

\subsection{Almost commutative
algebras}

Let $H$ be a filtered algebra.

\begin{definition}
$H$ is \emph{almost
commutative} if $\gr(H)$ is
commutative. 
\end{definition}

So explicitly, $H$ is almost
commutative if 
for all $a \in H_i,b \in
H_j$, we have
$ab-ba \in
H_{i+j-1}$.
In this case, 
the \emph{Poisson bracket} of 
$[a] \in \gr(H)_i,[b] \in 
\gr(H)_j$ is 
$$
	\{[a],[b]\} := 
	[ab-ba] \in 
	\gr(H)_{i+j-1}.
$$

The
following is an
immediate consequence of the
above definition; in short,
it states that
$(\gr(H),\{-,-\})$ is a
\emph{graded Poisson algebra}
(with Poisson bracket of degree
-1):

\begin{proposition}\label{almostcomm}
If $H$ is an almost
commutative filtered algebra, 
then 
$(\gr(H),\{-,-\})$ is a
Lie algebra over $k$, and
for all $X,Y,Z \in
\gr(H)$, we have 
$
	\{X,YZ\}
	=
	\{X,Y\}Z + 
	Y \{X,Z\}
.$
\end{proposition}

The aim of this
Section~\ref{differentialoperatorssec} is to
recall that any commutative
algebra $A$ defines three 
almost commutative filtered algebras
with $H_0=A$:
 
\begin{enumerate}
\item the universal enveloping
algebra $\UA$ of the
Lie--Rinehart algebra 
$ \DerA$ of all $k$-linear
derivations $  
A \rightarrow A$, 
\item the algebra $\DA$ of
differential operators on $A$,
and
\item its subalgebra $\DDA$ generated by
$\DA^1=A \oplus \DerA$.
\end{enumerate}
At the end we will recall that 
in characteristic 0, they 
all agree when $\DerA$ is a finitely
generated projective $A$-module 
(Theorem~\ref{luis}); in 
Section~\ref{smoothalgsec} this will
be related to the smoothness of $A$.

\subsection{Augmented
$A$-rings}\label{augmentedsec}
We adopt the following
terminology: 

\begin{definition}
An \emph{\(A\)-ring} is an
algebra morphism 
$\eta \colon A \rightarrow H$. 
An \emph{augmented
$A$-ring} is an $A$-ring
$\eta \colon A \rightarrow H$ 
together with a 
$k$-linear map $ \varepsilon \colon H
\rightarrow A$ for which 
$ \varepsilon \circ \eta =
\mathrm{id} _A$ and 
$ \mathrm{ker}\, \varepsilon 
\subset H$
is a left ideal.
\end{definition}

In an augmented $A$-ring, we
usually 
suppress $\eta $ and 
identify \(A\) with its image
in \(H\), that is, view $A$ as
a subalgebra of $H$ so that 
we have an $A$-module decomposition 
$ H = A \oplus \mathrm{ker}\,
\varepsilon $.

\begin{remark}\label{leftright1}
Even though $A$ is assumed to be
commutative, 
$\mathrm{im}\, \eta$ is not
assumed to lie in the centre of
$H$. That is, 
$H$ is a monoid in the category
$\AmodA$ of $A$-$A$-bimodules 
rather than a
monoid in the monoidal category 
$\Amod$ of
$A$-modules, which would be 
an $A$-algebra. 
Since $A$ is commutative, 
there is no actual difference
between left and right $A$-actions,
so there are 
two forgetful functors 
$\AmodA \rightarrow \Amod$. When
considering $H$ as an $A$-module
with respect to left multiplication,
$$
	A \otimesk H \rightarrow 
	a \otimesk h \mapsto 
	a \lact h := \eta (a) h,
$$
we will denote the resulting 
$A$-module
just by $H$. When forgetting the
left action, we will instead 
denote the resulting
$A$-module by $\bar H$,
$$
	A \otimesk \bar H 
\rightarrow \bar H,\quad 
	a \otimesk h \mapsto 
	a \blact h := 
	h \eta (a).
$$ 
\end{remark}
\begin{remark}\label{leftright2}
Note also
that an augmentation
$ \varepsilon $ of an $A$-ring
is not an algebra map in general.
Instead, it 
can be equivalently described
as a left $H$-module structure
on $A$ that extends the
multiplication in $A$; in
terms of $ \varepsilon $ and $
\eta $, the
action of $h \in H$ on $a \in
A$ yields
$ 
	\varepsilon (h \eta (a))
= \varepsilon (a \blact h),
$
while $ \varepsilon $ is by
construction $A$-linear with respect
to the left $A$-action, 
$ \varepsilon (a \lact h)= 
\varepsilon (\eta (a) h)= 
a \varepsilon (h)$. 
\end{remark}

\begin{example}\label{endkaex}
Consider the canonical
injection 
\[
	\eta \colon A
\hookrightarrow
\EndA \coloneqq \homk(A,A),\quad
	a \mapsto (x \mapsto ax)
\]
that assigns to each $ a \in
A$ the corresponding
multiplication operator on $A$.
This turns $ \EndA $ into an
augmented $A$-ring with
augmentation  
$ \varepsilon (\varphi ):=
\varphi (1)$. Yet an equivalent way
to define an augmentation of an
$A$-ring $H$ is as an
$A$-ring morphism 
$ \hat \varepsilon \colon 
H \rightarrow \EndA$. 
\end{example}

By applying the above to
the opposite ring $H^\op$ of
$H$, we obtain that 
right $H$-module structures on 
$A$ that extend the multiplication
in $A$ correspond bijectively to 
maps
\(\bar{\varepsilon} \colon \bar{H}
\to A\) that are $A$-linear, 
$$
	\bar \varepsilon (a
\blact h) = \bar \varepsilon (h \eta
(a)) = \bar \varepsilon (h)a = 
a\bar \varepsilon (h),\quad
	a \in A,h \in H,
$$
satisfy $ \bar \varepsilon
\circ \eta = \mathrm{id} _A$, and 
for which 
$ \mathrm{ker}\, \bar \varepsilon $ 
is a right ideal in $H$; the
corresponding right action of $h \in
H$ on $a \in A$ yields 
$ \bar \varepsilon (a \lact h)$. 

\subsection{Differential
operators}\label{diffopsubsect}
Now we recall the notion of
a differential operator over
$A$. 
We refer e.g.~to
\cite{coutinhoPrimerAlgebraicDModules1995,
mcconnellNoncommutativeNoetherianRings2001, 
grothendieckElementsGeometrieAlgebrique1967}
for more information. 

By definition, $ \eta $
identifies $A$ in
Example~\ref{endkaex} with the 
\emph{differential
operators of order 0} on $A$.
The differential operators of
higher order are defined
inductively as follows:

\begin{definition}
The $A$-ring of 
\emph{differential
operators} over \(A\)
  is the almost commutative 
subalgebra 
\[
	\DA = 
	\bigcup_{n \in \N} 
	\DA^{n} \subseteq 
	\EndA,
\]
where
\(\DA^{0} := A\) and for $n \ge
1$, we set
\[
	\DA^{n} = \{D \in \EndA
\mid Da - aD \in \DA^{n-1} \,
\forall a \in A \}.
\]
\end{definition}

\begin{remark}
For $n \ge 1$,
the embedding 
$\DA^n \hookrightarrow \DA^{n+1}$
does in general not split as a
morphism of $A$-modules, while
for $n=0$ we have a canonical
isomorphism
$
  \DA^{1} 
	\rightarrow   A \oplus
\DerA$, $D \mapsto (D(1),
D-D(1))$.
\end{remark}

The $A$-ring $\DA$ is by definition
an almost commutative filtered
algebra. We use upper indices
for the filtration given by the
order of a differential operator
since in our main examples, 
$A$ is itself a $ \mathbb{Z}
$-graded algebra, and then 
we use lower indices to denote the
differential operator of degree 
$d \in \mathbb{Z} $,
$$
	\D_{A,d} := 
	\{ D \in \DA \mid 
	D(A_r) \subseteq A_{r+d}
	\forall r \in \mathbb{Z}\},
	\quad 
	\D_{A,d}^n :=
	\DA^n \cap \D_{A,d}. 
$$

Note that in general we have 
$\bigoplus_d 
\D_{A,d} \subsetneq
\DA$, but we will show in 
Corollary~\ref{gradeisgrade} 
that equality holds when 
$A$ is \emph{essentially of finite
type} (is 
a localisation of a finitely
generated algebra).

Finally, we define the following
subalgebra of $\DA$:

\begin{definition}
We denote by 
$$
	\DDA:=\bigcup_{n \ge 0}
\DDA^n,\quad
	\DDA^n:=\mathrm{span}_A\{X_1
\cdots X_m \mid X_i \in
\DerA, 0 \le m \le n\} 	
$$
the augmented $A$-subring of
$\DA$ generated (as an algebra) 
by $\DA^1$.
\end{definition}

\subsection{Principal parts}
For an alternative description of
$\DA$, consider $A \otimesk A$ as
$A$-module with multiplication in
the first tensor component as
$A$-action so that
we have an isomorphism 
$$
	\EndA = 
	\homk(A,A) \cong 
	(A \otimesk A)^* =
	\mathrm{Hom} _A(A \otimesk A ,
A).
$$
Here and elsewhere, $M^*$ is the
dual of a left $A$-module
$M$.
This identifies $D \in \EndA$ with 
the
$A$-linear functional 
on $A \otimesk A$ given by
$a \otimesk b \mapsto 
	a D(b)$. 
A straightforward computation (see
e.g.~\cite[Proposition~2.2.3]{heynemanAffineHopfAlgebras1969})
shows that 
$\DA^n$ consists under this
isomorphism
precisely 
of those $A$-linear
functionals on $A \otimesk A$ 
that vanish on $I_A^{n+1}$, where 
$I_A$ is the
kernel of the multiplication map
$ \mu _A$ of $A$.
Thus we have
\begin{equation}\label{sweedlerprincipal}
	\DA^n \cong
	(\P_A^n)^* =
	\mathrm{Hom} _A(\P_A^n,A),
	\text{ where }  
	\P_A^n:= 
	(A \otimesk A)/I_A^{n+1}.
\end{equation}

\begin{definition}\label{principalpartdef}
$\P_A^n$ is 
the \emph{module of 
principal parts
of order \(n\)} of \(A\). 
\end{definition}

This terminology will be explained
further in
Remark~\ref{symbolcalculus}. 

\subsection{Lie--Rinehart
algebras}\label{lra}
Here we recall from
\cite{rinehartDifferentialFormsGeneral1963} 
the notion of a
Lie--Rinehart algebra $L$ over
$A$ and 
of its universal
enveloping algebra. 
 
\begin{definition}
A \emph{Lie--Rinehart
algebra}
over \(A\) is a Lie algebra 
$L$ over $k$ together
with an \(A\)-module
structure 
$a \otimesk X \mapsto a \cdot
X$ on $L$ and an
$L$-module structure 
$X \otimesk a \mapsto X(a)$ 
on $A$
which satisfy 
\[
	X(ab) = a X(b) +
X(a)b,\quad 
	a X(b) = 
	(a \cdot X)(b),\]\[ 
	[X, a \cdot Y] = 
	a \cdot [X,Y] + 
	X(a) \cdot Y,
\]
for $X,Y \in L$, $a,b \in A$.
The Lie algebra morphism  
\(\omega \colon L \to \DerA\),
\(X \mapsto X(-)\) is referred
to as the \emph{anchor map}. 
\end{definition}

\begin{remark}\label{liealgebroidrem}
When $A$ is the $ \mathbb{R}
$-algebra $C^\infty(X)$ 
of smooth functions
on a smooth manifold $X$ and $L$
is the module of sections of a
locally trivial real vector
bundle over $X$, one also
refers to $L$ as a \emph{Lie
algebroid}.
Rinehart himself called
Lie--Rinehart algebras 
$(k,A)$-Lie algebras. The
terminology Lie--Rinehart
algebra goes back to
Huebschmann
\cite{huebschmannPoissonCohomologyQuantization1990}, who
however uses it for the pair
$(A,L)$ in which both $A$ and
$L$ are variable. 
\end{remark}

The following 
consequence of
Proposition~\ref{almostcomm}
says that Lie--Rinehart
algebras are in a sense 
infinitesimal first order 
approximations
of almost commutative
algebras:

\begin{corollary}\label{almostcomm2}
If $H$ is an almost
commutative algebra, then 
$\{-,-\}$ turns 
$L=H_1/H_0$ into a Lie--Rinehart
algebra over $H_0$ with
anchor given by  
$
	X (a) :=
	\{X,a\}
$ 
and $a \cdot X$ given by left
multiplication in $\gr(H)$.
\end{corollary}

\subsection{Universal enveloping
algebras}
\label{univenvalgsec}
Conversely, any Lie--Rinehart
algebra $L$ 
determines an almost commutative
augmented $A$-ring with the
following universal property:

\begin{definition}
The \emph{universal
enveloping algebra} of a
Lie--Rinehart algebra $L$ over
$A$ is the universal pair 
$(\UA(L),\rho)$, where 
\begin{enumerate}
\item $\UA(L)$ is an $A$-ring 
and 
\item 
$
	\rho \colon L \rightarrow
\UA(L)  
$
is a $k$-linear map satisfying
\begin{align*}
	\rho ([X,Y]) &= 
\rho (X) \rho (Y) - \rho (Y)
\rho (X),\\
	\rho (a \cdot X) &= 
	\eta (a) \rho (X),\\
\eta(X (a)) 
&= 
	\rho (X)\eta(a) - \eta (a)
\rho (X),
\end{align*}
where $a \in A$ and $X,Y \in
L$. 
\end{enumerate}
\end{definition}

The algebra $\UA(L)$ is  
constructed as a quotient of the
tensor algebra 
$T_A (L \oplus A)$, 
of the $A$-bimodule $L \oplus A$
with left and right action 
$$
	a (X,b) := 
	(a \cdot X,ab),\quad 
	(X,b)a := 
	(a \cdot X, X(a) + ba).
$$
The filtration with respect
to which $\UA(L)$ is almost
commutative is similarly to
that of $\DDA$ given by
\begin{equation}
\label{pbwfiltration}
	\UA(L)^n:=\mathrm{span}_A\{\rho
(X_1) 
\cdots \rho (X_m) \mid X_i \in
	L,m \le n\}.
\end{equation}
We refer to \cite[p197]{rinehartDifferentialFormsGeneral1963}
for the details. 

Therein, Rinehart has also 
extended the
Poincaré-Birkhoff-Witt
theorem from Lie to
Lie--Rinehart algebras:

\begin{theorem}
[\cite{rinehartDifferentialFormsGeneral1963}, Theorem~3.1 on
p198]
\label{pbwtheorem}
If the $A$-module $L$ is
projective, then $ \rho $
induces an $A$-algebra
isomorphism
\(\SymA L \to \gr (\UA(L))\).
\end{theorem}

Here, $\SymA M$ denotes the
symmetric algebra of an
$A$-module $M$, that is, the free
commutative $A$-algebra on
$M$. 

The universal property of
$\UA(L)$ implies in
particular that the 
anchor map $ \omega $ extends
uniquely to a
morphism of $A$-rings 
\begin{equation}\label{omegahat}
	\hat \varepsilon \colon 
	\UA(L) \rightarrow 
	\DDA \subseteq \DA
\subseteq \EndA,
\end{equation}
and in this way, 
$\UA(L)$ is canonically 
an augmented
$A$-ring. 
By construction, 
we thus have 
$ \omega = 
\hat \varepsilon \circ \rho $,
which implies in particular:

\begin{lemma}
If $ \omega \colon L
\rightarrow \DerA$ is injective, so
is $ \rho \colon L \rightarrow 
\UA(L)$. 
\end{lemma}

\begin{remark}
In general, $ \rho $ is not
injective; in particular,
when $A=k$, then a
Lie--Rinehart algebra is just a
Lie algebra, and $\UA(L)$ is its
universal enveloping algebra.
In this setting, Cohn has
shown that $ \rho $ is
injective if $L$ is
torsion-free as an abelian
group under addition, and 
given an example 
for which it is not
\cite[Section~5 on
p202]{cohnRemarkBirkhoffWittTheorem1963}.
\end{remark}

For a general Lie--Rinehart algebra,
$\hat \varepsilon \colon 
\UA(L) \rightarrow \DA$ 
is neither injective nor
surjective.
However, let us now focus on
the case $L= \DerA$, $ \omega
= 
\mathrm{id} _{\DerA}$, and abbreviate 
$$
	\UA := \UA(\DerA).
$$
In this case, we have:

\begin{theorem}\label{luis}
If $ \mathbb{Q} \subseteq A$ 
and the $A$-module 
$\DerA$ is finitely
generated projective, then  
$ \hat \varepsilon \colon 
\UA \rightarrow \DA$ is an
isomorphism. In particular, 
we have $\DA=\DDA$. 
\end{theorem}
\begin{proof}
The morphism $\UA \rightarrow
\DA$ is a morphism of almost
commutative 
$A$-rings, hence induces an
associated morphism 
\begin{equation}\label{assgr}
	\gr(\UA) \rightarrow
\gr(\DA)
\end{equation}
of graded $A$-algebras. 
By the
Poincaré-Birkhoff-Witt
theorem
(Theorem~\ref{pbwtheorem}), 
we also have an isomorphism of
graded $A$-algebras
$$
	\SymA \DerA \rightarrow
\gr(\UA). 
$$
The composition of the two is 
the canonical morphism 
\begin{equation}\label{macarro}
	\tau \colon \SymA \DerA 
	\rightarrow \gr(\DA)
\end{equation}
that results from the fact
that $\gr(\DA)^1=\DerA$ and
that $\gr(\DA)$ is a
commutative $A$-algebra ($\SymA
\DerA$ is the free commutative
$A$-algebra on the $A$-module
$\DerA$).
It was shown by Narváez-Macarro
in
\cite[Corollary~2.17]{narvaezmacarroHasseSchmidtDerivations2009}
that
under the assumptions of our
theorem, $\tau$ is
an isomorphism.
Therefore, we deduce 
that (\ref{assgr}) is an
isomorphism. An elementary
induction argument implies that the
original filtered morphism 
$\UA \rightarrow \DA$ is thus
an isomorphism, too. 
\end{proof}

\begin{remark}
  If \(\mathbb{Q} \not \subseteq A\),
  then we need the notion of a
  \emph{Hasse-Schmidt derivation}
  in order to describe  \(\DA\).
  In this case, the condition that
   \(\DerA\) is finitely generated projective
   needs to be complemented with 
   an additional assumption,
   see \cref{rem:HSsmooth}
below.
   We refer to \cite[Theorem 16.11.2]{grothendieckElementsGeometrieAlgebrique1967} 
and \cite{travesDifferentialOperatorsNakais1998}
 for such a description of \(\DA\). 
 In particular,
 Narváez-Macarro has given a construction of  \(\DA\)
 as a universal enveloping algebra of
 Hasse-Schmidt derivations in \cite{narvaez-macarroRingsDifferentialOperators2020}.
\end{remark}

\subsection{Augmentations of
$\D^{op}_A$}\label{dopaugment}
Any $A$-ring
morphism $S \colon \DA \rightarrow 
\D^\op_A$ yields a right
$\DA$-module structure on $A$
that extends multiplication in
$A$, that is, an augmentation
$ \bar \varepsilon = \varepsilon
\circ S$ 
of $\D^\op_A$ as discussed at
the end of
Section~\ref{augmentedsec}. 
As observed by Quinlan-Gallego
in \cite{quinlan-gallegoSymmetryRingsDifferential2021a}, 
any right
action of $\DA$ on $A$ is
necessarily given by
differential operators, so in fact 
any 
augmentation of $\D^\op_A$
arises as $ \varepsilon \circ S$.
Furthermore, $S$
is always 
a morphism of filtered
\(A\)-rings,
and the induced map 
$\gr(S)$ of the associated graded
$A$-algebras satisfies
\begin{equation}\label{sfilt}
	\gr(S)([D]) =  (-1)^{n} [D],
\quad
[D] \in \gr(\DA)^n,
\end{equation}
which implies in particular that 
$S$ is an isomorphism. 
To summarise, we have:

\begin{lemma}
There is a bijective
correspondence between 
right $\DA$-module structures
on $A$ extending the
multiplication in $A$ and 
$A$-ring isomorphisms 
$S \colon \DA \rightarrow \D^\op_A$.
\end{lemma}
\begin{proof}
	See \cite[Lemma 4.3]{
	quinlan-gallegoSymmetryRingsDifferential2021a}.
\end{proof}
\begin{remark}
Equation (\ref{sfilt}) shows 
in particular that 
\(S^{2}(D) = D\) 
if \(D \in \DA^{1}\),
so $S$ is involutive
if  \(\DA = \DDA\).
\end{remark}

In particular,  
the following result will apply 
to the examples that 
we are going to study in the
last section of this paper:

\begin{theorem}\label{eamonthm}
Assume that $k$ is a field and 
that $A=\bigoplus_{n \ge 0} A^n$ is
a positively graded and finitely
generated algebra. If $A$ is 
Gorenstein (i.e. the $A$-module 
$A$ has finite injective dimension),
then there exists an isomorphism
of graded \(A\)-rings $\DA \cong \D^\op_A$. 
\end{theorem}
\begin{proof}
See \cite[Theorem 4.4]{
quinlan-gallegoSymmetryRingsDifferential2021a}.
\end{proof}

\begin{remark}\label{yekutielirem}
More generally, 
one has
$\D^\op_A \cong \omega _A
\otimes _A \DA \otimes _A 
\omega _A^{-1}$ for any
Gorenstein algebra, where $
\omega _A$ is the canonical
module (see 
\cite[Corollary 4.9]{ 
yekutieliResiduesDifferentialOperators1998}),
and under the assumptions of
the above theorem, one has 
$ \omega _A \cong A$. 
\end{remark}

\begin{remark} 
There are examples of algebras
$A$ for which $\DA$ is left,
but not right Noetherian 
\cite[Section 7]{smithDifferentialOperatorsAffine1988},
\cite[Section 5]{muhaskyDifferentialOperatorRing1988},
\cite{trippDifferentialOperatorsStanleyReisner1997},
\cite{saitoNoetherianPropertiesRings2009},
so in general, $\DA \not\cong
\DA^\op$. 
\end{remark}

\section{Smooth algebras}\label{smoothalgsec}
This section recalls the notions of
smooth, étale, and regular algebras
and a few variations thereof that
will be used freely in the main
text. 
We also show that $\DerA$ is
finitely generated projective
if $A$ is the algebra 
$C^\infty(X)$ of smooth
functions on a smooth manifold or
the algebra $k[X]$ 
of regular functions on a
non-singular affine variety over a
perfect field, so that
Theorem~\ref{luis} applies to
these examples.

As before, $A,B,K$ 
denote commutative 
algebras
over $k$, that is, morphisms of
commutative rings $k
\rightarrow A,k \rightarrow B,k
\rightarrow K$.

\subsection{Smooth and étale
algebras}\label{sesec}
We begin by defining 
formally smooth and 
smooth algebras. 
The definition 
is naturally combined
with that of an étale algebra 
that will also play a central
role in the proof of our main 
result. We follow the original
definitions of Grothendieck
\cite[Definition~0.19.3.1 and
Definition~0.19.10.2]{grothendieckElementsGeometrieAlgebrique1964}:

\begin{definition}\label{etaledef}
\begin{enumerate}
\item An \emph{infinitesimal
deformation} 
of $K$ 
is a surjective 
algebra morphism 
$ \psi \colon \idf
\epic K \cong \idf/N$ 
whose 
kernel $N \subset \idf$ 
is a nilpotent ideal, that is,
$N^n=0$ for
some $n \ge 0$. 
\item An algebra $A$ is 
\emph{formally smooth} 
if for all algebra morphisms
$ \varphi \colon A
\rightarrow K$ and all 
infinitesimal deformations 
$ \psi \colon \idf
\epic K$ of $K$ 
there is at least one
algebra morphism $ \rho \colon
A \rightarrow \idf$ with $ \psi
\circ \rho = \varphi $. 
\item The algebra $A$ is 
\emph{formally unramified}
if for all such 
$ \varphi,\psi$ 
there exists at most one
algebra morphism $ \rho \colon
A \rightarrow \idf$ with $ \psi
\circ \rho = \varphi $, and
\item it is 
\emph{formally étale} if it
is both formally smooth and 
unramified.
\end{enumerate}
$$
	\begin{tikzcd}
	A \ar[r, " \varphi "]
	\ar[dr, dashrightarrow, "
\rho "] &
	K \\
	k \ar[r] 
	\ar[u] &
	\idf \ar[u, twoheadrightarrow, " \psi" swap].
	\end{tikzcd}
$$
\end{definition}

\begin{remark}
An easy induction argument shows
that in (2) and (3) one may
equivalently consider only
infinitesimal deformations 
$ \psi \colon \idf \epic K$ 
with $ (\mathrm{ker}\, \psi) ^2=0$.
\end{remark}

\begin{definition}
The algebra $A$ 
is \emph{smooth},
\emph{unramified}, respectively 
\emph{étale} if it is
formally so and is in addition 
essentially of finite
type.
\end{definition}

\begin{remark}\label{iadicrec}
Grothendieck considers more
generally
topological rings and continuous ring
morphisms.
Matsumura uses in 
\cite{matsumuraCommutativeAlgebra1980}
the term ``smooth'' for what we call
``formally smooth'', that is, the
case of discrete rings. 
However, in
\cite{matsumuraCommutativeRingTheory1989},
he instead uses the term 
``$J$-smooth'' when $k$ is
discrete and $A$ is equipped
with the $J$-adic topology for
some ideal $J \subset A$. 
So therein, our ``formally
smooth'' is called 
``0-smooth''.
\end{remark}
\begin{remark}
Some include 
in the definition of
smoothness the assumption 
that $k$ is Noetherian, 
e.g.~\cite[Appendix
E]{lodayCyclicHomology1998}.
By Hilbert's
Basissatz, an
algebra $A$ that is essentially of
finite type is then Noetherian,
too. 
\end{remark}

For example, in the Noetherian
case, we have:

\begin{proposition}
Assume that $k$ is Noetherian
and $A$ is essentially of
finite type. Then
we have:
\begin{enumerate}
\item The following are equivalent:
\begin{enumerate}
\item $A$ is smooth.
\item For all prime ideals
$ \mathfrak{p} \subset A$, the
localisation 
$ A_{\mathfrak{p}}$ is
smooth.
\item For all maximal ideals
$ \mathfrak{m} \subset A$, the
localisation 
$ A_{\mathfrak{m}}$ is 
smooth.
\end{enumerate}
\item If $A$ is smooth, then
it is flat over $k$. 
\item If $A$ is unramified, the
converse of (2) holds. 
\end{enumerate}
\end{proposition}
\begin{proof} 
(1) and (2): See  
\cite[Théor\`eme X.7.10.4 on p103]{bourbakiAlgebreCommutativeChapitre2007}. 

(2): See \cite[Corollary 2.3 on
p65]{iversenGenericLocalStructure1973} or
\cite[Exercise~III.10.3]{hartshorneAlgebraicGeometry1977}.
\end{proof}

\begin{remark}
In particular, 
the assumption $ \mathrm{Tor}
^k_n(A,A) = 0$ for $n>0$ added in 
\cite[Proposition E.2 and 
Definition]{lodayCyclicHomology1998} follows 
from the definition of smoothness we
use; therein, it is made to be able
to make what we use as 
definition equivalent to the 
other characterisations given.
\end{remark}

\subsection{Differentially smooth
algebras}
Here we recall 
a weaker notion of
smoothness, and that for a
smooth algebra, $\DerA$ is finitely
generated
projective.

\begin{definition}\label{kaehlerdef}
\begin{enumerate}
\item Let $\mu_A \colon 
	A \otimesk A \rightarrow A$
be the multiplication in $A$. 
Then the $A$-module of 
\emph{Kähler differentials}
over $A$ is 
$$ 
\OmA
:= I_A/I_A^2,\quad
	I_A = \mathrm{ker}\, \mu _A.
$$
\item An algebra $A$ is
\emph{differentially smooth} 
if $\OmA$ is projective
and the canonical
surjections of $A$-modules
\begin{equation}\label{diffsmoothmapn}
	\SymnA \OmA 
	\epic
	I_A^n/I_A^{n+1}
\end{equation}
induced by multiplication in 
$ A \otimesk A$ are bijective.
\end{enumerate}
\end{definition}

So for a differentially smooth
algebra, we have an isomorphism of 
graded $A$-algebras
\begin{equation}\label{diffsmoothmap}
	\SymA \OmA 
	\rightarrow
	\bigoplus_{n \ge 0} 
	I_A^n/I_A^{n+1},
\end{equation}
but note that 
$\OmA$ is not necessarily
finitely generated.

\begin{remark}
In Grothendieck's
definition
\cite[Definition~IV.16.10.1]{grothendieckElementsGeometrieAlgebrique1967}
of differential smoothness, 
$\mathscr{P}$ is in the affine
case $A \otimesk
A $ equipped with the $I_A$-adic
filtration, see 
\cite[Definition~IV.16.3.1 and
(IV.16.3.7)]{grothendieckElementsGeometrieAlgebrique1967}).
\end{remark}

\begin{proposition}\label{dprops}
\begin{enumerate}
\item The map
$$
	\dd \colon A \rightarrow
\OmA,\quad
	a \mapsto [1 \otimesk a -
	a \otimesk 1]
$$
is a $k$-linear derivation.
\item If $A$ is essentially of
finite type, then 
$I^n_A/I^{n+1}_A$ and 
$\P_A^n$ is for all 
$n \ge 0$ a finitely
generated
$A$-module.
\item The derivation $\dd$ is 
universal in the sense
that for any $A$-module $M$,
we have $
	\mathrm{Hom}
_A(\OmA , M)
\cong \Derk(A,M)$ via 
	$\varphi \mapsto \varphi
\circ \dd$.
In particular, 
$\DerA \cong (\OmA)^* =
\mathrm{Hom} _A(\OmA,A)$. 
\end{enumerate}
\end{proposition}
\begin{proof}
(1) and (3) are 
\cite[Proposition
III.10.11.18]{bourbakiAlgebraChapters131989}.
(2) follows from 
the subsequent discussion therein:
if $B$ is an algebra 
generated by $b_1,\ldots,b_l
\in B$, then it follows from 
\cite[Lemma~III.10.10.1]{bourbakiAlgebraChapters131989},
that $\Omega^1_B$ is generated
as a $B$-module by $\dd
b_1,\ldots, \dd b_l$. If $S
\subseteq B$ is
multiplicatively closed and 
$A = S^{-1} B$, then it is
also straightforward that 
the images of the $ \dd b_i$
in 
$\OmA \cong S^{-1} \Omega^1_B$ 
generate $\OmA$ as an
$A$-module. 
It follows that the 
monomials $\dd b_{i_1} 
\cdots \dd b_{i_n}$, 
$1 \le i_1 \le \ldots \le 
i_n \le l$, generate 
$\SymnA \OmA$ as an
$A$-module.  
Applying the surjection 
(\ref{diffsmoothmapn}) yields 
generators of 
$I^n_A/I^{n+1}_A$.
To derive from this that all
$\P_A^n$ are finitely
generated, note that we have
short exact sequences of
$A$-modules 
$$
	0 \rightarrow 
	I_A^{n+1} /I_A^{n+2}
	\rightarrow 
	\P_A^{n+1} 
	\rightarrow 
	\P_A^n \rightarrow 0,
$$
and that $\P_A^0 = (A \otimesk
A)/I_A \cong A$. As all 
$I_A^n/I_A^{n+1}$ are finitely
generated, this
implies by induction that all 
$\P_A^n$ are finitely
generated. 
\end{proof}

\begin{corollary}\label{gradeisgrade}
If \(A\) is essentially of finite
type and $ \mathbb{Z} $-graded, 
then $\DA=\bigoplus_d \D_{A,d}$
is a $ \mathbb{Z} $-graded $A$-ring. 
\end{corollary}
\begin{proof}
If \(A\) is graded, then  \(A
\otimes A\)
and \(I_{A}\) are graded, hence
 \(\P^{n}_{A} = (A \otimes
A)/I^{n+1}_{A}\) 
 is also graded.
The grading on its dual \(\DA^{n}\)
follows from
\cite[Lemma 3.3.2]{nastasescuGradedFilteredRings1979}:
if \(M,N\) are graded \(A\)-modules
and  \(M\) is finitely generated,
then  \(\Hom_{A}(M,N) = 
\bigoplus_{d \in \Z}
\Hom_{A}(M,N)_{d}\), where
\[
	\Hom_{A}(M,N)_{d}
	\coloneqq 
	\{
		f \in \Hom_{A}(M,N) \mid
		f(M_{e}) \subseteq
		N_{d+e}
		\forall e \in \Z
	\}.
	\qedhere
\]
\end{proof}

\begin{corollary}
An algebra $A$ is formally
unramified if and only if 
$\OmA=0$.
In particular, this holds when 
$A$ is a localisation of $k$.
\end{corollary}
\begin{proof}
The main claim follows 
from the fact that  
if $ \rho \colon A \rightarrow
\idf$
lifts $ \varphi \colon 
A \rightarrow B$, then the other
lifts are precisely the maps of the
form $ \rho + \delta $ with 
$ \delta \in 
\Derk(A,\mathrm{ker}\,
\psi )$
(see e.g.
\cite[Proposition~X.7.1.1 on p83]{bourbakiAlgebreCommutativeChapitre2007}).
If $\eta \colon k \rightarrow A$ 
is a localisation, 
it is in particular an
epimorphism of rings, which is 
equivalent to $ \mu _A$ being an 
isomorphism, so
$I_A=0=I_A^2$.
\end{proof}

\begin{proposition}
A formally smooth algebra
is differentially
smooth. 
\end{proposition}
\begin{proof}
See
\cite[Proposition~IV.16.10.2]{grothendieckElementsGeometrieAlgebrique1967}.
\end{proof}

\begin{corollary}\label{pricpa}
If $ A$ is smooth, then 
we have
$$
	\P_A^n \cong 
	\bigoplus_{i=0}^n \SymA^i
	\OmA \cong 
	\bigoplus_{i=0}^n 
	I_A^i / I^{i+1}_A,
$$ 
and these $A$-modules are all 
finitely generated projective.
In particular, $ \DerA$ is
finitely generated projective.
\end{corollary}
\begin{proof} 
By the previous proposition, 
$\OmA$ is projective, and by 
Proposition~\ref{dprops} (2) it is
finitely generated. 
Hence its dual 
$ \DerA$ is finitely generated
projective, too. 
The statement about the 
$\P_A^n$ is shown by an easy
induction argument, using that the
symmetric powers 
$\SymnA M$ of a finitely generated
projective module $M$ are finitely
generated projective, and that
by induction
all the short
exact sequences 
$$
	0 \rightarrow I_A^{n+1} 
	\rightarrow I_A^n \rightarrow 
	I_A^n/I_A^{n+1} \cong 
	\SymnA \OmA \rightarrow 0
$$
split. 
\end{proof}

\begin{remark}
\label{rem:HSsmooth}
If $A$ is
differentially smooth, 
then the canonical map 
(\ref{macarro}) is the
inverse of the map dual
to (\ref{diffsmoothmap}).
In
\cite[Definition~2.3.11]{
narvaez-macarroRingsDifferentialOperators2020}
Narváez-Macarro calls
an algebra $A$
\emph{HS-smooth} if  
$\DerA$ is finitely generated
projective and the canonical 
map $ \gr(\DA) \rightarrow 
\bigoplus_{n \ge 0} (\SymnA \OmA)^*$
is bijective.  
So if $A$ is essentially of
finite type and differentially
smooth, then it is HS-smooth.
\end{remark}

\begin{remark}
Let $A$ be an integral
domain that is a finitely
generated algebra  
over a field $k$ of
characteristic $0$. 
The content of the 
\emph{Zariski--Lipman
conjecture} is that 
$A$ is smooth if
and only if $\DerA$ is
projective. The \emph{Nakai
conjecture} in turn asserts
that $A$ is smooth if
and only if $\DA$ is generated
by $\DA^1$, that is, if 
$\DA=\DDA$. Therefore, 
Theorem~\ref{luis}  
implies in particular the
(well-known) fact that a proof
of the Nakai conjecture would
also yield a proof of the
Zariski--Lipman conjecture.  
\end{remark}

\begin{remark}
In
\cite{ishibashiAnalogueNakaisConjecture1985},
Ishibashi has formulated
an analogue of the Nakai
conjecture 
for the case when \(k\) is a field of characteristic  \(p\).
It asserts that \(A\) is smooth
if and only if  \(\DA\) is generated
by the Hasse-Schmidt derivations on  \(A\).
See e.g. \cite{travesDifferentialOperatorsNakais1998}
for more details.
\end{remark}

\begin{example}\label{weirdex}
Among the examples studied by 
Sweedler were algebras $A$ that  
are \emph{purely inseparable} 
\cite[Definition~13.14]{sweedlerGroupsSimpleAlgebras1974},
meaning that  
\(I_{A}\) consists of nilpotent
elements. When $A$ is essentially of
finite type, this is equivalent to 
$I_A$ being nilpotent, 
say $I_{A}^{n+1}=0$.  
We then have 
$\PA^m \cong A \otimes A$ for 
$m \ge n$, and therefore 
$\DA = \DA^n \cong (A \otimes A)^* 
\cong \EndA$.  
\end{example}
\begin{example}\label{weirdexconc}
As a concrete example,
consider $k=\mathbb{Z} _2$ and 
$A = k[t]/ t^2  k[t]
$, the ring of truncated
polynomials, and let  
$\hbar \in A$ denote the residue
class of $t \in k[t]$. Then
$I_A$ is a free $A$-module with
basis $ 1 \otimesk \hbar - 
\hbar \otimesk 1$, and 
$I_A^2= 0$ (as we are in
characteristic 2), so we have
$\OmA = I_A \cong A$ and 
$\DerA \cong A$ (with a basis  
of the latter given by the 
$k$-linear derivation $ 
\frac{\dd}{\dd t}  \colon 
k[t] \rightarrow k[t]$, 
$t^n \mapsto n t^{n-1}$,
that descends to a derivation of
$A$). In particular, both 
$\OmA$ and $\DerA$ are finitely
generated projective. However, 
$A$ is not differentially smooth 
as $ \SymnA \OmA \cong A$ for all
$n$ while 
$I_A^n/I_A^{n+1} = 0$ for $n>1$.  
For this algebra, 
the canonical map
$\UA \rightarrow \DA = \EndA$ is
surjective, but not injective
(its kernel is the ideal
generated by $ (\frac{\dd }{\dd
t})^2$).
\end{example} 

\begin{remark}\label{symbolcalculus} 
By an \emph{affine
scheme} over $k$, 
we mean a representable
functor $X$ on the category of
commutative algebras, and we call
$X$ smooth if the  
algebras 
$A$ representing $X$ are smooth. 
In this case, the affine scheme 
$TX$ over $A$ represented by the 
$A$-algebra
$\Sym_A \OmA$ is the total space of
the \emph{tangent bundle} of $X$.
Its unit map is a morphism of
$k$-algebras and hence a morphism 
of affine schemes $TX \rightarrow X$ 
over $k$; this is a locally trivial
vector bundle whose 
$A$-module of global
sections is
$\DerA$. The $A$-module 
$\P_A^n$ plays the role 
of functions on $TX$ 
that in the 
direction of 
the fibres (the tangent spaces) 
are polynomials of degree at most 
$n$.
These are the 
\emph{principal parts}
(or \emph{principal symbols}) 
of differential
operators of order at most
$n$, hence the terminology.
Alternatively, the affine scheme 
represented by $\PA^n$ is the \emph{$n$-th
infinitesimal neighbourhood}
of $X$, embedded as the
diagonal into $X \times X$. 
If $I_A^{n+1}=0$ as in the
previous example, then this
$n$-th infinitesimal
neighbourhood is all of $X
\times X$, so geometrically,
$X \times X$ is then 
an infinitesimal deformation 
of $X$.
\end{remark}

\subsection{Regular algebras}
Let us briefly discuss 
the relation to the more classical
concept of regularity, and that
smoothness means regularity when $A$
is the coordinate ring of an affine
variety over a perfect field. 

\begin{definition}
\begin{enumerate}
\item A commutative 
ring $A$ is \emph{regular}
if it is Noetherian and for all
prime ideals 
$ \mathfrak{p} \subset A$, the
maximal ideal $ \mathfrak{p}
A_{\mathfrak{p}}$ of the 
localisation $A_{\mathfrak{p}}$ 
can be generated by 
$\dim (A_{\mathfrak{p}})$ 
elements, where
$\dim(A_{\mathfrak{p}})$ 
denotes the
Krull dimension of $
A_{\mathfrak{p}}$. 
\item 
An affine variety 
(irreducible algebraic set)
$X \subseteq k^N$ is 
\emph{non-singular} if 
its ring $k[X]$ of regular functions
$X \rightarrow k$ is regular. 
\end{enumerate}
\end{definition}

Over perfect fields, this just means
smoothness:

\begin{proposition}
If $k$ is a field and $A$ is
essentially of finite type over $k$,
then $A$ is smooth if and only if 
$A \otimesk A$ is regular. If $k$
is perfect, this is true if and only
if $A$ is regular. 
\end{proposition}
\begin{proof}
It is shown in 
\cite[Proposition X.6.5.8 on
p78]{bourbakiAlgebreCommutativeChapitre2007}
that under the assumptions made, 
the following statements are
equivalent:
\begin{enumerate}
\item 
$A \otimesk A$ is regular.
\item 
$A$ is \emph{absolutely regular}, 
meaning that for any finite radical
extension $k'$ of $k$, 
$k' \otimesk A$ is regular. 
\item $I_A$ is a complete
intersection. 
\end{enumerate}
In
\cite[Théor\`eme X.7.10.4 on
p103]{bourbakiAlgebreCommutativeChapitre2007}
it is shown that $I_A$ is a complete
intersection if and only if $A$ is
smooth. Finally, when $k$ is
perfect, then 
$A$ is absolutely regular if and
only if it is regular 
\cite[Example X.6.4.1 on
p75]{bourbakiAlgebreCommutativeChapitre2007} 
(by definition of a perfect
field, $k'$ is a separable
extension of $k$, hence 
$k' \otimesk A$ is regular
if $A$ is regular, 
see \cite[Proposition X.6.3.5
b) on
p75]{bourbakiAlgebreCommutativeChapitre2007}).
\end{proof}

\begin{corollary}
An affine variety $X$ over a perfect
field $k$ is non-singular if and
only if $k[X]$ is smooth. 
\end{corollary}

In particular, $\OmA$ is
projective in this case. However, 
as we have seen in
Example~\ref{weirdexconc}, 
this is not
sufficient to imply
regularity in general.  
Let us finally
point out that in the case of
integral domains, it actually
is:

\begin{proposition}\label{regularityprop}
If $k$ is a perfect field and
$A$ is an 
integral domain and essentially of
finite type, then $\OmA$
is projective if and only if
$A$ is regular.
\end{proposition} 
\begin{proof}
The earliest reference for
this fact that we are aware of is 
\cite{hochschildDifferentialFormsRegular1962}. 
Therein, the fundamental 
insight that 
$\OmA \cong 
	\mathrm{Tor} ^{A \otimesk
A} _1(A,A) 
$ is the first Hochschild homology
of the algebra $A$ with coefficients
in itself sneaks in somewhat tacitly
on p391,
but this follows directly from the
definitions. Thus
\cite[Theorem 2.1]{hochschildDifferentialFormsRegular1962} 
implies that if $A$ is
regular, so is $A
\otimesk A$, and then 
\cite[Theorem~3.1]{hochschildDifferentialFormsRegular1962} 
(applied to the multiplication
map $ A \otimesk A
\rightarrow A$) implies 
that 
$\OmA$ is
projective. The converse is
\cite[Theorem~4.1]{hochschildDifferentialFormsRegular1962}. 
\end{proof}

\begin{example}
The prototypical example of
$A$ we
are interested in is the cusp,
$A = k[t^2,t^3] \subset k[t]$. 
If $x := t^2$ and $y:=t^3$, then 
$A$ is as an algebra isomorphic to 
$k[x,y] / (x^3-y^2)k[x,y]$, and 
$\OmA$ is as an $A$-module 
generated by 
$\dd x, \dd y$ satisfying 
$ 3 x^2 \dd x - 2y \dd y =0$. 
In particular, $\OmA$ is not
torsion-free, as $ \omega = 
2x \dd y - 
3 y \dd x \neq 0$ but $y \omega =0$,
and $ \DerA$ is not projective
\cite[Example~15.3.12]{
mcconnellNoncommutativeNoetherianRings2001}.
Further examples of 
torsion elements in \(I_{A}/I_{A}^{n}\)
 were given in
\cite[Example 4.11]{barajasModuleDifferentialsOrder2020}.
\end{example}

\subsection{Smooth manifolds}
For the sake of completeness, let us
also point out that
Theorem~\ref{luis} applies equally
well in the setting of differential
rather than algebraic geometry, that
is, to the case of smooth
manifolds.
So as in 
Remark~\ref{liealgebroidrem},
assume that $k= \mathbb{R} $
and $A=C^\infty(X)$ is the algebra 
of smooth $ \mathbb{R}
$-valued functions
on a smooth manifold $X$. 
See any textbook on
differential geometry 
for the fact that 
the $k$-linear derivations of $A$  
are then identified geometrically
with the smooth vector fields on
$X$. That is, $ \DerA$ is the 
$A$-module of smooth sections of
the tangent bundle 
$TX$ of $X$. This
is a locally trivial smooth
real vector bundle over $X$, so by
the Serre--Swan theorem, one has:
\begin{proposition}
If $k=\mathbb{R} $ and 
$ A = C^\infty(X)$ for a
smooth manifold $X$, then 
the $A$-module $\DerA$ 
is finitely
generated projective.
\end{proposition}

\begin{remark}
In this case, $\OmA$ is not 
finitely
generated projective and is rather
ill-behaved; it is the $A$-module  
$ (\OmA)^{**} = \DerA^*$ of 
smooth 1-forms that is the
module of smooth sections of the
smooth cotangent bundle $T^*X$, 
not $\OmA$ itself. 
\end{remark}

\section{Hopf
algebroids}\label{hopfalgebroidssec}
In this
Section~\ref{hopfalgebroidssec}  we gather all
necessary definitions from the 
theory of corings,
bialgebroids, and Hopf
algebroids over commutative
base algebras. For more information and
background, see for example
\cite{bohmHopfAlgebroids2009,bohmHopfAlgebrasTheir2018,
ghobadiHopfMonadsSurvey2023}. 
For bialgebras and Hopf
algebras, see 
e.g.~\cite{sweedlerHopfAlgebras1969,
montgomeryHopfAlgebrasTheir1993,
brzezinskiCoringsComodules2003,
radfordHopfAlgebras2011}.

We retain the assumption that
$A,B$ are commutative
algebras over a commutative
ring $k$. 

\subsection{\coring{B}{A}-corings}\label{comonoidsec}
A Hopf algebroid over $A$ 
is in particular an augmented
$A$-ring and at the same time a
\emph{coalgebra} over $A$, that is, 
an $A$-module $D$ together
with $A$-module morphisms
$\Delta \colon 
	D \rightarrow D \otimes _A
D$ and $	\varepsilon \colon D
	\rightarrow A$
that satisfy
\begin{align*}
	(\Delta \otimes _A 
	\mathrm{id} _D) \circ
\Delta &= (\mathrm{id} _D 
	\otimes _A \Delta ) 
	\circ \Delta,
	\\
	(\varepsilon \otimes _A 
	\mathrm{id} _D) \circ 
	\Delta &= 
	(\mathrm{id} _D \otimes _A 
	\varepsilon ) \circ \Delta
=	\mathrm{id} _D,
\end{align*}
where we suppress both the
isomorphisms $ A \otimes _A D
\cong D \otimes _A A \cong D$
and $(D \otimes _A D) \otimes
_A D \cong D \otimes _A (D
\otimes _A D)$. Recall that  
a coalgebra is
\emph{cocommutative} if 
$ \Delta = \tau \circ
\Delta$, where $ \tau (c 
\otimes_A d) := d \otimes _A
c$. 
As usual, we adopt Sweedler's
notation $c_{(1)} \otimes _A 
	c_{(2)} := \Delta (c)$, 
$c \in D$.

\begin{remark}\label{leftright3}
A coalgebra is a comonoid in
$\Amod$. In contrast to
this, an $A$-ring is a monoid 
in $\AmodA$ 
(recall Remarks~\ref{leftright1} and
\ref{leftright2}). For a Hopf algebroid,
the compatibility between the coalgebra
structure and the right
$A$-module structure will be the
special case of the next definition
in which $A=B$. 
\end{remark}

\begin{definition}
A \emph{\coring{B}{A}-coring} is an
$A$-$B$-bimodule $D$ with
symmetric action of $k$,
together with an $A$-coalgebra
structure 
such that for all $d \in D,b
\in B$, we have  
$(db)_{(1)} \otimes _A 
	(db)_{(2)} = 
	d_{(1)}b \otimes _A 
	d_{(2)} = 
	d_{(1)} \otimes _A 
	d_{(2)} b.
$
\end{definition}

There are
several motivations for this
concept;
for us, the essence is 
that \coring{B}{A}-corings $D$ are
precisely the bimodules for
which the functor $ D \otimes
_B -$ takes 
$B$-coalgebras to
$A$-coalgebras, 
see \cref{opmonpropo}
below.

\begin{remark}
The concept of 
\coring{B}{A}-coring admits an
extension to noncommutative
base algebras $A,B$ -- in this
case, $D$ is an $A \otimesk
A^\op$-$B \otimesk
B^\op$-bimodule, see
\cite[Theorem-Definition 3.10]{takeuchiMoritaTheoryFormal1987} 
or \cite{szlachanyiMonoidalMoritaEquivalence2005}
for the full
definition.  
\end{remark} 

Note also that 
\coring{B}{A}-corings are preduals of
$A \otimesk B$-algebras:

\begin{lemma}\label{dualalgebra}
The $A$-linear dual $D^*=\mathrm{Hom}
_A(D,A)$ of a \coring{B}{A}-coring
is canonically an algebra
over $A \otimesk B$ with unit given
by 
\begin{equation}\label{coringeta}
	\abeins \colon A \otimesk B 
	\rightarrow
	D^*,\quad 
	\abeins (a \otimesk b) (d)
	=
	\varepsilon(adb)
\end{equation} 
and product 
$D^* \otimes _{A \otimesk
B} D^* \rightarrow D^*$ 
given by the so-called convolution 
\begin{equation}\label{convolutionproduct}
	(r * s) (d) :=
	r(d_{(1)} ) s(d_{(2)} ),
	\quad r,s \in D^*,d \in D. 
\end{equation}
\end{lemma}
\begin{proof}
This is verified by
straightforward verification,
see for
example~\cite{brzezinskiCoringsComodules2003,
bohmHopfAlgebroids2009,takeuchiMoritaTheoryFormal1987}. 
\end{proof}

\subsection{Conilpotency}
The Hopf algebroids that we 
will construct are not just
cocommutative, but also
\emph{conilpotent} 
(the term \emph{cocomplete} is used instead
in \cite{moerdijkUniversalEnvelopingAlgebra2010}), 
a property
that is crucial in the
structure theory especially of
cocommutative coalgebras.  

\begin{definition}
A \emph{coaugmentation} of a
coalgebra $D$ over $A$ 
is a coalgebra 
morphism $ \eta \colon A
\rightarrow D$. 
\end{definition}

\begin{remark}
Equivalently, 
a coaugmentation is given by a
\emph{group-like element} 
$ g \in D$, that is, an
element with $ \varepsilon (g)
= 1$ and $ \Delta (g) = g
\otimes _A g$ (namely $ g := 
\eta (1)$). 
\end{remark}

A choice of
coaugmentation yields 
dually to the case of augmented
$A$-rings an $A$-module splitting 
\[
	D \rightarrow 
	A \oplus \mathrm{ker}\,
\varepsilon ,\quad	
	c \mapsto (\varepsilon (c), 
	c - \varepsilon (c) g )
\]
and the (in general not
counital, but coassociative)
\emph{reduced coproduct} 
$$
	\bar \Delta \colon 
	\mathrm{ker}\, \varepsilon 
	\rightarrow 
	\mathrm{ker}\, \varepsilon  
	\otimes _A
	\mathrm{ker}\, \varepsilon, 
	\quad 
	c \mapsto
	\Delta (c) - g \otimes _A c
- c \otimes _A g. 
$$
We denote by 
$
	\bar \Delta ^n \colon
	\mathrm{ker}\, \varepsilon 
	\rightarrow (\mathrm{ker}\,
\varepsilon )^{\otimes _A
n+1} 
$
the map defined inductively by 
$$
	\bar \Delta ^0 = \id_{\ker \varepsilon},
	\quad
	\bar \Delta ^1 = \bar
\Delta,\quad 
	\bar \Delta ^{n+1 } 
	= (\bar \Delta^n \otimes _A 
	\mathrm{id} _{\mathrm{ker}\,
\varepsilon }) \circ \bar \Delta. 
$$
By construction, 
the assignment $(D, \Delta
, \varepsilon ,\eta ) \mapsto 
(\mathrm{ker}\, \varepsilon
,\bar \Delta )$ is an
equivalence between the
categories of coaugmented
counital coalgebras and of 
noncounital coalgebras. 

\begin{definition}
\begin{enumerate}
\item The $A$-module 
$\Prim(D):= \mathrm{ker}\,
\bar \Delta $ is the set of 
\emph{primitive elements} in
$D$, and the  
sequence $\{
\mathrm{ker}\, \bar \Delta
^n\}_{n \ge 0}$ is the
\emph{primitive filtration} of
$  \mathrm{ker}\, \varepsilon$.
\item One calls $D$ \emph{conilpotent} if 
its primitive filtration is exhaustive, 
that is, if for
all $ c \in 
\mathrm{ker}\, \varepsilon $ there exists 
$ n \ge 0$ such that 
$ \bar \Delta ^n (c) =0$. 
\item One calls \(D\) \emph{graded projective}
if the associated graded
$A$-module \(\gr (D)\)
is projective over \(A\).
\end{enumerate}
\end{definition}

\begin{remark}
	\label{primfiltrrem}
When 
$A$ is a field, then 
the primitive
filtration of \(\ker \varepsilon\)
is easily seen to be
a coalgebra filtration,
that is,
\[
	\bar{\Delta}(
	\ker \bar{\Delta}^{n})
	\subseteq
	\vspan_{A}
	\{ d_i \otimes_{A} d_j
		\mid 
		d_{i} \in \ker \bar{\Delta}^{i}, \,
		d_{j} \in \ker \bar{\Delta}^{j}, \,
		i+j = n
	\}.
\]
However, this
is not true in general,
as we will illustrate with
the next example.
\end{remark}

\begin{example}
Consider $A$-modules $M,N$ 
satisfying 
\[
	M \otimes _A M = 
N \otimes _A N =0,\quad  
	M \otimes _A N \neq 0,
\]
and set 
$V := M \oplus N$. Then we have
\[
	V \otimes _A V \cong 
	M \otimes _A N \oplus 
	N \otimes_{A} M \neq 0,\quad  
	V \otimes _A V \otimes _A V = 0.
\] 
Thus any $A$-linear map 
$\bar \Delta \colon V \rightarrow 
V \otimes _A V$ is automatically
coassociative, and 
$ \mathrm{ker}\, \bar \Delta ^2=V$.
Here, 
the primitive filtration
is 
$$
	0 \subseteq 
	\mathrm{ker}\, \bar \Delta 
	\subseteq  
	V = V = V = \ldots,
$$ 
and it fails to be a coalgebra filtration
whenever there exists an element
\(v \in V\)
for which \(\bar{\Delta}(v)\)
is not in the image of
\(\ker\bar{\Delta} \otimes_{A} 
\ker \bar{\Delta}\)
 in \(V \otimes_{A} V\).

Explicit examples of this type can
be constructed by choosing two
multiplicatively closed subsets 
$S_1,S_2 \subseteq A$ and by setting 
\[
	M\coloneqq (S^{-1}_1 A)/A, 
	\quad
	N\coloneqq 
	S^{-1}_2 A/A.
\]
Concretely, when 
$A = k[x,y]$ and $S_1= \{x^i \mid 
i \ge 0\}, 
S_2 =\{y^j \mid j \ge 0\}$, then  
$M$ and $N$ have as 
$A$-modules a presentation in terms
of generators and relations as
follows: $M$ has generators 
$m_i \in M$, 
$i<0$, which are the classes of 
$x^i$ in $M=k[x,x^{-1},y]/k[x,y]$, and
$N$ has generators
$n_j \in N$, 
which are the classes of 
$y^j \in N=k[x,y,y^{-1}]/k[x,y]$. 
The former are a basis of 
the $k[y]$-module $M$, while the
latter are a basis of the 
$k[x]$-module $N$. We have
$$
	x m_i = 
	\begin{cases}
	m_{i+1} & i<-1 \\
	0 & i=-1 
	\end{cases}
	\quad\text{and} \quad
	y n_j = 
	\begin{cases}
	n_{j+1} & j<-1 \\
	0 & j=-1. 
	\end{cases}
$$
So $M \otimes _A N$
is free
as a \(k\)-module,
with basis 
$m_i \otimes _A n_j$, 
$i,j < 0$. 
We define
the \(A\)-linear map $\bar \Delta \colon 
V \rightarrow V \otimes _A V$
given by
$$
	V \supset M \rightarrow 
	M \otimes _A N \subset 
	V \otimes_A V,\quad 
	m_i \mapsto 
	m_i \otimes _A n_{-1}
$$
and 
$$
	V \supset N \rightarrow 
	N \otimes _A M \subset 
	V \otimes_A V,\quad 
	n_j \mapsto 
	n_j \otimes _A m_{-1}.
$$
Then
$ \mathrm{ker}\, \bar \Delta $
is the $A$-submodule 
of \(V\)
spanned
by the $ym_i \in M$
and $x n_j \in N$,
hence $ \bar \Delta (m_{-1}) = 
m_{-1} \otimes _A n_{-1}$ is not in 
$ \mathrm{ker}\, \bar \Delta \otimes
_A \mathrm{ker}\, \bar \Delta $. 
\end{example}

\subsection{Bialgebroids}
\label{sec:bialgebroids}
Let $\eta \colon A
\rightarrow H$ be an
$A$-ring
and consider $H$ as an $A$-module
via $ a \otimesk h \mapsto 
a \lact h = \eta (a) h$, $a \in A,h \in
H$. The complication in the
definition of a bialgebroid
(in comparison to that of a
bialgebra) stems from the fact
that multiplication in $H$
does not induce a well-defined 
componentwise multiplication
on $ H \otimes_A H$. 
We rather have:

\begin{lemma}
The subset
$H \times_A H \subseteq 
H \otimes _A H$ given by
$$	
	\Bigl\{ \sum_i g_i \otimes _A
	h_i \mid
	\sum_i g_i \eta (a) \otimes _A
	h_i =
	\sum_i g_i \otimes _A
	h_i \eta(a) \forall a \in
A\Bigr\}
$$
is an $A$-ring with
respect to the componentwise
product
$$
	(g \otimes _A h)(x \otimes
_A y) := gx \otimes _A
hy,\quad
	g,h,x,y \in H
$$
and the unit map
$
	\eta \colon A \rightarrow H \times_A
	H$,
$a \mapsto
	a \otimes _A 1 =
	1 \otimes _A a.
$
\end{lemma}

\begin{definition}
The $A$-ring $H \times_A H$ is 
called the \emph{Takeuchi
product} of $H$ with itself over
$A$.
\end{definition}

\begin{remark}
As defined here, the
concept was introduced by
Sweedler \cite{sweedlerGroupsSimpleAlgebras1974}, and
Takeuchi generalised it to the
case of noncommutative $A$
\cite{takeuchiGroupsAlgebrasOtimes1977}.
\end{remark} 

Now one can meaningfully
define:

\begin{definition}\label{bialgebroid}
A (left) \emph{bialgebroid}
over $A$ is an augmented
$A$-ring $H$ together with a
morphism of $A$-rings
$ \Delta \colon
H \rightarrow H \times_A H$
for which $(H, \Delta ,
\varepsilon )$ is an \coring{A}{A}-coring.
\end{definition}

By very definition of $H
\times_A H$, this means in
particular that
$H$ is not just a
coalgebra over $A$, but an 
\coring{A}{A}-coring. 

\begin{remark}When $A=k$, then $H \times_A H = H \otimes _A H$
and the definition above reduces to that of a bialgebra over \(A\).
\end{remark}

\begin{remark}
\label{noncommutativerem}
There is a more general
definition of a left bialgebroid 
over a possibly noncommutative
algebra $A$ \cite{bohmHopfAlgebroids2009,
bohmHopfAlgebrasTheir2018,
takeuchiGroupsAlgebrasOtimes1977,
brzezinskiCoringsComodules2003}. This comes equipped
with both an $A$-ring and an
$A^\op$-ring structure; the
two unit maps are usually
called \emph{source} and 
\emph{target}.
The bialgebroids in
Definition~\ref{bialgebroid}
are exactly those whose
base algebras are commutative
and whose source and target
maps are equal. 
\end{remark}

\subsection{Hopf
algebroids}\label{hasection}
There are several
non-equivalent definitions of
the notion of a 
Hopf algebroid.
We will first recall the one
introduced by Schauenburg
\cite{schauenburgDualsDoublesQuantum2000},
and afterwards relate it to two
others due to Böhm
\cite{bohmHopfAlgebroids2009}
respectively 
Böhm and Szlachányi 
\cite[Definition~4.1]{bohmHopfAlgebroidsBijective2004}.
They all consider
bialgebroids over possibly 
noncommutative
base algebras as in 
\cref{noncommutativerem}. We recall
the definitions in the simpler
setting in which $A$ is commutative
(and the source and target maps
agree). As in
Remark~\ref{leftright1}, we denote 
by $H \in \Amod$ respectively
$\bar H \in \Amod$ 
the $k$-module $H$ with the
$A$-actions 
\begin{align*}
& A \otimesk H \rightarrow H,\quad 
	a \otimesk h \mapsto 
	a \lact h = \eta (a) h,\\
& A \otimesk \bar H 
	\rightarrow \bar H,\quad 
	a \otimesk h \mapsto 
	a \blact h = h\eta (a).
\end{align*}
 
The first notion, due to Schauenburg
\cite[Theorem and Definition~3.5]{schauenburgDualsDoublesQuantum2000},
focuses on \(H\) being
a Galois extension of  \(A\):

\begin{definition}
A bialgebroid \(H\)
is called a \emph{left Hopf algebroid}
(or $\times_A$-Hopf algebra)
if its \emph{(left) Galois map}
\begin{equation}\label{galoismap}
	\delta \colon
	\bar H \otimes _A
	H \rightarrow
	H \otimes _A H,\quad
	g \otimes _A h \mapsto
	g_{(1)} \otimes_A
	g_{(2)}h
\end{equation}
is bijective.
\end{definition}

\begin{remark}
There is the analogous notion
of a \emph{right Hopf algebroid},
which means that $H^\op$ is 
a left Hopf algebroid.
\end{remark}

The map $ \delta$ is
right $H$-linear (where the
right $H$-module structure is
given by multiplication from the
right in the second tensor
component), hence its inverse 
$ \delta^{-1}$ is so, too, if it exists. 
The inverse \(\delta^{-1}\) 
is thus uniquely determined
by its value on the elements
of the form 
\(h \otimes_{A} 1 \in H \otimes_{A}
H\). This information gets stored in
the following map:

\begin{definition}
	The \emph{translation map}
	of a left Hopf algebroid \(H\)
	is the map 
$$ 
	\tau 
\colon 
	H \to
\bar{H} \otimes _A H,
\quad
h \mapsto 
\delta^{-1}( h \otimes_{A} 1).
$$
As is customary in the literature, we 
denote $ \tau (h)$ in a
Sweedler-type notation by 
$h_+ \otimes _A h_-$. 
\end{definition}

\begin{proposition}\label{plusminusprops}
We have for any \(h \in H\), and  \(a \in A\):
\begin{align*} 
	a \lact (h_{+}) \otimes_{A} h_{-}
&=
h_{+} \otimes_{A} a \blact (h_{-}),
\\
	h_{+} \otimes _A 
	h_{-(1)} \otimes _A 
	h_{-(2)} 
&= 
	h_{++} \otimes _A 
	h_{-} \otimes _A 
	h_{+-},\\ 
 (hg)_{+} \otimes_{A} (hg)_{-}
&=
h_{+} g_{+} \otimes_{A} g_{-} h_{-}
.\end{align*}
\end{proposition}
\begin{proof}
See \cite[Proposition 3.7]{schauenburgDualsDoublesQuantum2000}.
\end{proof}

\subsection{Antipodes}
\label{sec:antipodes}
Assume we are given
a bialgebroid $H$ in the sense
of
Definition~\ref{bialgebroid} and
an isomorphism of $A$-rings
$$
	S \colon H \rightarrow
	H^\op,\quad S \circ \eta =
	\eta.
$$
This gives rise to
well-defined maps
\begin{align*}
	\gamma &\colon
	H \otimes _A H \rightarrow
	H \otimes _A H,
	\quad
	g \otimes _A h \mapsto
	S(g)_{(2)} \otimes _A
	S(g)_{(1)} h,\\
	\bar \gamma &\colon
	H \otimes _A H \rightarrow
	H \otimes _A H,\quad
	g \otimes _A h \mapsto
	S^{-1}(g) _{(1)} \otimes _A
	S^{-1}(g) _{(2)} h.
\end{align*}

\begin{definition}\label{hopfalgebroid}
A \emph{full Hopf algebroid} is a 
bialgebroid $H$ together with an isomorphism of
$A$-rings
$
	S \colon H \rightarrow
  	H^{\mathrm{op}}
$
for which
$ \gamma \circ \bar \gamma = \bar \gamma \circ \gamma = \mathrm{id} _{H \otimes _A H}$. 
The map $S$ is called an \emph{antipode} for $H$.
\end{definition}
\begin{proposition}\label{lefthacor}
A full Hopf
algebroid  
is a left Hopf algebroid 
with 
$$
	\delta^{-1} \colon
	g \otimes _A h \mapsto
	S^{-1}( S(g)_{(2)}) \otimes_A
	S(g)_{(1)}h.
$$
\end{proposition}
\begin{proof}
This is shown by
straightforward computation. 
\end{proof}

For bialgebras over $A=k$, 
the bijectivity of $ \delta $  
is equivalent to the existence
of an antipode 
given by $ S(h) \coloneqq 
\varepsilon(h_{+}) h_{-}.$
In general, this expression is
ill-defined, though, 
since the counit \(\varepsilon\) 
is not right \(A\)-linear, and there
are left Hopf algebroids that do not
admit an antipode
\cite{krahmerLieRinehartAlgebra2015}.

More precisely, the situation
is this: 
recall from
Section~\ref{augmentedsec} that 
right $H$-module structures on 
$A$ that extend the multiplication
in $A$ correspond bijectively to 
$A$-linear maps
\(\bar{\varepsilon} \colon \bar{H}
\to A\) 
satisfying $ \bar \varepsilon
\circ \eta = \mathrm{id} _A$, and 
for which 
$ \mathrm{ker}\, \bar \varepsilon $ 
is a right ideal in $H$, with the
corresponding right action of $h \in
H$ on $a \in A$ given by
$ \bar \varepsilon (a \lact h)$. 
We have seen in
Section~\ref{dopaugment} that for 
the $A$-ring $\DA$, such 
$\bar \varepsilon $ correspond
further to $A$-ring
isomorphisms $\DA \cong 
\D^\op_A$. 
At least when $H$ is
cocommutative, one similarly 
obtains that in the setting of
bialgebroids, such $\bar
\varepsilon $ correspond to antipodes:

\begin{proposition}\label{antipodprop}
Let \(H\) be a left Hopf algebroid.
\begin{enumerate}
\item 
An augmentation $\bar \varepsilon \colon
\bar H \rightarrow A$ of
$H^\op$ defines an $A$-ring
morphism 
$$
	S \colon H \rightarrow 
	H^\op,\quad 
	h \mapsto \bar \varepsilon
(h_+) \lact h_-.
$$
\item If $H$ is cocommutative,
$S$ is an antipode, and 
$S^2 = \mathrm{id} _H$.
\item For cocommutative Hopf
algebroids, this establishes a
bijective correspondence
between antipodes and
augmentations of $H^\op$. 
\end{enumerate}
\end{proposition}
\begin{proof}
(1): Clearly, $S$ is
well-defined, and 
Proposition~\ref{plusminusprops} 
shows that $S$ is an algebra 
morphism $ H \rightarrow H^\op$: 
\begin{align*}
	S(hg) 
&=
	\bar \varepsilon (h_+g_+) \lact
	(g_- h_-) \\
&= 
	\bar \varepsilon (
	\bar \varepsilon (h_+) \lact 
	g_+) \lact
	(g_- h_-) \\
&= 
	\bar \varepsilon ( 
	g_+) \lact
	(\bar \varepsilon (h_+) \blact	
	g_- ) h_- \\
&= 
	(\bar \varepsilon ( 
	g_+) \lact 
	g_-)  \eta ( 
	\bar \varepsilon (h_+)) h_- \\
&= 
	S(g)  S(h). 
\end{align*}
In particular, one obtains 
that for all $a \in A$,
$$
	S(a \lact g) = 
	a \blact S(g),\quad
	S(a \blact h) = 
	a \lact S(h), \quad 
	S( \eta (a)) = \eta (a).
$$

(2): In any left Hopf
algebroid, one has
$$
	\delta (h_+ h_{-+} \otimes
_A h_{--} ) = 
	h_{+(1)} h_- \otimes _A 
	h_{+(2)} .
$$
If $H$ is cocommutative, this
reduces to $1 \otimes _A h$,
see
e.g.~\cite[p69]{kowalzigHopfAlgebroidsTheir2009}.

Using that $ \delta $ and
hence $ \delta ^{-1} $ are
right $H$-linear so that 
$ \delta ^{-1}(1 \otimes _A h)
= 1 \otimes _A h$ shows 
\begin{equation}\label{ppmwunder}
	h_+ h_{-+} \otimes
_A h_{--}  = 1 \otimes _A h.  
\end{equation}
From this, it follows by
straightforward computation
that 
$S^2 = \mathrm{id} _H$. 

So for cocommutative left Hopf
algebroids, we have $ \gamma =
\bar \gamma $. Furthermore, 
Proposition~\ref{plusminusprops}
shows 
$$\Delta (S(h)) = 
	h_- \otimes _A 
	S(h_+).
$$

Using this and once more
(\ref{ppmwunder}) one computes
directly that 
$ \gamma $ is an involution so
that $S$ is an antipode. 

(3): One recovers $ \bar 
\varepsilon $ from $S$ as 
\[
	\varepsilon (S(h)) = 
	\bar \varepsilon (h_+) 
	\varepsilon (h_-) = 
	\bar \varepsilon (
	\varepsilon (h_-) 
	\blact h_+) = 
	\bar \varepsilon (h).
\qedhere
\]
\end{proof}

Note that 
Böhm has also given a more
general definition of full Hopf
algebroids
in which $S$ is not
necessarily bijective; the
two definitions are related as
follows:

\begin{proposition}
A Hopf algebroid in the sense of
Definition~\ref{hopfalgebroid} defines a
Hopf algebroid in the sense of
\cite[Definition~4.1]{bohmHopfAlgebroids2009}
as follows:
\begin{enumerate}
\item The left base algebra is
$L:=A$ and the left
bialgebroid $H_L$ is
the one underlying $H$:
$
	s_L=t_L:= \eta$,
$	\varepsilon _L:=
\varepsilon$,
$\Delta _L:= \Delta$.
\item The right base algebra
is
$R:=A$ and the right bialgebroid is
given by
$s_R=t_R:= \eta$,
$\varepsilon _R:=
\varepsilon \circ S$, and
$$
	\Delta _R(h) :=
S( S(h)_{(2)} )
	\otimes _A
	S( S(h)_{(1)} ) .
$$
\end{enumerate}
\end{proposition}
\begin{proof}
Straightforward computation,
cf.~\cite[Section~2.6.8]{kowalzigHopfAlgebroidsTheir2009}.
\end{proof}

\begin{remark}
A bialgebroid can in general
admit several antipodes,  see 
Example~\ref{haneqha} below and  
\cite{bohmAlternativeNotionHopf2005}
for a more detailed discussion.
Therefore, when \(A = k\), the 
definition of a full Hopf
algebroid does \emph{not} reduce to
the definition of a Hopf algebra
(here the antipode 
is unique, as it is the convolution
inverse of the identity map
\(\mathrm{id}_H\)).
See e.g.~\cite[Example
4.1.1]{bohmHopfAlgebroids2009}
for an example. 
\end{remark}

\subsection{Closed monoidal
categories}
\label{sec:closedmonoidal}
The following theorem
expresses the representation-theoretic 
meaning of bialgebroid and
Hopf algebroid structures on 
an augmented $A$-ring $H$:

\begin{theorem}
\label{monoidalthm}
Let $ \eta \colon A
\rightarrow H$, $ \varepsilon
\colon H \rightarrow A$ be an
augmented $A$-ring. 
\begin{enumerate}
\item 
The $A$-bialgebroid
structures on $H$ correspond
bijectively to mo\-no\-id\-al
structures on the category $\Hmod$ for
which the forgetful functor 
$\Hmod \rightarrow \Amod$ is
strict monoidal. 
\item 
	$H$ is a left Hopf
algebroid 
if and only if the left closed
structure on $(\Amod, \otimes
_A, A)$ lifts to $\Hmod$. 
\end{enumerate}
\end{theorem}
\begin{proof}
See \cite[Theorem
5.1]{schauenburgBialgebrasNoncommutativeRings1998}
for (1) and  
\cite[Theorem
3.1]{schauenburgDualsDoublesQuantum2000}
for (2).
\end{proof}

Explicitly, the unit 
object of $\Hmod$ is $A$ with
$H$-module structure given by
$ \varepsilon $, and 
the $H$-module structure on $
M \otimes _A N$ is given for
$H$-modules $M,N$ by
$$ 
	h (m \otimes _A n) = 
	h_{(1)} m \otimes _A 
	h_{(2)} n, \quad h \in H, m
\in M, n \in N.
$$ 

\begin{remark}
For any bialgebroid, the monoidal category 
$\Hmod$ is  
left closed with internal
hom functor given by
$\mathrm{Hom} _H (H \otimes _A
M , N)$, where 
$H \otimes _A M$ is the
monoidal product of the left
$H$-modules $H$ and $M$ and
the left $H$-module structure
on $ 
\mathrm{Hom} _H (H \otimes _A
M , N)$ is given by 
$$
	(h f) ( g \otimes _A m) =
	f( gh \otimes _A m),\quad 
	g,h \in H, m \in M, f \in
\mathrm{Hom} _H (H \otimes _A
M , N). 	 
$$ 
The point in
Theorem~\ref{monoidalthm} (2)
is that when the map $ \delta $
from (\ref{galoismap}) 
is bijective, it induces a
natural isomorphism of
$H$-modules
$ H \otimes _A M \cong \bar H
\otimes _A M$, where the
$H$-module structure on the
latter is given by 
$$
	h (g \otimes _A m) = 
	hg \otimes _A m, 
$$
and therefore one has 
$$
	\mathrm{Hom} _H( H \otimes
_A M , N) \cong 
	\mathrm{Hom} _H (\bar H
\otimes _A M ,N ) \cong 
	\mathrm{Hom} _A (M,N). 
$$
\end{remark}

\subsection{Opmonoidal functors}
\label{sec:opmonoidal}
An \emph{opmonoidal structure} (also called
\emph{comonoidal}) on a 
functor 
$F \colon 
\Cat \rightarrow 
\DCat$
between monoidal categories
is a monoidal structure
on its opposite 
\(F^{\op} \colon \Cat^{\op}
\to \DCat^{\op}\)
(see
e.g.~\cite[Definition 3.5]{bohmHopfAlgebrasTheir2018}).
Explicitly, this is 
given by natural morphisms  
$$
	\mu_{M,N} \colon 
	F(M \otimes_{\Cat} N)
	\rightarrow F(M) \otimes_{\DCat} 
	F(N)
$$
and a morphism 
$ \eta \colon F(1_{\Cat}) 
	\rightarrow 1_{\DCat}
$
(where $M,N$ are objects in 
$\Cat$ and 
$1_{\Cat},1_{\DCat}$
are the unit objects in the
two categories) which satisfy
the coassociativity and counitality
conditions.
In particular, this
means that $F$ induces a 
functor from comonoids in
$\Cat$ 
to comonoids in $\DCat$. 

\begin{remark}
If a functor \(F \colon \Cat \to \DCat\) 
admits a right adjoint 
\(G \colon \DCat \to \Cat \),
then there exists a bijective
correspondence between
opmonoidal structures on \(F\)
and monoidal structures on  \(G\).
The latter is the approach used
in the sources
\cite{takeuchiMoritaTheoryFormal1987,
szlachanyiMonoidalMoritaEquivalence2005}
cited below.
\end{remark}

When the categories are the
categories of modules over
commutative algebras, this leads directly to
corings: 

\begin{theorem}
\label{opmonpropo}
Let $D$ be an
$A$-$B$-bimodule. Then there
exists a bijective
correspondence
between
\coring{B}{A}-coring structures on
$D$ and opmonoidal structures
on $ D \otimes _B - \colon
\Bmod \rightarrow \Amod$ via 
$$
	\mu _{M,N} (d \otimes _B (m
\otimes _B n)) :=
	(d_{(1)} \otimes_B m)
	\otimes _A 
	(d_{(2)} \otimes_B n),\quad
	\eta (d \otimes _B b) :=
	\varepsilon (db), 
$$
where $d \in D,m \in M,n \in
N,b \in B=1_{\Bmod}$. 
\end{theorem}
\begin{proof}
See \cite[Theorem 3.6]{
	takeuchiMoritaTheoryFormal1987}.
\end{proof}

This gets upgraded to a
statement about the monoidal
categories of modules over
bialgebroids over different
base algebras as follows:

\begin{proposition}
Let $H$ be a
bialgebroid over $A$, $G$
be a bialgebroid over $B$, and
\(D\) be an  \(H\)-\(G\)-bimodule.
Then the opmonoidal structures on 
\({D \otimes_{G} -} \colon 
\Gmod \to \Hmod\)
correspond bijectively to 
\coring{B}{A}-coring
structures on $D$ 
such that for any 
\(d \in D, g \in G\)
we have
\begin{enumerate}
	\item \(\Delta_{D}(dg) = 
		d_{(1)} g_{(1)}
		\otimes_{A}
		d_{(2)} g_{(2)} \), and
	\item \(\varepsilon_{D}(dg)
		= \varepsilon_{D}(d \, \varepsilon_{G}(g))\).
\end{enumerate}
\end{proposition}
\begin{proof}
	See \cite[Lemmata 2.2 and 3.3]{szlachanyiMonoidalMoritaEquivalence2005}.
\end{proof}

\subsection{Rings of differential operators}
\label{sec:sweedler-BZ-N}
Recall that our goal is to find
a Hopf algebroid structure on the ring of differential
operators \(\DA\).
The augmented \(A\)-ring structure on  \(\DA\) is fixed,
so the main task lies in finding a suitable
comultiplication map  
\(\Delta \colon \DA \to \DA \otimes_{A} \DA\).
In particular, the axioms of a bialgebroid imply that
\begin{equation}
  \label{eq:leibniz}
  D(ab) = D_{(1)}(a)D_{(2)}(b)
\end{equation}
for all \(a,b \in A\), and  \(D \in \DA\).

Here are two known instances
where the ring of differential
operators is a bialgebroid.
The first comes from Sweedler's seminal paper
\cite{sweedlerGroupsSimpleAlgebras1974}
and the second one from
Ben-Zvi and Nevins' paper
\cite{ben-zviCuspsModules2004}.

Sweedler introduced the
following notion:
\begin{definition}
An algebra \(A\) has 
\emph{almost finite projective differentials}
if there exists a directed
system \((I_{\alpha})_{\alpha}\)
of ideals in \(A \otimes A\) 
which is cofinal with
\((I_{A}^{n})_{n}\) (for each
$ \alpha $ there exists $n$
such that $I_A^n \subseteq
I_\alpha$ and for each $n$ there
exists $
\alpha $ such that 
$ I_\alpha
\subseteq I_A^n$),
and for which 
\((A \otimes A)/I_{\alpha}\) 
is finitely generated projective as an \(A\)-module
for all \(\alpha\).
\end{definition}

\begin{example}
If \(A\) is smooth, then  \(A\)
has almost finite projective differentials,
as follows from \cref{pricpa}.
\end{example}

\begin{example}
If \(A\) is purely inseparable as in \cref{weirdex} 
and  \(A\) is finitely generated projective over \(k\), 
then \(A\) has almost finite projective differentials.
Indeed, if \(I_{A}^{n+1} = 0\), then
the singleton \({0}\) satisfies the desired
properties, since \((A \otimes
A)/I_A^{n+1} = A \otimes A\)
 is finitely generated projective
over \(A\) 
by base change.
\end{example}

We refer to \cite{sweedlerGroupsSimpleAlgebras1974}
for further constructions of
algebras with almost finite projective
differentials. For such
algebras, we have:

\begin{theorem}
If \(A\) has almost finite projective differentials, then
there exists a unique left Hopf algebroid structure on \(\DA\).
This is cocommutative and
the comultiplication induces 
an isomorphism \(\DA \cong \DA \times_{A} \DA\).
\end{theorem}
\begin{proof}
The existence of the bialgebroid structure is shown in 
\cite[Theorem 8.7]{sweedlerGroupsSimpleAlgebras1974}.
The existence of the translation map is shown in
\cite[Theorem 12.1, Remark 12.2]{sweedlerGroupsSimpleAlgebras1974}
through the existence of what Sweedler calls an \emph{Ess map}.
\end{proof}

For smooth algebras, this
implies:
\begin{corollary}
If \(A\) is smooth over  \(k\),
then  \(\DA\) is a cocommutative Hopf algebroid.
\end{corollary}

\begin{remark}
If \(A\) is smooth, then the
primitive filtration
of \(\DA\) agrees with its canonical filtration.
Therefore, \(\DA\) is conilpotent and graded projective.
\end{remark}

In the next sections, we shall consider
the notion of \emph{local projectivity}
for rings of differential
operators. As will be
explained in
Example~\ref{exa:almostfin},
this holds for algebras with almost finite
projective differentials.

The examples of monomial
curves that we consider in the
last section of this paper
could also be approached using
the results of Ben-Zvi and
Nevins, as they are \emph{good Cohen--Macaulay varieties} 
\cite[Definition~2.9]{ben-zviCuspsModules2004}.
We briefly outline how the existence
of a bialgebroid structure on the rings of differential
operators on such varieties can be
deduced from Theorem~3.11
therein. 

To do so, 
the notion of a bialgebroid
respectively Hopf algebroid
has to be formulated in a more general setting,
where one replaces the closed monoidal category \(\Amod\)
with the categories of 
quasi-coherent respectively
\emph{pro-coherent} 
sheaves
over a variety \(X\) 
(see
\cite[Remark~2.7]{ben-zviCuspsModules2004}
for the latter notion). 
Recall
that when $X$ is affine with
coordinate ring $A$, then a
quasi-coherent sheaf is
equivalent to an $A$-module; a
pro-coherent sheaf corresponds to an
inverse system of $A$-modules. 
The following
theorem implies that for an
affine good Cohen--Macaulay  
variety with coordinate ring
$A$, $\DA$ as defined in the
present paper carries a
bialgebroid structure:

\begin{conjecture}
\label{thm:benzvinevins}
If \(X\) is a good Cohen--Macaulay variety,
then the sheaf of differential operators  \(\D_{X}\)
is a bialgebroid 
in the category of 
quasi-coherent sheaves over \(X\).
\end{conjecture}
\begin{proof}[Sketch of proof]
In their paper, the authors introduce a
pro-coherent Hopf algebroid \(\mathscr{J}_{X}\),
called the \emph{jet
algebroid}. This is commutative as an
algebra but its source and
target maps are distinct; in
the affine case, it
corresponds to the inverse
system $(\P^n_A)_n$ from
(\ref{sweedlerprincipal}).   
The sheaf of differential operators \(\D_{X}\)
is its dual in the category of
pro-coherent sheaves, i.e.~a
differential operator is a
continuous $\mathscr{O}_X$-linear functional
on $\J_X$; this sheaf 
turns out to be quasi-coherent.

Theorem 3.11(2) in \cite{ben-zviCuspsModules2004}
states that the category of left \(\D_{X}\)-modules 
is equivalent to the category of
\(\mathscr{J}_{X}\)-\emph{comodules}
which are quasi-coherent.
The category of all 
\(\mathscr{J}_{X}\)-comodules
is monoidal since
\(\mathscr{J}_X\) is a Hopf
algebroid
(the proof of this fact carries over
verbatim from the purely
algebraic setting, see e.g.~\cite[Theorem 3.18]{bohmHopfAlgebroids2009}).
The tensor product of
comodules that are
quasi-coherent can be shown 
to be the usual tensor 
product of quasi-coherent
sheaves. That is, the
subcategory of quasi-coherent
comodules is monoidal, and the
forgetful functor to
\(\mathscr{O}_X\)-modules is
strict monoidal. 
Hence we have a monoidal structure as in 
\cref{sec:closedmonoidal}
on the category of left \(\D_{X}\)-modules,
thus the bialgebroid structure on \(\D_{X}\)
 follows.
\end{proof}
Our approach to the monomial curves 
relies on elementary methods, which 
allow us to compute explicitly the 
Hopf algebroid structure for such curves.

\subsection{The example
$\UA(L)$}
\label{UALsec}
Universal enveloping algebras
of Lie algebras are  
cocommutative and conilpotent
Hopf algebras. 
Similarly, we have: 

\begin{proposition}
The universal
enveloping algebra $\UA(L)$ of any
Lie--Rinehart algebra is a
cocommutative and conilpotent left Hopf algebroid
over $A$, whose coproduct
and counit are 
determined by 
$$
	\Delta (\rho (X)) = 
	1 \otimes _A \rho (X) 
	+
	\rho (X)  \otimes _A 1,
\quad
	\varepsilon (\rho (X)) = 0,\quad
	\forall X \in L.
$$
Its translation map
is given by
\[
	\tau(\rho(X)) =
	\rho(X)_{+} \otimes_{A} \rho(X)_{-}
	=
	\rho(X) \otimes_{A} 1
	- 1 \otimes_{A} \rho(X).
\]
\end{proposition}
\begin{proof}
This is well-known; 
the bialgebroid structure of 
$\UA(L)$ has 
appeared first in
\cite{xuQuantumGroupoids2001}, and 
the fact that it is a left Hopf
algebroid is pointed out e.g.~in
\cite[Example~2]{kowalzigDualityProductsAlgebraic2010}.
Recall the filtration
$\{\UA(L)^n\}$ given in
(\ref{pbwfiltration}); it
is immediate that 
$ \UA(L)^n \cap 
\mathrm{ker}\, \varepsilon
\subseteq \mathrm{ker}\,  \bar
\Delta ^n$,
so $ \UA(L) $ is
conilpotent (the coaugmentation
is given by $\eta$, that is,
the group-like element $1$).
\end{proof}

Conversely, the set 
$\Prim(H)$ of primitive elements of a
bialgebroid \(H\) has a natural structure
of a Lie--Rinehart algebra
over $A$, with Lie bracket
given by the commutator in $H$. 
The universal property of
$\U_A(\Prim(H))$ yields a
morphism of filtered $A$-rings
and of bialgebroids
$$
	\zeta \colon \U_A(L)
\rightarrow H,
$$
and Moerdijk and Mr\v{c}un have 
extended the classical
Cartier--Milnor--Moore theorem
to this
setting:

\begin{theorem}
Let \(k\) be a field of
characteristic 0,
\(A\) be a commutative
algebra over \(k\), and 
\(H\) be a bialgebroid over
\(A\). The following two
statements are equivalent:
\begin{enumerate}
\item $H$ is cocommutative,
conilpotent,
and graded projective. 
\item 
\(\Prim (H)\) is a projective
$A$-module and $\zeta$ is an
isomorphism.  
\end{enumerate}
\end{theorem}
\begin{proof}
See 
\cite[Theorem
3.1]{moerdijkUniversalEnvelopingAlgebra2010}.
\end{proof}
The examples we construct later
are left Hopf algebroids which are
cocommutative, conilpotent, but
not graded projective.

Universal enveloping algebras
of Lie--Rinehart algebras
do not always admits
an antipode $S$ 
\cite{krahmerLieRinehartAlgebra2015},
however they are
also good examples to discuss
that $S$, if it exists, is in
general not unique. Before we
move on to the main part of
the paper, we do this for two
explicit examples.

The first one is the
localisation of the 
well-known Weyl algebra and
will play a crucial role in
the main
Section~\ref{numsgsec} of our
paper as all the Hopf algebroids we
will consider there are sub Hopf
algebroids thereof.  

\begin{example}\label{weylalgebra}
Consider the algebra
$K:=k[t,t^{-1}]$ of Laurent
polynomials and the Lie--Rinehart
algebra $ (K,
\Derk(K))$ over $K$. Then
$H:=\UK \cong \DK$ 
is the (localised) Weyl algebra
$$
	H \cong k \langle t,
u, \partial
\rangle / \langle \partial t -
t \partial - 1 , tu -1 ,
ut -1 \rangle .
$$
We will describe this example
in more detail in
Section~\ref{lpsec} below. 
As follows from the discussion in
Section~\ref{dopaugment}, 
each \(K\)-ring morphism  \(H \to H^{\op}\) is
uniquely determined by its
value
$$
	S( \partial ) = - \partial + p,
$$
for some Laurent polynomial $p
\in K$.
Each such morphism is involutive, and
is in fact an antipode in the
sense of Definition~\ref{hopfalgebroid}.
\end{example}

The purpose of our second
example is to stress that 
for $A=k$, an antipode in the
Hopf algebroid sense is not
the same as an antipode in the
Hopf algebra sense:

\begin{example}\label{haneqha}
If $A=k$, then a Lie--Rinehart
algebra is just a Lie algebra
over $k$, and $\U_{k}(L)$ is
the universal enveloping
algebra of the Lie algebra in
the usual sense. 
It is a bialgebra (over $k$), and one verifies by direct computation
that the antipodes in the sense of
Definition~\ref{hopfalgebroid}
correspond to linear
functionals $ \lambda \in
\homk(L/[L,L],k)$ as follows:
$$
	S(X) = -X + \lambda ([X]), \quad X \in L.
$$
So if $ L$ is semisimple,
the antipode is unique and is
the antipode of the Hopf
algebra $\U_{k}(L)$, but in general
this is not true.
\end{example}

\begin{remark}
For a Lie--Rinehart
algebra \((A,L)\), a left
respectively right
$\UA(L)$-module is the same as
an $A$-module with 
a \emph{flat left respectively 
right connection} 
that assigns to each $X \in L$ the
$k$-linear map $\nabla_X$ by which
$X$ acts on the module. Hence
Proposition~\ref{antipodprop} 
reduces for this example to 
the statement that antipodes for 
$\UA(L)$ correspond bijectively to
flat right connections on $A$
(see e.g.~\cite[Proposition 4.2.11]{
kowalzigHopfAlgebroidsTheir2009}). 
\end{remark}

\section{Descent for Hopf algebroids}
\label{descentforcoringssec}
In this section, we develop the 
general method that we will 
apply in the next section to 
differential operators.
Our goal is to prove 
Theorem~\ref{descenttheorem}:
if $A,B \subseteq K$ are
subalgebras, then bialgebroid 
structures on an augmented 
$K$-ring $C$ 
descend to
\coring{B}{A}-coring structures 
on $C(B,A) = \{c \in C \mid 
\forall b \in B : 
\varepsilon (cb) \in A \}$, 
provided that 
$C \cong K \otimes _A C(B,A)$ and 
the $A$-module $C(B,A)$   
shares a module-theoretic 
property known as
local projectivity, which is recalled
in the first subsection. 
In our applications, an even
stronger condition will be
satisfied that is 
discussed in Section~\ref{slpsec}. 

\subsection{Local
projectivity}\label{locprocsec}
We fix a left \(A\)-module   
\(M\) and a right $A$-submodule
\(R \subseteq M^*=\mathrm{Hom}
_A(M,A)\) of the dual module. 
We will distinguish
between left and right
$A$-modules, and in fact, in this
present
Section~\ref{locprocsec} we
nowhere use the fact that $A$
is commutative. 

\begin{definition}
We say that \(M\) is 
\emph{\(R\)-locally projective}
if for all right \(A\)-modules \(X\),
the evaluation maps
$$
	\evl_{X,r} \colon 
	X \otimes _A M \rightarrow X,\quad 
	x \otimes _A m \mapsto 
	x r(m),\quad r \in R
$$
are jointly injective, that
is, if $\bigcap_{r \in R} 
\mathrm{ker}\, \evl_{X,r} =0$.
\end{definition}

In
\cite{vercruysseGaloisTheoryCorings2007}, 
such $M$ were called 
\emph{weakly $R$-locally
projective}. For 
\(R = M^{*}\) , several names occur in
the literature, including 
\emph{locally projective}
\cite{brzezinskiCoringsComodules2003,
zimmermann-huisgenPureSubmodulesDirect1976},
\emph{universally torsionless} 
\cite{garfinkelUniversallyTorsionlessTrace1976}, 
\emph{trace modules} \cite{ohmContentModulesAlgebras1972},
and \emph{flat and strict
Mittag-Leffler}
\cite{raynaudCriteresPlatitudeProjectivite1971}.

We will use two results that
in particular 
relate $R$-local projectivity
to flatness and projectivity,
and which are
contained in most of the above 
sources; we include full
proofs for the sake of
completeness.

\begin{lemma}\label{flatlemma}
Let \(M\) be \(R\)-locally projective. 
Then \(M\)
is flat, and 
for any injective map 
$ \iota \colon X \monic
Y$ of right \(A\)-modules,
we have 
\[
	\mathrm{im}
	(\iota \otimes _A
	\mathrm{id}_M) 
	= 
	\bigl\{
	y  
	\in Y \otimes_{A} M
\mid
	\evl_{Y,r}(y) \in
\mathrm{im}\, \iota 
	\, \forall r \in R
	\bigr\}.
\]
\end{lemma}
\begin{proof}
By the definition of the
evaluation maps, each $r \in
R$ gives rise to a commutative
diagram of abelian groups 
$$ 
\begin{tikzcd}
0 \arrow[r] & 
X \otimes _A M \arrow[r,"\iota \otimes _A
\mathrm{id} _M"] \ar[d,"\ev_{X,r}"] &
Y \otimes _A M \arrow[r,"\xi \otimes _A
\mathrm{id} _M"] \ar[d,"
\ev_{Y,r}"] &
(Y/X) \otimes _A M \arrow[r]
\ar[d," \ev_{Y/X,r}"] &
0 \\
0 \arrow[r] & 
X \arrow[r,"\iota"] &
Y  \arrow[r,"\xi"] &
(Y/X) \arrow[r] &
0, 
\end{tikzcd}
$$
where $ \xi \colon Y
\rightarrow Y/X$ is the
canonical projection. 

The flatness of $M$ follows
from the commutativity of the
left square: if $x \in X
\otimes _A M$ is non-zero,
then by
assumption, there exists 
some $r \in R$ with $ \evl_{X,r}(x) 
\neq
0$. As we have  
$ \evl_{Y,r} ( (\iota \otimes
_A \mathrm{id} _M)(x) ) = 
\iota (\evl_{X,r} (x))$, it follows that  
$(\iota \otimes _A \mathrm{id}
_M)(x) \neq 0$.

So the exactness of the bottom
row of the above diagram
implies the exactness of the
top row, that is,  
$ \mathrm{im} (\iota
\otimes _A \mathrm{id} _M)
=
\mathrm{ker} ( \xi \otimes _A
\mathrm{id} _M)$. 
However, if  
$y \in Y \otimes _A M$, then 
$ (\xi \otimes _A \mathrm{id}
_M) (y) = 0 $ if and only if
for all $r \in R$ 
we have 
$$ 
	0 = 
	\ev_{Y/X,r} (
	(\xi \otimes _A \mathrm{id}_M) (y)) 
	= 
	\xi ( \ev_{Y,r}(y)),
$$
that is, if and only if 
$ \ev _{Y,r}(y) \in 
\mathrm{im}\,\iota 
$.
\end{proof}

For the second result, we 
define the following generalisation of
finite-rank operators:

\begin{definition}
For any $A$-module $N$, we
denote by
$$
	\hhom^R_A(M,N) \subseteq
\mathrm{Hom} _A(M,N) 
$$
the set of all $A$-module
morphisms $M \rightarrow N$  
of the form 
$$
	\varphi \colon 
	M \rightarrow N,\quad
	m \mapsto 
	\sum_{i=1}^l r_i(m) n_i 
$$
for some $ r_1,\ldots,r_l \in R,
n_1,\ldots,n_l \in N$, $l \ge 0$. 
\end{definition}

\begin{remark}
Clearly, $\hhom^R_A(M,N)$ is closed under
addition, and for all $
\varphi \in \hhom^R_A(M,N)$
and for all $A$-module
morphisms
$ \psi \colon N \rightarrow P$,  
we have $ \psi \circ \varphi
\in \hhom^R_A(M,P)$ with
$$	
	(\psi \circ \varphi) (m) = 
	\sum_i 
	r_i (m) p_i,\quad  
	p_i := \psi (n_i). 
$$  
In particular, 
$\hhom^R_A(M,M) \subseteq 
\mathrm{Hom} _A(M,M)$ is a
left ideal. If $R$ is closed 
under precomposition
with arbitrary endomorphisms
of $M$, that is, is 
a $ \mathrm{Hom}
_A(M,M)$-$A$-subbimodule of 
$ M^*$, then $\hhom^R_A(M,M)$ 
is a two-sided ideal.   
\end{remark}
\begin{remark}
By definition,
$\hhom^R_A(M,N)$ is the image of 
the canonical map
$R \otimes _A N \rightarrow 
M^* \otimes _A N \rightarrow 
\mathrm{Hom} _A(M,N)$. 
If $M$ is $R$-locally
projective, this
canonical map is for $N=M$ 
injective, so that  
$ \hhom^R_A (M,M) \cong R \otimes _A
M$.
\end{remark}

\begin{proposition}
\label{locprojprop}
The following statements are
equivalent:
\begin{enumerate}
\item $M$ is \(R\)-locally
projective.
\item For each $m \in M$ there
exists $ \pi \in
\hhom^R_A(M,M)$ with $ \pi
(m)=m$.
\item For 
each finitely generated
submodule \(N \subseteq M\)
there exists $ \pi \in
\hhom^R_A(M,M)$ with $
\pi|_N = \mathrm{id} _N$.
\end{enumerate}
\end{proposition}
\begin{proof}
``$ (3) \Rightarrow (1)$'': Let $X$ be
a right $A$-module and 
assume that 
$x = \sum_i x_i \otimes _A m_i
\in X \otimes _A M$ is an
element with $ \evl_{X,r} (x)
=\sum_i x_i r(m_i) =0$ for all
$ r \in R$. 
Let $N \subseteq M$
be the $A$-module generated by
the $m_i$, and $ \pi \in \hhom^R_A(M,M)$ be as
in (3), with $ \pi (m) =
\sum_j r_j (m) n_j$. 
Then $ \pi (m_i)=m_i$ implies
$$
	x = \sum_i x_i \otimes _A
m_i = 
	\sum_{ij} x_i \otimes _A 
	r_j(m_i) n_j = 
	\sum_{ij} x_i r_j(m_i)
	\otimes _A n_j = 0.
$$

``$ (1) \Rightarrow (2)$'': Assume
that $M$ is $R$-locally
projective. Let $J
\subseteq A$ be the right 
ideal generated by all
elements of the form $r(m)$
with $r \in R$, and set
$X:=A/J$ so that 
$X \otimes _A M \cong M/JM$ as
abelian group. 
For all $r \in R$,
we now have 
$$
	\evl_{A/J,r} ([1] \otimes _A m) = 
	[r(m) ] = 0,
$$
where $[a] \in A/J$ is the
class of $a \in A$.
So by assumption, $ [1] \otimes
_A m = 0$. This means that $m \in
JM$, so there exist $
a_1,\ldots,a_d \in
J,m_1,\ldots,m_d \in M$ with 
$ m = \sum_i a_i m_i$. 
By the definition of $J$, we have 
$$
	a_i = \sum_j
	r_{ij}(m)b_{ij}
$$ 
for some  
$r_{ij} \in R, b_{ij} \in A$,
and we
conclude 
$$ 
	m = \sum_{ij}
	r_{ij}(m) m_{ij},
	\quad 
	m_{ij}:=b_{ij}m_i.
$$    
So $ \pi (x) := 
\sum_{ij} r_{ij} (x) m_{ij} $ 
defines an element of  
$\hhom^R_A(M,M)$ with $ \pi (m)
= m$. 

``$(2) \Rightarrow (3)$'': 
We prove by
induction on $ l \ge 0$ that
for any given elements 
$n_1,\ldots,n_l \in M$, there
exists $ \pi \in \hhom^R_A(M,M)$
with $ \pi  (n_j) = n_j$ for
$j \le l$. The claim then
follows by choosing $n_1,\ldots,n_l$
as a finite
set of generators of $N$. 
The base case $l=1$ follows
from (2), so assume that the claim is
proven for any set of $l-1$
elements and that
$n_1,\ldots,n_l \in M$ are
given. 
Choose auxiliary elements
$ \alpha , \beta \in \hhom^R_A(M,M)$ with 
$ \alpha  (n_j) =
n_j$ for $ j < l$ and 
$ \beta (a_l) = a_l $ (which exist by the
induction hypothesis), where 
$$
	a_l := n_l - \alpha (n_l) 
$$
so that we have
\[
\beta (n_l) - \beta \circ 
	\alpha (n_l) =
	\beta (a_l) =
	a_l = n_l - \alpha (n_l).
\]
Now define
$$
	\pi := 
	\alpha +
	\beta - 
	\beta \circ \alpha.
$$
Then we have 
$$
	\pi  (n_j) = 
	\alpha (n_j)  
	+ \beta (n_j)   
	- \beta \circ \alpha
	(n_j)  =
	n_j   
	+ \beta (n_j)  
	- \beta (n_j) = 
	n_j	
$$
for $j < l$ and 
\[
	\pi  (n_l) = 
	\alpha (n_l)  
	+ \beta (n_l)   
	- \beta \circ \alpha
	(n_l)  =
 	\alpha (n_l) 
	+ n_l
 - \alpha (n_l) =
	n_l.\qedhere
\]
\end{proof}

\begin{remark}
We stress that 
this does not mean that  
finitely generated submodules $N$ of
$M$ are 
projective, since $ \pi \in \hhom^R_A(M,M) $
is usually neither
idempotent, nor has it image $N$.
However, the next corollary 
establishes a weak form
of the lifting property of
projective modules. 
\end{remark}

\begin{corollary}
The $A$-module $M$ is $R$-locally projective
if and only if for all
finitely generated submodules
$N \subseteq M$ and all
$A$-module morphisms  
$ \varphi \colon M \rightarrow W,
 \psi \colon V \rightarrow W$ 
with $ \psi $ 
surjective, there exists 
$ \rho \in \hhom^R_A (M,V)$ with 
$ (\psi \circ \rho)|_N = 
\varphi|_N$.   
$$
	\begin{tikzcd}
	N \arrow[r,hook] &
	M \arrow[d,"\varphi"]
	\arrow[dl,dashrightarrow,"\exists\rho" {description,
xshift=-1ex,yshift=1ex}] \\
	V \arrow[r,"\psi" swap] 
	& W
	\end{tikzcd}
$$

\end{corollary}
\begin{proof}
``$ \Rightarrow $'': Given $N$,
$\varphi $, and $ \psi $, choose $ \pi $ as in the
proposition with $ \pi (m) = \sum_i r_i
(m) m_i$. For each $i$, choose 
$v_i \in V$ with $ \psi (v_i) = 
 \varphi (m_i)$, and set 
$$ 
	\rho (m) := \sum_i r_i(m) v_i \in
	V.
$$ 
Then we have for $n \in N$ 
\begin{align*} 
\psi (\rho (n)) &= 
\sum_i r_i(n) \psi  (v_i) = 
\sum_i r_i (n) \varphi (m_i) \\
&=  \varphi (\sum_i r_i(n) m_i) = 
\varphi (n).
\end{align*}

``$ \Leftarrow$'': Choose a
set $S$ of generators of $M$.
Consider the instance of the
corollary in which $V:=\bigoplus_S A$,
$W:=M$, 
$ \varphi := \mathrm{id} _M$
and $	\psi \colon \bigoplus_S A \rightarrow
	M$ is the map
that sends a function $f
\colon S \rightarrow A$ with
finite support to 
$\sum_{s \in S} f(s) s \in M$.
Let 
$ \rho \in
\hhom^R_A(M,\bigoplus_S A)$
be as in the corollary with
$ \rho (m) = \sum_{i=1}^l r_i(m)
f_i$. Let $T \subseteq S$ be the
(finite) set which is the
union of the supports of 
$f_1,\ldots,f_l$. 
Then for all $n \in N$,
we have  
$$
	n = \psi (\rho (n)) = \sum_{i=1}^l r_i(n)
	\sum_{s \in T} f_i(s) s = 
	\sum_{s \in T} 
	\sum_{i=1}^l (r_i f_i(s))(n)  
	s.
$$ 
So if we enumerate
$T=\{m_1,\ldots,m_q\}$ and
define $p_j := \sum_{i=1}^l
r_i f_i(m_j) \in R$, then 
$ \pi (m):= \sum_{j=1}^q
p_j(m) m_j$ defines an element 
$ \pi \in \hhom^R_A(M,M)$ with 
$ \pi (n) = n$ for all $n \in
N$. 
\end{proof}

\subsection{Strong local
projectivity}\label{slpsec}
While $R$-local projectivity 
is sufficient to 
apply 
Theorem~\ref{descenttheorem}, the
modules that appear in our
applications have in
fact a stronger property:

\begin{definition}
We say that \(M\) is
\emph{strongly \(R\)-locally
projective}
if for every finitely generated
\(A\)-submodule \(N \subseteq M\) 
there exists an idempotent 
$ \pi \in \hhom^R_A(M,M)$ (so 
$ \pi = \pi \circ \pi $) and 
$ \pi|_N = 
\mathrm{id} _N$. 
\end{definition}

\begin{remark}
\begin{enumerate}
\item By the dual basis lemma,
	a submodule \(Q \subseteq M\)
	is a direct summand of \(M\)
	which is finitely generated projective
	if and only if there exists 
	an idempotent element \(\pi\)
	$ \pi \in \hhom^{M^{*}}_A(M,M)$
	such that \(\im \pi = Q\).

\item In \cite{markiMoritaEquivalenceRings1987},
a module is called
\emph{locally projective}
if each finitely generated
submodule $N \subseteq M$ is
contained in a finitely generated
projective submodule which is a direct
summand of $M$.
The definition above,
first introduced in
\cite{vercruysseGaloisTheoryCorings2007}, 
generalises this notion,
by restricting 
the projective submodules
to those which are images
of idempotent elements
\(\pi \in \hhom^R_A(M,M)\).
\end{enumerate}
\end{remark}

For the examples of rings of
differential operators, 
we will verify the
strong local projectivity by
explicitly constructing a
sufficiently large system of
idempotents in $\hhom^R_A(M,M)$:

\begin{lemma}
\label{stronglocidem}
\(M\) is strongly  \(R\)-locally
projective if and only if
there exists a subset
\(S \subseteq \hhom_{A}^{R}(M,M)\)
consisting of idempotent elements
with the property that
\begin{enumerate}
\item \(0 \in S\),
\item for all \(\alpha,\beta \in S\),
 \(\alpha + \beta - \beta \circ \alpha \in S\)
 if  \(\alpha \circ \beta = 0\), and
\item for any  \(m \in M\) and
 \(\alpha \in S\),
there exists  \(\beta \in S\)
such that
 \[
\beta(m - \alpha(m)) = m - \alpha(m)
\text{ and }
\alpha \circ \beta = 0.
\]
\end{enumerate}
\end{lemma}
\begin{proof}
``\(\Rightarrow\)"
We choose \(S\) to be the set of all
idempotent elements in \(\hhom_{A}^{R}(M,M)\).
Then \(0 \in S\), so (1) is satisfied. 
If \(\alpha,\beta \in \hhom_{A}^{R}(M,M)\)
are idempotent
with \(\alpha \circ \beta = 0\),
then
\[
	\pi \coloneqq \alpha + \beta
	- \beta \circ \alpha
\]
is again idempotent, so (2) is satisfied.
Let \(m \in M\), 
\(\alpha \in S\), and 
consider
the submodule \(N \subseteq M\)
generated by  \(m\) and  \(\im \alpha\).
It is finitely generated, 
so there exists
an element \(\pi \in S\) such that 
\(\pi|_{N} = \id_{N}\)
by the strong
\(R\)-locally projectivity of \(M\).
This means in particular 
that \(\pi \circ \alpha = \alpha\) and \(\pi(m)=m\).
Simple computations then show
that \(\beta \coloneqq \pi -
\alpha \circ \pi\)
is idempotent, thus is in \(S\),
that \(\alpha \circ \beta = 0\),
and that 
\(\beta(m - \alpha(m)) 
= m - \alpha(m)\),
thus (3) is satisfied.

``\(\Leftarrow\)"
We perform the same induction proof 
as in \cref{locprojprop}~(3), with the same
notation.
The only difference here is that 
we use the properties of \(S\)
to ensure that the
element \(\pi \in \hhom_{A}^{R}(M,M)\) 
we construct 
in the induction step 
is in \(S\), and is thus idempotent.
Let \(n_{1}, \ldots, n_{l} \in M\).
The base case \(l = 1\) follows from
applying the property (3) of \(S\)
to  \(n_{1}\) and \(\alpha = 0 \in S\). 
Next, suppose \(l > 1\)
and let \(\alpha \in S\) be
such that  \(\alpha(n_{j}) = n_{j}\)
for all \(j < l\).
Property (3), applied
to \(n_{l}\) and \(\alpha\),
yields an element
\(\beta \in S\) such that 
\[
	\beta(n_{l}-\alpha(n_{l}))
	= n_{l} - \alpha(n_{l})
	\text{ and }
	\alpha \circ \beta = 0.
\]
It follows then that
\(\pi \coloneqq \alpha + \beta 
- \beta \circ \alpha \in S\),
and that
\(\pi(n_{j}) = n_{j}\) 
for all \(j=1,\ldots,l\).
\end{proof}

A standard situation in which
strong local projectivity holds is
the following:
assume that \(P\) is a right
\(A\)-module, \(\{I_ \lambda
\}_{\lambda \in
\Lambda}\) is a directed set
of
submodules, $P_ \lambda :=
P/I_ \lambda $, and  
\(M \coloneqq \colim_{\lambda
\in \Lambda}
P_{\lambda}^{*}\). 
The inclusions 
$P^*_ \lambda \rightarrow P^*$
dual to the quotient $P
\rightarrow P_ \lambda $ yield
a canonical morphism of left
$A$-modules $M \rightarrow
P^*$, and hence a 
morphism $ \nu \colon P \rightarrow
P^{**} \rightarrow M^*$.  Let
$R := \mathrm{im}\, \nu \subseteq M^*$ be the image
of this map. Then we have:

\begin{lemma}
\label{stronglocprojcolim}
If \(P_{\lambda}\) is finitely generated projective
for all \(\lambda \in \Lambda\), then
the left $A$-module $M$
is strongly \(R\)-locally
projective. 
Conversely, any strongly
$R$-locally projective module
$M$ arises as such a colimit.
\end{lemma}
\begin{proof}
Let \(N \subseteq M\) 
be a finitely generated
submodule, $N = \sum_i Am_i$. 
If we identify $P_ \lambda ^*$
with its image in $M$, then 
$M = \sum_\lambda P_\lambda
^*$ because \(M\) is a
colimit. So there exist 
$\lambda_i^1,\ldots, \lambda
_i^{d_i} \in \Lambda $ with 
$ Am_i \in \sum_{j=1}^{d_i} 
P_ {\lambda_i^j}^*$. As $
\Lambda $ is directed, there
exists an upper bound 
$ \lambda \in \Lambda $ 
so that for all $i,j$, we have 
$ I _{\lambda_i^j}
\subseteq I_ \lambda $, hence 
there exists an $A$-module
epimorphism
$P _ \lambda \rightarrow P_{
\lambda _i^j}$ so that 
$ P_{\lambda _i^j}^*
\subseteq P_ \lambda ^*$. It
follows that $N \subseteq
P_\lambda ^*$.    
As \(P_\lambda\) is finitely
generated projective, the dual 
module \(P_\lambda^{*}\) is finitely
generated projective as well. 
By the dual basis lemma, there
exist $n_1,\ldots,n_l
\in P _\lambda^*$ and  
$ \varphi _1,\ldots, \varphi _l 
\in P_\lambda ^{**}$ with 
$ n = \sum_i \varphi _i(n)
n_i$ for all $n \in P_\lambda
$. However, as $P_ \lambda $
is finitely generated
projective it is reflexive. 
So there exist
$p_i \in P$ representing 
$[p_i] \in P_ \lambda
\cong P_\lambda^{**}$
such that 
\[ 
	\varphi _i (n)
	= n( [p_i]) = \nu(p_i) (n)
	\quad \forall n \in
P_\lambda^*.
\]  
It follows that 
\[
 n = \sum_{i = 1}^{l} n([p_{i}]) n_{i}
 = \sum_{i=1}^{l} \nu(p_{i})(n) n_{i}
	=: \pi (n) .
\]
For the converse, 
take 
$P:=R$, $ \Lambda $ to be the 
set of finitely generated
projective direct summands
$M_\lambda$ of
$M$, and $I_ \lambda
\subseteq R$ to be
the submodule of those $ \varphi \in R$
with $ \varphi |_{M_\lambda} =
0$. Then $P_ \lambda^* \cong
M_\lambda$ and $ M=\colim
M_\lambda$ (by strong local
projectivity, every  
$m \in M$ is contained in some
$M_ \lambda $). 
\end{proof}

\subsection{Descent}\label{bcca-sec}
Assume we are given an algebra
morphism 
\[ 
	\iota _0 \colon A \rightarrow
K,
\]
a $K$-module $C$,  
an $A$-module $D$, and a morphism 
of $A$-modules 
\[ 
	\iota _1 \colon D
	\rightarrow C,
\]
where $C$ is an
$A$-module via $ \iota _0$.
For $n \ge 0$, we abbreviate   
$$ 
	\iota _n \colon 
	D^{\otimes _A n} 
	\rightarrow 
	C^{\otimes _K n},\quad
	x_1 \otimes _A \cdots 
	\otimes _A x_n \mapsto 
	\iota _1(x_1) \otimes _K
	\cdots 
	\otimes _K \iota _1(x_n).
$$  

\begin{definition}
\label{restrictcoalg}
We say that a $K$-coalgebra
structure $( \Delta
,\varepsilon )$ on $C$  
\emph{restricts} (or
\emph{descends}) 
to an  \(A\)-coalgebra
structure  $( \Delta _D, \varepsilon
_D)$ on \(D\) if the following diagrams
commute:
\[
	\begin{tikzcd}
		D \arrow[r, " \iota _1"]
		\arrow[d, "\Delta_{D}"']
		& C 
		\arrow[d, "\Delta"] \\
		D \otimes_{A} D 
		\arrow[r,"\iota_{2}"]
		& C \otimes_{K} C
	\end{tikzcd},
	\qquad
	\begin{tikzcd}
		D \arrow[r, " \iota _1"] 
		\arrow[d, "\varepsilon_{D}"']
		& C \arrow[d, "\varepsilon"]
		\\
		A \arrow[r, " \iota _0"] 
		& K
	\end{tikzcd}.
\]
\end{definition}

The easiest situation 
is when all $ \iota _n$ are
embeddings. Then we have:

\begin{lemma}
\label{basechange}
If all $ \iota _n$,
$n\ge 0$, are injective, then 
a coalgebra structure $( \Delta ,
\varepsilon )$ on $C$
admits a restriction $( \Delta
_D, \varepsilon _D)$ to $D$ if
and only if we have 
$$ 
	\mathrm{im} (\varepsilon
	\circ \iota _1)	
	\subseteq \mathrm{im}\,
\iota _0,\quad
	\mathrm{im}(\Delta \circ
\iota _1) \subseteq
\mathrm{im}\, \iota _2.
$$ 
In this case, the restriction
is unique, and we have:
\begin{enumerate}
\item If $C$ is cocommutative,
then $D$ is cocommutative. 
\item If $C$ is conilpotent
with coaugmentation
by a group-like $g \in
\mathrm{im}\, \iota _1
\subseteq C$, 
then $D$ is conilpotent. 
\item If $C$ is a 
\coring{K}{K}-coring
and $ \iota _1(D) \subseteq C$ 
is an $A$-$B$-subbimodule for
some subalgebra $B \subseteq
K$, then $D$ is a
\coring{B}{A}-coring.
\item \label{descentbialg}
	If $C$ is a bialgebroid
over  \(K\),
$D$ is an $A$-ring, and   
$ \iota _1$ is morphism of
$A$-rings, 
then  $D$ is a
bialgebroid  
 over  \(A\).
\item In this case, if 
$S$ is an antipode for
$C$ that restricts to $D$,
then $S|_D$ is an antipode for
$D$.  
\end{enumerate}
\end{lemma}
\begin{proof}
The claims are mostly obvious;
the injectivity
of $ \iota _3$ ensures
that the coassociativity of 
$ \Delta $ implies that of 
$ \Delta _D$. Note that 
the injectivity of the
$ \iota _n $, $n > 3$, is only needed
to infer conilpotency.
In this case, 
the primitive filtration of \(D\)
is the restriction
of the one of \(C\).
\end{proof}

\begin{remark}\label{polishedrem}
If the coalgebra \(C\) 
is graded, 
then its restriction is also graded.
However, 
if \(C\) is filtered,
then its restriction
is not filtered in general.
Indeed, for the filtration  
\(D_{n} \coloneqq \iota_{1}^{-1}(C_{n})\),
the restriction
\(\Delta_{D}\)
always maps \(D_{n}\) to
\(\iota_{2}^{-1}((C
	\otimes_{K} C)_n)\),
but not necessarily to
\((D \otimes_{A} D)_{n}\),
which is in general a proper submodule
of
\(\iota_{2}^{-1}((C
	\otimes_{K} C)_n)\).
We will
see in \cref{remfiltration} 
that this is the reason why in general
the coproduct on 
differential operators
does not respect the
filtration by order. 
\end{remark}

Finally, we want to restrict
the property of being a left
Hopf algebroid, but this
requires additional
assumptions:

\begin{lemma}
\label{descentlefthopflem}
Assume \(D\) and \(C\)
are bialgebroids as in
\cref{basechange}\cref{descentbialg}.
If \(C\) is a left Hopf algebroid,
the canonical map
\(\bar{\iota}_{2}:
\bar{D} \otimes_{A} D 
\to \bar{C} \otimes_{K} C\)
is injective, and
\(\im (\tau_{C} \circ \iota_{1}) \subseteq 
\im \bar \iota_{2}\),
then \(D\) is a left Hopf algebroid.
\end{lemma}
\begin{proof}
The following diagram commutes
\[
\begin{tikzcd}
	\bar{C} \otimes_{K} C 
	\ar[r, "\delta_{C}"] &
	C \otimes_{K} C \\
	\bar{D} \otimes_{A} D 
	\ar[u, "\bar{\iota}_{2}"] 
	\ar[r, "\delta_{D}"] &
	D \otimes_{A} D
	\ar[u, "\iota_{2}"'].
\end{tikzcd}
\]
Since \(\bar{\iota}_{2},
\iota_{2}\) are injective, 
and \(\delta_{C}\)
is bijective,
\(\delta_{D}\) is injective. 
It is surjective
if and only if
\(\delta_{C}^{-1}\)
maps 
\(\im \iota_{2}\) 
to \(\im \bar{\iota}_{2}\).
This, however, is equivalent
to the condition that 
\(\tau_{C} = \delta_{C}^{-1}( \cdot
\otimes_{K} 1)\) maps \(\iota_{1}(D)\)
to \(\bar{\iota}_{2}(\bar{D} \otimes_{A} D)\),
since \(\delta_{C}^{-1}\)
is right \(C\)-linear.
\end{proof}

\subsection{Base change}
We now focus on
the case that 
the $K$-coalgebra $C$ 
is obtained by base change
from a coalgebra $D$ over
$A \subseteq K$.

More precisely, 
assume first that $ \iota_0 \colon A
\hookrightarrow K$ is an
inclusion of a subalgebra, $D$
is an $A$-module, and define 
$$ 
	C \coloneqq K \otimes _A D,\quad
	\iota _1 \colon 
	D \rightarrow C,\quad 
	d \mapsto 1 \otimes
_A d.
$$

In our applications, $K$ is not
faithfully flat over $A$, so it is a
priori not guaranteed that 
$ \iota _1$ is injective. However, 
let us
suppress the canonical embedding 
\begin{align}\label{canemb}
	D^* 
	= 
	\mathrm{Hom} _A(D,A) 
&\hookrightarrow 
	\mathrm{Hom} _K(C,K) 
\cong
	\mathrm{Hom} _A(D,K),\\
	r 
&\mapsto \evl_{K,r},\nonumber
\end{align}
 and 
view any $R \subseteq D^*$ 
as a subset of $ \mathrm{Hom}
_K(C,K)$. 
The observation made next
is that for $R$-locally
projective $D$ it
is sufficient to verify that
the $K$-algebra structure
of $ \mathrm{Hom} _K(C,K)$ 
dual to a
$K$-coalgebra structure on $C$ 
(Lemma~\ref{dualalgebra}) 
restricts to an $A$-algebra
structure on $R$ in order to
deduce that the $K$-coalgebra
$C$ is obtained by base
change: 

\begin{lemma}\label{standinlemma}
If $D$
is $R$-locally projective, 
then a $K$-coalgebra structure
on $C=K \otimes
_A D$ for which $R$ is an 
$A$-subalgebra of 
$ \mathrm{Hom} _K(C,K)$ with
respect to the convolution
product
(\ref{convolutionproduct})
restricts uniquely
to \(D\).
\end{lemma}
\begin{proof}
We first show that all $ \iota
_n$ are injective with image
\begin{equation}\label{imiotan}
	\mathrm{im}\, \iota _n = 
	\{
	x \in
	C^{\otimes _K n}  \mid 
	\forall r_1,\ldots,r_n \in
	R : (r_n \otimes _A \cdots 
	\otimes _A r_1) (x) 
	\in A\}, 
\end{equation}
where we extend the
identification of $ r \in D^*$ with
$\evl_{K,r} \colon 
C \rightarrow K$ to tensor
products, that is, denote by 
$ r_1 \otimes _A \cdots 
\otimes _A r_n$ 
the $K$-linear
functional 
$$
	C^{\otimes _K n}
	\rightarrow K,\quad 
	c_1 \otimes _K \cdots 
	\otimes _K c_n \mapsto 
	\evl_{K,r_1}(c_1) 
	\cdots 
	\evl_{K,r_n}(c_n).
$$ 
Indeed, under the canonical
identification
$$ 
	C^{\otimes _K n} 
= 
	(K \otimes _A D) \otimes _K 
	\cdots \otimes _K
	(K \otimes _A D) 
\cong 
	K \otimes _A D^{\otimes
	_A n},
$$
the maps $ \iota _n$ get
identified  
with the base change morphisms
$$
	D^{\otimes _A n} 
	\rightarrow K \otimes _A 
	D^{\otimes _A n},\quad 
	d_1 \otimes_A \cdots 
	\otimes_A d_n 
	\mapsto 
	1 \otimes _A 
	d_1 \otimes_A \cdots 
	\otimes_A d_n.
$$
Both the injectivity of $
\iota _n$ and the description
(\ref{imiotan}) of its image now follow readily by
an inductive application of 
Lemma~\ref{flatlemma}.

The fact that $R$ 
is a subring of $ \mathrm{Hom}
_K(C,K)$ means on the
one hand that
the unit of $ \mathrm{Hom}
_K(C,K)$, which is 
$ \varepsilon $, belongs to $R$, 
or, more explicitly, that 
$ \varepsilon
= \evl_{K,e}$ for some $e \in
R$. In particular,
$ \varepsilon (1 \otimes
_A d) \in A$ for $d \in D$,
which is the condition
$ \mathrm{im}\, (\varepsilon \circ
\iota _1) \subseteq \mathrm{im}\,
\iota _0$ from
Lemma~\ref{basechange} above. 
On the other hand, it means
that $R$ is closed under
convolution, so for any \(r,s \in R\),
there exists some \(t \in R\)
such that for any \(d \in D\),
we have
 \[
	(r \otimes _A s) 
	(d_{(1)} \otimes_{A} d_{(2)})
= 
	(\evl_{K,r} * \evl_{K,s} )
	(1 \otimes_{A} d)
=
	\evl_{K,t}
	(1 \otimes_{A} d)=t(d).
\]
In particular, we have 
for all $r,s \in R$
\[
	(r * s) ( \mathrm{im}\,
\iota _1) = 	
(r \otimes _A s) (
	\mathrm{im} (\Delta \circ
	\iota _1)) \subseteq A,
\] 
hence the description of $
\mathrm{im}\, \iota _2$
obtained in (\ref{imiotan}) yields 
the last condition 
$ \mathrm{im}\, (\Delta \circ \iota
_1) \subseteq \mathrm{im}\, \iota
_2$ from
Lemma~\ref{basechange}.
\end{proof}

We also refer
to \cite{wisbauerCategoryComodulesCorings2002} for
for a general study
of the relevance of local
projectivity specifically of
coalgebras.

\subsection{The descent
theorem}
\label{subsecCAB}
Finally, we zoom in 
on the instance of
Lemma~\ref{standinlemma} 
that we are really
interested in:
let 
$C$ be a given 
bialgebroid over $K$,   
$A,B \subseteq K$ be
subalgebras, and define 
\begin{equation}\label{CAB}	
	C(B,A) \coloneqq 
	\{ c \in C \mid
	\forall b \in B :	
	\varepsilon(cb) \in A
	\}.
\end{equation}
When \(A = B\), 
we shall write \(C(A)\)
instead of \(C(A,A)\).

Just from the fact that 
$C$ is an augmented $K$-ring, we
obtain:

\begin{lemma}\label{praha}
We have:
\begin{enumerate}
\item 
\(C(A) \subseteq C\)
is an augmented  \(A\)-subring, and
\item $C(B,A) \subseteq C$ is a
$C(A)$-$C(B)$-subbimodule.  
\end{enumerate}
If furthermore \(A \subseteq B\), 
then 
\begin{enumerate}
\item \(C(B,A)\)
is a left 
ideal of \(C(A)\),
\item \(C(B,A)\)
is a right
ideal of \(C(B)\), and
\item \(C(A) \subseteq C(A,B)\).
\end{enumerate}
\end{lemma}
\begin{proof}
$K$ carries a left $C$-module
structure that extends
multiplication in $K$ and is given
by $ \hat \varepsilon (c) (x) = 
\varepsilon (cx)$, $c \in C$,
$x \in K$. That is,
for all $c,d \in C$ and $x \in K$, 
we have 
$
	\varepsilon (cdx) =
	\varepsilon (c
	\varepsilon (dx)). 	
$
Now all claims follow
directly from the definitions of 
$C(A), C(B)$, and $C(A,B)$.
\end{proof}

In particular, $C(B,A)$ is 
an \(A\)-\(B\)-subbimodule
of $C$. Last but not least, 
consider the
$B$-$A$-subbimodule of 
$C(B,A)^*$ generated by 
$ \varepsilon $, 
$$
	R(B,A) \coloneqq 
	\{ r \in C(B,A)^*
\mid \exists a_i \in A,b_i \in
B : r(c) = \varepsilon( 
	\sum_i a_i c b_i)\}.  
$$ 

Now we can prove
Theorem~\ref{descenttheorem}:

\begin{theoremm}
If 
\(C(B,A)\) is 
\(R(B,A)\)-locally
projective and 
\[ 
	\mu \colon  
	K \otimes _A C(B,A) 
	\rightarrow C,\quad
	x \otimes _A c 
	\mapsto xc 
\]
is surjective,  
then we have:
\begin{enumerate}
\item 
$(\Delta , \varepsilon )$
uniquely restricts to a
\coring{B}{A}-coring structure on  
\(C(B,A)\).
\item 
If $B=A$, this turns
\(C(A)\) into a bialgebroid over
\(A\).
\item If $C$ is in addition 
a left Hopf
algebroid, then $C(A)$ is a left
Hopf algebroid. 
\end{enumerate}
\end{theoremm}

\begin{proof}
``(1)'':
The assumptions made imply that we
are
given an instance of
Lemma~\ref{standinlemma} with 
$$
	D = C(B,A),\qquad R=R(B,A).
$$
To see this, 
note first that the $R(B,A)$-local
projectivity of 
$C(B,A)$ implies 
the injectivity of 
$$ 
	\mu \colon 
	K \otimes _A C(B,A)
\rightarrow C.
$$
Indeed, if for 
$c = \sum _i x_i \otimes_A c_i \in 
K \otimes _A C(B,A)$ we have
$\sum_i x_ic_i = 0$ in $C$, then 
we have in particular
\[
	0=
	\varepsilon (\sum_iax_ic_ib)=
	\sum_i x_i \varepsilon (ac_ib)
\]
for all $ b \in B,a \in A$, 
so
$ \evl_{K,r} (c) =0$ for all 
$r \in R(B,A)$ and $c=0$. 

Second, knowing that   
$ \mu $ is an isomorphism we can
view $R(B,A)$  
as in the discussion preceding 
Lemma~\ref{standinlemma} 
as a subset  of 
$ \mathrm{Hom} _K(C,K)$, 
and as such it is simply
the image of $A \otimesk B$ 
under the
unit map $\abeins$ 
from (\ref{coringeta}). 
In particular, $R(B,A)$ is 
a (commutative)   
$A \otimesk B$-subalgebra of 
$ \mathrm{Hom} _K(C,K)$ 
that is isomorphic to
\begin{align*}
	R(B,A) 
&\cong 
	(A \otimesk B)/I(B,A),\\
	I(B,A) 
&= \{\sum_i a_i \otimesk 
	b_i \in A \otimesk B \mid 
	\forall c \in C(B,A) : 
	\varepsilon (\sum_i a_icb_i)=0
	\}. 
\end{align*}
So Lemma~\ref{standinlemma} implies that 
$ (\Delta , \varepsilon )$ uniquely
restricts to an $A$-coalgebra
structure on $C(B,A)$, and 
as a special case of
Lemma~\ref{basechange}, this
is in fact a
\coring{B}{A}-coring.

``(2)'': This is immediate (using the
first statement in 
Lemma~\ref{praha}).

``(3)'': 
This is an instance of
\cref{descentlefthopflem}. To apply
the lemma, we need to show 
that the canonical map 
\(\bar{\iota}_{2} \colon 
\overline{C(A)} \otimes_{A} C(A) \to
\bar{C} \otimes_{K} C\) is injective.
We assume that 
$C \cong K \otimes _A C(A)$, so 
\(\bar{\iota}_{2}\) 
is equal to the composition
\[
\begin{tikzcd}
	\overline {C(A)} \otimes _A C(A) 
	\ar[r, "\bar \iota "] &  
	\bar C \otimes _A C(A)
	\ar [r, "\cong"] &
	\bar C \otimes _K C,
\end{tikzcd}
\]
where the map 
\(\bar{\iota}\) 
is induced by the inclusion
$\overline{C(A)} \subseteq \bar{C}$.
However, we assume 
 that $C(A)$
is $R(A)$-locally projective,
so by \cref{flatlemma},
\(\bar{\iota}\) is injective.
Hence \(\bar{\iota}_{2}\) is
injective as required.

We shall show next that
the translation map
\[
	\tau_{C} \colon
	C \to \bar{C} \otimes_{K} C,
	\quad
	x \mapsto x_{+} \otimes_{K} x_{-}
\]
maps \(C(A)\) to
\(\im \bar{\iota}_{2}\).
To do so, we will use the description of
\(\im \bar{\iota}\)
given in
Lemma~\ref{flatlemma}, and 
show that
$$ 
\ev_{\bar C,r} (\tau_C(x))
= 
x_+ r(x_-)  
\in \overline{C(A)}
$$
holds for all 
$x \in C(A)$ and $r \in R(A)$.
As shown in \cite[Proposition~3.7]{
schauenburgDualsDoublesQuantum2000},
we have for all $x \in C,a \in
K$ that 
$$
	x_+ \otimes _K x_-a = 
	ax_+ \otimes _K x_-,\quad 
	x_+ \varepsilon (x_-) = 
	\varepsilon (x_-) \blact x_+ = 
	x.
$$
If \(r \in R(A)\) is given by 
$ r (c) \coloneqq \sum_i 
\varepsilon(a_icb_i)$
for some $a_i,b_j \in A$,
then using the above, 
we obtain
\begin{align*}
	x_+ r(x_-) &= 
	\sum_i x_+ 
	\varepsilon (a_ix_-b_i)= 
	\sum_i b_ix_+ 
	\varepsilon (a_ix_-)\\
&= 
	\sum_i b_ix_+ 
	a_i\varepsilon (x_-) = 
	\sum_i b_ix_+ 
	\varepsilon (x_-)a_i = 
	\sum_i b_ixa_i. 
\end{align*}
for all \(x \in C\),
so the claim follows. 
\end{proof}

In combination with 
Proposition~\ref{opmonpropo}
and Lemma~\ref{praha}, this
also implies:

\begin{corollary}
If $C(A)$, $C(B)$,
and $C(B,A)$
are respectively 
$R(A)$-,
{$R(B)$-,} and 
$R(B,A)$-locally
projective,
then the functor
$$
	C(B,A) \otimes_{C(B)} -:
C(B)\jmod \to C(A)\jmod
$$
is canonically opmonoidal.
\end{corollary}

\begin{remark}
Evidently, we have 
$ C(B,A) \subseteq C(E,A)$ for all
subalgebras $E \subseteq B$. In
particular, we have
\[ 
	C(B,A) 
	\subseteq 
	C(k,A) = 
	\{c \in C \mid 
	\varepsilon (c) \in
	A\}.
\] 
It is therefore a necessary condition 
for the applicability of the theorem
that $ \mu \colon 
K \otimes _A C(k,A) \rightarrow C$
is surjective. Once this is given, 
one may determine maximal
subalgebras $B \subseteq K$ for
which $ \mu $ is surjective and then
determine whether for these, $C(B,A)$
is $R(B,A)$-locally projective. 
\end{remark}

\section{The case of
differential
operators}\label{applsec}
We now apply Theorem~\ref{descenttheorem} to
the case $C=\DK$ so that 
\[
	C(B,A) = \DK(B,A) = 
	\{ D \in \D_{K} \mid
		D(b) \in A
	\, \forall b \in B\}.
\]
In the literature, 
$\DK(M,N)$
is also used to denote the 
differential operators between
$K$-modules $M$ and $N$; we
hope this causes no confusion.  

We initially consider any algebra
morphism $A \otimesk B 
\rightarrow K$. 
Section~\ref{gppsec} discusses 
inverse systems of
$A \otimesk B$-algebras 
$
	\ldots \epic P_2 \epic P_1 \epic
	P_0 
$
that generalise the principal parts
from
\cref{principalpartdef}. 
In Section~\ref{eesec} we 
use these to deal with
the assumption 
$\DK \cong K \otimes 
_A \DK(B,A)$ in 
Theorem~\ref{descenttheorem}; our
results are a slight
generalisations of results of 
Másson 
\cite{massonRingsDifferentialOperators1991}
who considered the case 
$A=B=P_0$. 
In Section~\ref{sec53}, we
rephrase 
Theorem~\ref{descenttheorem}
and its consequences in 
the setting $C = \DK$.

\subsection{Generalised principal
parts}\label{gppsec}

In Theorem~\ref{descenttheorem}, 
$R(B,A)$ is the image of the
canonical map
$$
	A \otimesk B \rightarrow 
	C(B,A)^*,\quad 
	a \otimesk b \mapsto r,\quad
	r (c):= a \varepsilon (cb) =	
	\varepsilon (acb).  
$$
When $A=B=K$ and 
$C=C(B,A)=\DA$, this map is 
$$
	A \otimesk A \rightarrow 
	{\DA}^*,\quad 
	a \otimesk b \mapsto r,
	\quad 
	r(D) := aD(b),
$$ 
which is  
the composition of the 
canonical map 
\[
	A \otimesk A \rightarrow 
	((A \otimesk A)^*)^* =
	\mathrm{Hom} _A(A 
	\otimesk A,A)^* \cong 
	\EndA^*
\]
with the restriction map 
$ \EndA^* \rightarrow 
\DA^*$ that is dual to the inclusion
$$
	\DA \hookrightarrow 
	\EndA \cong \mathrm{Hom} _A(A
	\otimesk A,A).
$$ 
Thus $R(A)$ consists here
precisely of the functionals that
are given by principal parts (recall
Definition~\ref{principalpartdef}
and Remark~\ref{symbolcalculus}). 

In our application of
Theorem~\ref{descenttheorem}
to differential operators on
singular curves, we will
consider the following more
general setting: 

\begin{definition}
Let $A \rightarrow K,B
\rightarrow K$ be two algebra
morphisms and $ \abeins
\colon P \rightarrow
K \otimesk K$ be a morphism
of $A \otimesk B$-algebras. 
Assume that 
\[
	P = I_0 
	\supseteq I_1
	\supseteq I_2 \supseteq 
	I_3 \supseteq \ldots
\]
is a descending chain of 
ideals of $P$
that satisfy (for 
all $m,n \ge 0$) 
\[
	I_1 = \sigma ^{-1}(I_K),\quad
	\sigma (I_n) \subseteq I_K^n,\quad
	I_{n} I_{m} 
\subseteq I_{n+m}.
\] 
Then we call the inverse system 
of the quotients 
\[
	P_{n}\coloneqq 
	P / I_{n+1}
\]
a \emph{system of 
generalised principal
parts}.
\end{definition}

In this situation, we denote 
for all \(n \in \N\) by 
 \[
 	\psi_{n} \colon
	K \otimes_{A} P_{n} \to
\P^{n}_{K},\quad
	x \otimes _A [p] \mapsto 
	[x \sigma (p)]
\]
the canonical morphism of 
$K$-algebras and by
\[
	\mu_{n} \colon K \otimes_{A} P_{n}
	\to K,
	\quad
	x \otimes_{A} [p]
	\mapsto x \mu _K(\sigma
(p))
\]
the multiplication
map.

\begin{remark}\label{ulislowlyunderstands}
Note that $P_0 = P/I_1 \cong \mu_K(
\sigma(P)) \subseteq
K$.
\end{remark}

We will be interested in the
case in which 
$K \otimes _A P_0 \cong K$ and
$K \otimes _A
P_n$ is for $n > 0$ 
an infinitesimal
deformation thereof. This will
be achieved using the
following lemma:

\begin{lemma}\label{lemmagoodf}
\label{kermu_n}
If \(A \rightarrow K\) is flat  and
$ \mu _0$ is injective (and
hence an isomorphism),
then we have isomorphisms  
\(\ker \mu_{n} \cong K \otimes_{A} (I_{1} / I_{n+1})\)
for all \(n\).
In particular, $
\mathrm{ker}\, \mu _n$ is nilpotent
with \((\ker \mu_{n})^{n+1} = 0\).
\end{lemma}
\begin{proof}
The following sequence
is exact since
\(K\) is flat over  \(A\):
\[
	\begin{tikzcd}
[column sep=small, row sep=small]
	0 \ar[r] &
	K \otimes_{A} (I_{1}/I_{n+1}) 
	\ar[r] &
	K \otimes_{A} P_{n} \ar[rr, "K
\otimes_{A} \pi_{n}"] & &
	K \otimes_{A} P_{0} \cong K
\ar[r] &
	0.
\end{tikzcd}
\]
Here \(\pi_{n} \colon P_{n} \epic P_{0}\)
is the canonical projection,
and the isomorphism \(K \otimes_{A} P_{0} \cong K\) 
is given by \(\mu_{0}\).
Therefore, the result  follows from
the fact that 
\(\mu_{n} = \mu_{0} \circ (K \otimes_{A} \pi_{n})\).
\end{proof}

\begin{example}\label{muoiso}
This is always satisfied
when  
\(A \rightarrow
K\) is a flat epimorphism of
rings (i.e.~$K$ is a flat
$A$-algebra whose
multiplication map  
$ K \otimes _A K \rightarrow
K$ is bijective),
and in particular if it
is a localisation.
\end{example}

\begin{example}\label{universalgpp}
When $P=\P_{A,B} := A \otimesk B$, 
there is a universal 
system of generalised
principal parts, namely
\[
	\quad I^{(n)}_{A,B} \coloneqq 
	\sigma ^{-1} (I_{K}^{n}),\quad
	\P^{n}_{A,B} 
	\coloneqq 
	(A \otimesk B)/I^{(n+1)}_{A,B}.
\]
Indeed, there exists by definition
for any system of generalised
principal parts
\((P_{n})\) a canonical morphism
\(P_{n} \to \P^{n}_{A,B}\) for all
\(n \in \N\). Note that the modules
 \(\P^{n}_{A}\)
 and \(\P^{n}_{A,A}\)
 are not isomorphic in general;
there is a canonical surjection
\(\P^{n}_{A} \twoheadrightarrow
\P^{n}_{A,A}\)
induced by the inclusion 
\(I^{n}_{A} \subseteq I^{(n)}_{A,A}\),
however, these ideals
are not necessarily equal.
\end{example}

\subsection{Étale
extensions}\label{eesec}
We now assume that
$ B \rightarrow K$ is
formally étale 
(Definition~\ref{etaledef}). 
Throughout this section, we fix 
a system 
$\{P_n\}$ of generalised
principal parts as above.

\begin{lemma}\label{improvedlemma}
If $A \rightarrow K$ is flat,
$B \rightarrow K$ is formally étale, and 
\(\mu_{0}\) is an isomorphism,
then for all \(n \in \N\),
the canonical morphism
 \[
	 \psi_{n} \colon K \otimes_{A} P_{n}
	 \to
	\P^{n}_{K}
\]
is an isomorphism of $A
\otimesk B$-algebras.
\end{lemma}
\begin{proof}
Let \(n \in \N\). Being étale
is easily seen to 
be preserved under any base change,
so $K \otimesk B \rightarrow 
K \otimesk K$ is formally
étale.
Thus since 
\(\ker \mu_{n}\) is nilpotent (Lemma~\ref{kermu_n}),
there exists a unique
$K \otimesk B$-algebra 
morphism \[\phi_{n} \colon K \otimesk K
\to K \otimes_{A} P_{n}\]
making the following diagram
commutative:
\[
\begin{tikzcd}
	K \otimesk K \ar[r, "\mu_{K}"]
	\ar[dr, dashrightarrow, "\phi_{n}"] &
	K \\
	K \otimesk B \ar[r]
	\ar[u] &
	K \otimes_{A} P_{n} \ar[u, "\mu_{n}"'],
\end{tikzcd}
\]
where the morphism
\(K \otimesk B \cong 
K \otimes _A (A \otimesk B)
\to K \otimes_{A} P_{n}\) 
is induced by the unit map 
$A \otimesk B \rightarrow P$
of the $A \otimesk
B$-algebra $P$. 

Since \(\phi_{n}\) maps \(\ker \mu_{K}\)
 to \(\ker \mu_{n}\) and 
 \((\ker \mu_{n})^{n+1} = 0\),
the morphism \(\phi_{n}\) 
factors uniquely through a morphism of algebras
\[
	\bar \phi_{n}  \colon \P^{n}_{K}
	\to K \otimes_{A} P_{n}.
\]

Let us show that \(\psi_{n}\)
 and \(\bar \phi_{n}\)
 are inverse to each other.
Let us denote by 
\(\chi_{n} \colon K \otimesk K \to \P^{n}_{K}\)
 the canonical surjection.
Both \(\chi_{n}\) and \(\psi_{n} \circ \phi_{n}\)
make the following diagram commutative,
\[
\begin{tikzcd}
	K \otimesk K \ar[r, "\mu_{K}"]
	\ar[dr, dashrightarrow] &
	K \\
	K \otimesk B \ar[r]
	\ar[u] &
	\P^{n}_{K} \ar[u, "\mu_{n}"'].
\end{tikzcd}
\]
As $K \otimesk B \rightarrow 
K \otimesk K$ is étale, 
this shows
\(\psi_{n} \circ
\phi_{n} = \chi_{n}\),
so that \(\psi_{n} \circ \bar \phi_{n} = \id_{\P^{n}_{K}}\).
A similar argument shows 
that \(\phi_{n} \circ \psi_{n} = \id_{K \otimes_{A} P_{n}}\).
\end{proof}

\begin{corollary}
\label{DKcolimit}
Under the assumptions of
Lemma~\ref{improvedlemma},
we have:
\begin{enumerate}
\item \(\D^{n}_{K} \cong 
\mathrm{Hom} _K(\P_K^n,K) \cong
\Hom_{A}(P_{n}, K)\),
\item 
If $A \rightarrow K,B
\rightarrow K$ are inclusions,
this isomorphism induces a 
canonical embedding 
\[ 
	\mathrm{Hom} _A(P_n,A)
	\monic \DK^n(B,A)
	:= 
	\DK^n \cap \DK(B,A).
\]
This is an isomorphism provided
that $P= \P_{A,B} = A \otimesk B$.
\item If  
\(P_{n}\) is a finitely
presented $A$-module, 
then 
\[
	\D^{n}_{K} \cong 
K \otimes_{A} \mathrm{Hom} _A
(P_n,A)
\]
and the canonical
morphism $ K \otimes _A 
\DK(B,A) \rightarrow \DK$ is
surjective.
\end{enumerate}
\end{corollary}
\begin{proof}
\begin{enumerate}
\item 
Recall from
(\ref{sweedlerprincipal}) that 
$\DK^n \cong \mathrm{Hom} _K
(\P_K^n,K)$ holds for all algebras
$K$. Lemma~\ref{improvedlemma} 
asserts that  
	\(\psi_{n} \colon 
K \otimes_{A} P_{n}
	\to \P^{n}_{K}\)
is an isomorphism, so in
combination, we obtain 
\begin{align*}
& \DK^{n}  \\
\cong \ &
\Hom_{K}(\P^{n}_{K},K) \\
\cong \ &
\Hom_{K}(K \otimes_{A} P_{n}, K) \\
\cong \ &
\Hom_{A}(P_{n},K).
\end{align*}
\item Under the isomorphism 
$\DK^n \cong 
\mathrm{Hom} _K(\P_K^n,K)$, 
we have
\begin{align*}	
& \D^{n}_K(B,A) \\
\cong \ &
	\{D \in \Hom_{K}(\P^{n}_{K},K)
		\mid
		D( \pi _n(a \otimesk b)) 
	\in A
	\, \forall a \in A, b \in B\},
\end{align*}
where $\pi_n$ is the composition 
$$
	A \otimesk B \rightarrow 
	K \otimesk K \epic 
	\P_K^n = (K \otimesk K) /
	I_K^{n+1}.
$$ 
Furthermore, 
we have by definition
\[
\Hom_{A}(P_{n},A) = 
	\{ D \in \Hom_{A}(P_{n}, K) \mid
	\mathrm{im}\, D \in A
	\},
\]
so we obtain the inclusion as
stated. 
If $ P = A \otimesk B$, then 
$ A \otimesk B
\epic P_n$ is
surjective, so this inclusion
is an identity.
\item Since \(K\) is flat over
\(A\)
	and  \(P_{n}\) is finitely presented, we have
(see e.g.
\cite[Proposition~I.2.9.10]{
bourbakiAlgebreCommutativeChapitres2006})
\[
	\Hom_{A}(P_{n}, K) \cong
	K \otimes_{A} \Hom_{A}(P_{n},A).
\]
As any tensor product 
$ K \otimes _A
- $ commutes
  with colimits, 
\begin{align*}
	\DK = \colim \DK^n &\cong
	\colim (K \otimes _A 
	\mathrm{Hom} _A(P_n,A)) \\
&\cong 
	K \otimes_A 
	\colim \mathrm{Hom}
_A(P_n,A) . 
\end{align*}
The inclusions 
$ \mathrm{Hom} _A (P_n,A)
\monic \DK^n(B,A)$ yield
a morphism 
\[ 
	\colim \mathrm{Hom} _A
(P_n,A) \rightarrow 
\colim \DK^n(B,A) = \DK(B,A).
\]
By construction, 
the resulting morphism 
$ \DK \rightarrow K \otimes _A 
\DK(B,A)$ splits the canonical
map $ K \otimes _A \DK(B,A)
\rightarrow \DK$ given by
multiplication, so the latter
is surjective. \qedhere
\end{enumerate}
\end{proof}

The following two corollaries
describe instances of
particular interest. The first
is the case $A=B$, where we
recover Másson's
results \cite[Theorems 2.2.2 and 2.2.5]{massonRingsDifferentialOperators1991}

\begin{corollary}\label{massonres}
If \(\iota \colon A \rightarrow K\) is  
flat and formally étale, then
\[
	\P^{n}_{K} \cong 
	K \otimes_{A} \P^{n}_{A} \cong
	K \otimes_{A} \P^{n}_{A,A}.
\]
If $\iota$ is in addition
injective, then
\(
	\D^{n}_K(A,A) \cong
	\D^{n}_{A} \cong
	\Hom_{A}(\P^{n}_{A,A},A).
\)
If  
\(\P_A^{n}\) is a finitely
presented $A$-module, 
then 
\(
	\D^{n}_{K} \cong 
K \otimes_{A} \DA^n.
\)
\end{corollary}

\begin{remark}
\label{InAnottorsionless}
In this case, the submodule 
\(I^{(n)}_{A,A} / I^{n}_{A}\)
of \(\P^{n}_{A}\) consists
of the elements \(\sum_i \pi _n(a_i
\otimesk b_i)\) such that 
\(\sum_i a_i D(b_i) = 0\) holds 
for all 
\(D \in \D^{n}_{A}\),
that is,
\(I^{(n)}_{A,A}/I^{n}_{A}\)
is the kernel of the canonical
morphism \(\P^{n}_{A} \to (\P^{n}_{A})^{**}\).
If \(A\) is smooth,
then  this module is zero,
since \(\P^{n}_{A}\) 
is finitely generated projective,
thus torsionless.
We provide an example where
\(I^{(n)}_{A,A} \supsetneq I^{n}_{A}\)
in \cref{cuspnottorsionless}.
\end{remark}

In the setting of varieties, 
the second corollary is about 
(quasi)affine varieties 
$X,Y$ that are birationally
equivalent, hence both contain a
(quasi)affine variety $Z$ as
dense open
subset.
The most studied case is the
one in which $A$ is the
coordinate ring of an affine
variety $X$, $B$ is the
coordinate ring of its
normalisation $Y$, and $K$
is their common field of
fractions; this is considered 
e.g.~in  
\cite{smithDifferentialOperatorsAffine1988,
muhaskyDifferentialOperatorRing1988,
hartDifferentialOperatorsSingular1987,
chamarieWhenRingsDifferential1987,
ben-zviCuspsModules2004}.

\begin{corollary}\label{DKcolimit2}
Assume that
\begin{enumerate} 
\item $k$ is Noetherian
and $A,B$ are essentially of
finite type over $k$, 
\item $S_A \subset A$ and  
$S_B \subset B$ are
multiplicatively closed sets
that do not contain any zero
divisors and with
$S^{-1}_A A \cong S_B^{-1}B=:K$,
and  
\item $\P^0_{A,B} \subseteq K$ is
a finitely generated
$A$-module.
\end{enumerate}
Then the canonical map 
$ K \otimes _A \DK(B,A) 
\rightarrow \DK$ is
surjective. 
\end{corollary}
\begin{proof}
Condition (2) ensures that 
the canonical maps 
$A \rightarrow S^{-1}_A A=K$
and $B \rightarrow S^{-1}_BB =
K$ are flat and étale 
embeddings and epimorphisms. 
In particular, $ \mu
_0$ is an isomorphism
(Example~\ref{muoiso}), so the
assumptions of
Lemma~\ref{improvedlemma} are
all satisfied. From now on, we
assume that $A,B \subseteq K$. 

In order to be able
to apply
Corollary~\ref{DKcolimit} (3),
we need to work with a system 
of generalised principal parts 
$P_n$ consisting of finitely
presented $A$-modules. As
such, we may use 
$$
	P_n := \P^n_A \otimes _A 
	\P^0_{A,B},
$$
where 
$$
	\P^0_{A,B}
	= \{\sum_i a_ib_i \in K \mid 
	a_i \in A,b_i \in B\}
$$ 
is the subalgebra of $K$ generated
by both $A$ and $B$ and the
tensor product $ \otimes _A$ 
relates multiplication by $a
\in A$ in $\P^0_{A,B}$ to the
right $A$-module structure 
$
	[u \otimesk v] a := 
	[u \otimesk va]$,
$ 
	u,v,a \in A
$
on $\P^n_A = (A \otimesk
A)/I_A^n$. 
Condition (1) ensures that 
$A$ is Noetherian and that  
the $A$-modules $\P^n_A$ are
finitely generated
(Proposition~\ref{dprops}
(2)). So in combination with
condition (3) we obtain that  
$P_n$ is a finitely
generated $A$-module. As $A$ is
Noetherian, $P_n$ is finitely
presented.  
\end{proof}
\begin{example}
Condition (3) is 
violated when $A = k[t]
\subseteq K =k[t,t^{-1}]$ and
$B \subseteq K$ is the
subalgebra of all polynomials
in $t^{-1}$. 
\end{example}
\begin{example}
If $ A = k[X],B=k[Y]$ and 
$AB=\P_{A,B}^0 =k[Z] \subset K$ are 
the coordinate rings of affine
varieties, then multiplication
$A \otimesk B \rightarrow AB$
induces a diagonal embedding  
$Z \rightarrow X \times Y$. That 
$A \rightarrow AB$ is injective and
that $AB$ is a finitely generated
$A$-module means that composition 
of $Z \rightarrow X \times Y$ with
the projection onto $X$ is a
morphism $Z \rightarrow X$ with
dense image (in the Zariski
topology) and with finite fibres.
Similarly, $Z \rightarrow Y$ has
dense image. 
\end{example}

\subsection{Descent for $C = \DK$}\label{sec53}
Let us summarise the results so far. 

\begin{theorem}
\label{descentdiffops}
Assume that $k$ is Noetherian and 
\(A \subseteq K,B \subseteq K\)
are algebras
essentially of finite type over
\(k\)
such that  \(A \monic K\) is
a flat epimorphism, 
$B \monic K$ is formally
étale, and $AB \subseteq K$ 
is a finitely
generated $A$-module. 
If \(\DK(B,A)\) 
is \(R(B,A)\)-locally 
projective, where
\begin{align*}
	R(B,A) &:=
	\{D \mapsto 
	\sum_i a_i D(b_i) \mid 
	\sum_i a_i \otimesk b_i  
	\in A \otimesk B\} \\
&\subseteq 
	\mathrm{Hom} _A(\DK(B,A),A),
\end{align*}
then we have:
\begin{enumerate}
\item 
Any bialgebroid structure
on \(\D_{K}\) over $K$
uniquely restricts to a
\coring{B}{A}-coring structure on  
\(\DK(B,A)\).
\item 
If \(A = B\), this turns
\(\DA\) into a bialgebroid over
\(A\).
\item If $\D_{K}$ is in addition 
a left Hopf
algebroid over $K$, then $\D_{A}$ is
a left Hopf algebroid over $A$. 
\end{enumerate}
\end{theorem}
\begin{corollary}\label{descentdiffopscor}
If $\DA$, $\D_{B}$,
and $\DK(B,A)$
are respectively 
$R(A)$-,
{$R(B)$-,} and 
$R(B,A)$-locally
projective,
then the functor
$$
	\DK(B,A) \otimes_{\D_{B}} -:
\D_{B} \jmod \to \DA \jmod
$$
is canonically opmonoidal.
\end{corollary}

\begin{corollary}
Let $k$ be a field of
characteristic 0 and 
\(A\) be an integral domain
which is an algebra over $k$
that is essentially of finite
type. 
If  \(\DA\) is
\(R(A)\)-locally projective,
then  \(\DA\) is a cocommutative,
conilpotent left Hopf algebroid
over  \(A\).
If $A$ is Gorenstein with trivial 
canonical module $\omega_A \cong A$, 
then $\DA$ admits an antipode. 
\end{corollary}

\begin{proof}
It is shown in
\cite[Theorem~15.2.6]
{mcconnellNoncommutativeNoetherianRings2001}
that $A$ contains an element 
$c \neq 0$ such that $\Omega_K^1$
is a finitely generated 
free $K$-module, where
$K:=A_c$ is the localisation
of $A$ at $S = \{c^n\}_{n \in
\mathbb{N} }$. It follows that 
$\Derk(K)$ is finitely
generated free, so 
Theorem~\ref{luis} implies
that $\DK \cong \UK$, hence is
a left Hopf algebroid with the
desired properties.    
Now use the theorem with $A=B=AB$, 
$K=A_c$ to obtain the left Hopf
algebroid structure. The existence
of an antipode follows in the
Gorenstein case from 
\cref{yekutielirem}.   
\end{proof}

\begin{example}
\label{exa:almostfin}
If \(A\) has 
almost finite projective differentials,
then \(\DA\) is strongly \(R(A)\)-locally
projective, as follows from 
\cref{stronglocprojcolim}.
\end{example}

\begin{example}
If $A$ is purely inseparable 
(i.e.~if $I_A^{n+1}=0$ as in
Example~\ref{weirdex}) but not
necessarily finitely generated
projective over $k$, 
then $\DA = \EndA$ is 
$R(A)$-locally projective 
if and only if $\DA$ is
finitely generated projective
over $A$ and $ \mathrm{Hom}
_A(\DA,A) = R(A)$. 
\end{example}

\section{Application to numerical
semigroups}\label{numsgsec}
Finally, we
discuss a class of
algebras $A,B \subseteq K=k[t,t^{-1
}]$ for which all
conditions of
Theorem~\ref{descentdiffops} and 
Corollaries~\ref{massonres} and 
\ref{DKcolimit2} are met,
namely the semigroup
algebras of 
numerical semigroups $\A,\B$
(cofinite
submonoids of $ (\mathbb{N},+)$)
\cite{assiNumericalSemigroupsApplications2020,barucciNumericalSemigroupsIMNS2020,rosalesNumericalSemigroups2009}. In
algebro-geometric terms, these
form a well-behaved class of
affine curves with a unique
cuspidal singularity, the
so-called monomial curves
\cite{perkinsCommutativeSubalgebrasRing1989,eriksenDifferentialOperatorsMonomial2003a,barucciDifferentialOperatorsNumerical2013}.  

Throughout this
Section~\ref{numsgsec}, $k$
denotes a field of
characteristic 0. 

\subsection{Numerical semigroups}
\label{sec:numsem}
The algebras $A,B$ 
to which we want to
apply the general theory
developed above are associated
to the following type of monoids:

\begin{definition}
A \emph{numerical semigroup} 
is an additive monoid 
\(\A \subseteq \N\) such that
\(\N \setminus \A\)
is finite. Its \emph{Frobenius
number} is
$F(\A):=\max (\N \setminus \A)$, and
$\A$ is \emph{symmetric} 
if for any $ z \in \mathbb{N}
\setminus \A$, we have 
$F(\A) - z \in \A$.
\end{definition}

So a submonoid 
$\A \subseteq \mathbb{N} $ is
numerical if and only if it
can be generated by finitely
many elements $a_1,\ldots,a_l
\in \mathbb{N} $ whose
greatest common divisor is 1,
or, equivalently, 
if the subgroup of \(\Z\)
generated by  \(\A\) is all of
\(\Z\). 

\begin{remark}
To determine whether the numerical 
semigroup
generated by $a_1,\ldots,a_l$ is
symmetric is intricate. However,
Sylvester has shown that all
numerical semigroups generated by
two coprime numbers $a_1,a_2$ are
symmetric \cite{
sylvesterSubvariantsSemiInvariantsBinary1882}.
\end{remark}

For the remainder of this article,
we fix numerical semigroups
\(\A, \B , \C\),
and we denote by
\[
	A \coloneqq k[\mathcal{A}],
	\quad
	B \coloneqq k[\mathcal{B}]
\]
the corresponding semigroup
algebras. That is, $A$ is the free
$k$-module with basis $\A$
whose unit map and
multiplication are obtained by
linearisation of the monoid
structure
$$
	\{0\} \hookrightarrow 
	\A,\quad + \colon 
	\A \times \A \rightarrow \A.
$$
Equivalently, 
an element $ f \in A$ can be
interpreted as a finitely supported
function $ \A \rightarrow k$. Then
the product $ fg $ of $ f,g \in A$
is given by convolution 
$$
	(f*g) (a) \coloneqq
	\sum_{b,c \in \A \atop a=b+c} 
	f(b) g(c)
$$
and the unit element $1 \in A$ is
the delta-function $ \delta _0$
supported in the unit element 
$ 0 \in \A$. 
As all numerical semigroups are
submonoids of $ \mathbb{Z} $, all
their convolution algebras embed
naturally into the group algebra 
of $ \mathbb{Z} $, that is, the
algebra of Laurent polynomials
\[
	k[\Z] \cong 
	K \coloneqq k[t,t^{-1}].
\]
We therefore view $A,B$ from now on
as subalgebras of $K$, identifying 
$a \in \A$ (respectively the
delta-function $ \delta _a$) with 
$ t^a \in K$. 

\begin{remark}
The Frobenius
number $F(\A)$ is sometimes also called the
\emph{conductor} of $\A$.
Recall that more generally, the
conductor of an integral
domain $A$ is the annihilator of the $A$-module
$B/A$, where $B$ is the normalisation
of $A$ (the integral closure
of $A$ in its field
of fractions). For any numerical
semigroup algebra $A$, its
normalisation is $B=k[t]$, and
it follows straight from the
definition that the conductor
in this sense is the ideal in
$A$ generated by
$t^{F(\A)+1}$. 
\end{remark}

Viewed 
algebro-geometrically, $A$ thus
becomes the coordinate ring of a
\emph{monomial curve}. The
properties of such curves have been
studied intensively. In particular,
one has:

\begin{theorem}\label{kunzthm}
$A$ is Gorenstein if and only if
$\A$ is symmetric. 
In this case, \(A \cong \omega_{A}\).
\end{theorem}
\begin{proof}
See 
\cite[Theorem]{
kunzValuesemigroupOnedimensionalGorenstein1970}
for the main statement.
The canonical module was
described explicitly in
\cite[proof of Theorem 4.1]{
trungAffineSemigroupsCohenMacaulay1986}.
We have
\[
	\omega_{A}
	=
	[K/A,k]
	\cong
	k[t^{-n} \colon n \in \Z \setminus \A],
\]
where \([K/A,k]\)
is
the internal hom as in
Remark~\ref{internalhomrem}.
If \(\A\) is symmetric, then the involution
\(
	\nu \colon \Z \to \Z\),
\(	a \mapsto F(\A) - a,
\)
maps \(\A\) bijectively to \(\Z \setminus \A\),
and vice versa, so that we have
an isomorphism 
of graded \(A\)-modules \(\bar{\nu} \colon A \to
	 \omega_{A}\),
\(t^{a} \mapsto t^{-\nu(a)}.\) 
\end{proof}

In the theory of differential
operators over semigroup
algebras, the following sets play a
central role (see 
e.g.~\cite{saitoDifferentialAlgebrasSemigroup2001,
saitoFiniteGenerationRings2004,
saitoNoetherianPropertiesRings2009}
for some background).  

\begin{definition}
For \(d \in \Z\), let 
\[
\ZZ_{d}(\B, \A)
\coloneqq 
\B \setminus (-d + \A)
=
\{i \in \B \mid 
i + d \not \in \A\}.
\] 
\end{definition}

We shall write \(\ZZ_{d}(\A)\)
for \(\ZZ_{d}(\A,\A)\). 

The following lemma collects some of
the properties of these sets that we
will use
below:

\begin{lemma}\label{lemmaz1}
For all \(d \in \Z\), we have:
\begin{enumerate}
\item 
The set \(\ZZ_{d}(\B,\A)\)
	 is finite.
\item If $\B \subseteq
\C $, then we have $\ZZ_d(\B,\A)
\subseteq \ZZ_d(\C,\A)$.
\item If 
$\A \subseteq \C$, then we
have  
$\ZZ_d(\B,\C) \subseteq 
\ZZ_d(\B,\A)$.  
\end{enumerate}
\end{lemma}
\begin{proof}
``(1)'':
For \(d \in \Z\), we have
\[\ZZ_{d}(\B,\A)
	 \subseteq
	 \{i \in \mathbb{N} \mid i + d < 0\} 
	 \sqcup
	 \{i \in \mathbb{N} \mid
	i+d \in \N \setminus \A\},
\]where $\sqcup$ denotes the
disjoint union.
Both sets on the right hand
side are finite, thus
so is \(\ZZ_{d}(\B,\A)\).

The claims (2) and (3) follow
directly from the definition. 
\end{proof}

\begin{lemma}\label{lemmaz2}
\label{setsZd}
For all \(d,e \in \Z\) and all \(a \in \A\), we have:
\begin{enumerate}
\item \(\ZZ_{a} (\B,\A) = \emptyset\)
	if \(a > F(\A)\),
\item \(\ZZ_{a} (\B,\A) = \emptyset\)
if \(\B \subseteq \A\),
\item
	\(\ZZ_{d+e} (\B,\A)
	\subseteq \ZZ_{d} (\B,\A)
	\sqcup (-d + \ZZ_{e}(\A))\),
\item \(\ZZ_{d+a} (\B,\A) \subseteq
\ZZ_{d}(\B,\A)\); more precisely, we
have 
	\[
		\ZZ_{d}(\B,\A) =
		\ZZ_{d+a}(\B,\A) \sqcup
		\left(-d-a + \ZZ_{-a}(\A)\right)
		\cap \B,
	\]
\item \(\ZZ_{d} (\B,\A) = 
\bigcap_{b \in \B, \, b+d \in \A}
	\ZZ_{-b}(\B,\A)\).
\end{enumerate}
\end{lemma}
\begin{proof}
``(1)'': 
If $ a > F(\A)$ and $ i \in
\mathbb{N} $,
then $i+a > F(\A)$, hence 
$i+a \in \A$. Thus 
$\ZZ_a( \mathbb{N} ,\A)
=\emptyset$. So by
Lemma~\ref{lemmaz1} (2), 
$\ZZ_a( \B ,\A) = \emptyset$. 

``(2)'': 
Similarly, as $\A$ is closed
under addition, $ \ZZ_a(\A,\A)
= \emptyset$, so by
Lemma~\ref{lemmaz1} (3), 
$\ZZ_a(\B,\A) = \emptyset $ if 
$\B \subseteq \A$. 

``(3)'': If $i \in
\ZZ_{d+e}(\B,\A)$ then 
$i \in \B$ and 
$ i+d+e \notin \A$. If in
addition $ i
\notin \ZZ_d(\B,A)$ this means  
that $ j:=i+d \in \A$, so 
$ j \in \ZZ_e(\A)$ and 
$ i \in -d + \ZZ_e(\A)$. 

``(4)'': If $i \in
\ZZ_{d+a}(\B,\A)$, then 
$ i \in \B$ and 
$i+d+a \notin \A$, so 
we also have $i+d \notin \A$
as $a \in \A$ and $\A$ is
closed under addition, that
is, $i \in \ZZ_{d}(\B,\A)$. If
conversely $i \in \ZZ_d(\B,\A)$
but $i \notin
\ZZ_{d+a}(\B,\A)$, then 
$i \in \B$ and $ i+d \notin
\A$ but $j:=i+d+a \in \A$, so 
$j \in \ZZ_{-a}(\A)$ and 
$ i \in \B \cap -d-a+
\ZZ_{-a}(\A)$.  

``(5)'': The inclusion
$\subseteq$ follows from 
(4) as $d = -b +
(b+d)$, where $b+d \in \A$.
For the reverse inclusion 
$\supseteq$, 
assume that 
$i \in \B \setminus
\ZZ_d(\B,\A)$. 
Then $i+d \in \A$. But 
for all $i \in \B$, we have
$ i \notin \ZZ_{-i}(\B,\A)$. 
Therefore, $i$ is not in the
intersection on the right hand
side. 
\end{proof}

Taking $\B=\A$, we in
particular obtain: 

\begin{corollary}
\label{monoidZ-a}
For any \(a,b \in \A\), we
have:
\begin{enumerate}
\item \(\ZZ_{a} (\A) =
\emptyset\),
\item \(\ZZ_{-a-b} (\A)
	= \ZZ_{-a} (\A) \sqcup
	(a + \ZZ_{-b}(\A))\).
\end{enumerate}
\end{corollary}

\begin{remark}
\(\ZZ_{-a}(\A)\) is
called 
the \emph{Apéry set} of \(a \in \A\).
This set consists of  
\(|\ZZ_{-a}(\A)| = a\) elements
(the map $\{0,\ldots,a-1\}
\rightarrow \ZZ_{-a}(\A)$ that maps 
$j$ to the smallest element in $\A$
of the form $j+na$ for some $n
\ge 0$ is
a bijection). In particular, 
the assignment 
$a \mapsto  \ZZ_{-a} (\A)$ 
is injective and identifies 
$\A$ with a set of subsets of 
$ \mathbb{Z} $. The above
corollary describes the
resulting monoid 
structure on the set of
Apéry sets. 
\end{remark}

\subsection{The Laurent
polynomials}\label{lpsec}
In Example~\ref{weylalgebra},
we have already remarked that 
for \(K =
k[t,t^{-1}]\), the ring of
differential operators is 
the localised Weyl algebra
$$
	\DK \cong \UK \cong 
	k \langle t,u,\partial
\rangle / \langle \partial t -
t \partial -1,tu-1,ut-1
\rangle.
$$

In particular, the $ \mathbb{Z}
$-grading of the Laurent polynomials
given by the degree induces a 
$ \mathbb{Z} $-grading of 
$\DK$
as discussed at the end of
Section~\ref{diffopsubsect},
with the degree $d$ part of 
$\DK$ given by
$$
	\D_{K,d} = 
	\{ D \in \D_K \mid 
	\forall r \in \mathbb{Z} 
	\exists \lambda _r \in k : 
	D (t^r) = \lambda_r
	t^{r+d}\}.
$$
The aim of this brief
subsection is to use this 
to express $\DK$
as follows as a crossed
product:
the group \(\Z\) acts on
the ring of polynomials
\(k[x]\) by the automorphism
given by \(\sigma(x) = x + 1\).
We denote by $k[x] \hash
\mathbb{Z} $ the resulting 
crossed product
(see
e.g.~\cite[Example 2.7]{montgomeryHopfGaloisTheory2009}
for this notion; other authors
speak of semidirect products,
smash products, or skew-group
rings rather than crossed
products):

\begin{definition}
The \emph{crossed product}
\(k[x] \hash \Z\) is the algebra
whose underlying vector space
is  \(k[x] \otimesk
k[t,t^{-1}]\),
endowed with the multiplication
\[
	(f \otimesk t^d)(g
\otimesk t^e)
	:=
	\sigma^{e}(f) g
	\otimes t^{d+e},\quad 
	f,g \in k[x],d,e \in
\mathbb{Z}.
\]
\end{definition}
\begin{remark}
Note this multiplication is 
the opposite of the one
usually described 
in Hopf algebra
text books.
\end{remark}

We will adopt the standard
notation and write 
$f \hash p$ instead of $f
\otimesk p$ for $f \in
k[x]$, $p \in k[t,t^{-1}]$ when 
considering the element in
this crossed product. 
Observe that $ k[x] \hash
\mathbb{Z} $ is $ \mathbb{Z}
$-graded,
where the component of degree \(d\)
is given by  \(k[x] \hash t^d\).

\begin{lemma}
\label{isoDKlem}
Let \(D_{0}\coloneqq t\partial \in \D_{K}\).
Then
\[
	\Phi \colon k[x] \hash \Z
	\to \D_{K},
	\quad
x \hash t^d \mapsto t^{d} D_{0}
\]
defines an isomorphism
of graded algebras.
\end{lemma}
\begin{proof}
This is straightforward, see
e.g.~\cite{perkinsCommutativeSubalgebrasRing1989} or
\cite{eriksenDifferentialOperatorsMonomial2003a}. 
\end{proof}

Explicitly, if 
\(\sum_{d \in \mathbb{Z}} 
f_d \hash t^d \in k[x] \hash \Z\),
with all but finitely many 
\(f_d\) being zero, then 
\(D = \Phi
\left(\sum_d f_d \hash
t^d\right) \) is given by
\[
	D(t^a) =
	\sum_d f_d(a) t^{d+a}.
\]

For the sake of readability,
we will omit the isomorphism \(\Phi\)
and 
will identify the elements
of \(k[x] \hash \Z\) with their
image in  \(\DK\).

\begin{remark}
In particular, \(\D_{K,d} = t^{d} k[D_{0}]\).
Furthermore, 
we also have 
\(\D_{K,d} =  k[D_{0}] t^{d}\)
since  \([D_{0}, D] = d D\) for any 
\(D \in \D_{K,d}\).
Note also that the set
$\D^n_K$ of
differential operators of
order at most $n$ corresponds
under $ \Phi$ to 
\(k[x]_{\le n} \hash \Z\).
\end{remark}

\subsection{Semigroup algebras}
We now find ourselves in the setting of
Corollary~\ref{DKcolimit2}
and are able to describe 
\[
	\D_{K,d}(B,A) \coloneqq \D(B,A) \cap \D_{K,d}
\]
using the sets
\(\ZZ_{d}(\B,\A)\).
We refer e.g. to 
\cite{saitoDifferentialAlgebrasSemigroup2001, 
saitoFiniteGenerationRings2004,
saitoNoetherianPropertiesRings2009}
for a more general treatment of rings of
differential operators over
general semigroup algebras,
and to \cite{perkinsCommutativeSubalgebrasRing1989, 
eriksenDifferentialOperatorsMonomial2003a,
barucciNumericalSemigroupsIMNS2020}
specifically for numerical
semigroups, that is, monomial curves.

For \(d \in \Z\), consider the
ideal
\[
	I_{d}(\B,\A) 
	\coloneqq 
	\{f \in k[x] \mid f(b) = 0
\, \forall b \in \ZZ_{d}(\B,\A) \}
\subseteq k[x].
\]

It follows directly 
from Lemma~\ref{setsZd}
that we have:

\begin{lemma}
	\label{Idlem}
For \(d \in \Z\) and  \(a \in \A\), 
we have:
\begin{enumerate}
\item \(I_{a}(\B,\A) = 
k[x]\) if \(a > F(\A)\)
	or  \(\B \subseteq \A\),
\item \(I_{d+e}(\B,\A) \supseteq
	I_{d}(\B,\A) \sigma^{d}(I_{e}(\A))\),
\item \(I_{d + a} (\B,\A) 
\supseteq I_{d}(\B,\A)\),
\item \label{Idisgcd}
\(I_{d}(\B,\A) =
\sum_{b \in \B, b+d \in \A} 
I_{b}(\B,\A)\).
\end{enumerate}
\end{lemma}

Furthermore, we have:

\begin{lemma}\label{unstop}
\(
\Phi (I_{d} (\B,\A) 
\hash t^d) =
\D_{K,d}(B,A).\) 
\end{lemma}
\begin{proof}
Let \(D = \Phi(f\hash t^d) 
\in \D_{K,d}(B,A)\).
Recall that \(D(t^{a}) = 
f(a)t^{d+a}\),
thus \(D(t^a) \in A\) if and only
\(f(a) = 0\) when 
\(d+a \not \in \A\).
\end{proof}

\begin{remark}
All statements made so far about
\(\ZZ_{d}(\B,\A)\)
and \(I_{d}(\B,\A)\) 
can be easily generalised to higher
dimensions, meaning to 
semigroups
\(\A,\B \subseteq \N^{n}\)
which generate \(\Z^{n}\) as an
abelian group. Geometrically, these
are higher-dimensional affine spaces
in which the origin gets replaced by
a cuspidal singularity. 
Their rings of differential operators
were studied in
\cite{saitoDifferentialAlgebrasSemigroup2001, 
saitoFiniteGenerationRings2004,
saitoNoetherianPropertiesRings2009},
and we are hopeful that 
the techniques introduced here
can be used to investigate
the local projectivity 
in the general case.
The main differences to the 
one-dimensional
case that have to be dealt with 
are the facts that  
\(\ZZ_{d}(\B,\A)\)
is no longer finite and 
that 
the ring of polynomials in
more than one variable 
is not a principal ideal
domain, 
which will be crucially used 
in the proof of
\cref{refpolylocproj}. 
\end{remark}

\begin{definition}\label{elldef}
For \(d \in \Z\), we define
\[
	\ell_d 
	\coloneqq 
	\prod_{b \in 
	\ZZ_{d}(\B,\A)}(x - b)
	\quad\mbox{and}\quad
	L_{d} \coloneqq 
	\Phi(\ell_{d} \hash t^{d}).
\]
\end{definition}

Then
\(I_{d}(\B,\A) = \langle \ell_{d} \rangle\) 
and
\(\D_{K,d}(B,A) = L_{d} k[D_{0}]
= k[D_{0}]L_{d}\).
When $\A=\B$, it follows from 
Corollary~\ref{monoidZ-a} that 
for \(a,b \in \A\), we have
\(L_{-a-b} = L_{-a}L_{-b}\). 
Thus we recover a result by Perkins:
\begin{proposition}[{
\cite[Theorem 4.2]{perkinsCommutativeSubalgebrasRing1989}}]
\label{L-asubalg}
The subalgebra of \(\DA\)
generated by the set
\(\{L_{-a} \mid a \in \A\}\) 
is isomorphic to \(A\).
\end{proposition}

\begin{remark}
By the above, we can express
 \(\DA\) as the crossed product
 \(k[x]\hash_{\tau}\Z\) which is
twisted by the 2-cocycle
(we refer again 
to~\cite[Example 2.7]{montgomeryHopfGaloisTheory2009}
for the notion of twisted
crossed products, also called cleft
extensions)
  \[
\tau \colon \Z \times \Z \to k[x],
\quad
(d,e) \mapsto 
\frac{\sigma^{e}(\ell_{d}) \ell_{e}}
{\ell_{d+e}}.
 \]
\end{remark}

\subsection{Strong local
projectivity}
\label{sec:numsemstrongloc}
Using its description given above,
we are now able to prove that
\(\DK(B,A)\) is  \(R(B,A)\)-locally
projective.

\begin{lemma}
\label{refpolylocproj}
For any \(d \in \Z\) and  \(f \in
I_{d}(\B,\A)\),
there exist 
\(b_{1}, \ldots, b_{n} \in \B\)
such that 
\begin{enumerate}
\item \(d+b_{i} \in \A\) and
\(f(b_{i}) \neq 0\)	
for all \(i\),
\item \(b_{i} \in
\ZZ_{-b_{j}}(\B,\A)\) if \(i \neq j\), and
\item \(f \in I_{-b_{1}}(\B,\A) +
\ldots + I_{-b_{n}}(\B,\A)\).
\end{enumerate}
\end{lemma}
\begin{proof}
Assume $f \neq 0$ (otherwise
the statement is clear). Set 
\[
	S_0 := \{ b \in \B \mid 
	f(b) \neq 0\}.
\]
As $ f \neq 0$, $S_0$ is 
not empty. 

We inductively construct a
descending chain 
\[
	S_0 \supset
	S_1 \supset \ldots \supset 
	S_n \supset S_{n+1}=\emptyset
\]
of finite subsets of $S_0$ by 
letting $b_{i+1} \in S_i$
be the smallest element and 
setting $S_{i+1} := S_i \cap
\ZZ_{-b_{i+1}} ( \B, \A)$.
For $ i>0$, $S_i \subseteq 
\ZZ_{-b_i}(\B,\A)$ is finite
by Lemma~\ref{lemmaz1} (1).
Furthermore, we have 
$S_i \supsetneq S_{i+1}$: we
have $b_i \in S_i$, but 
$b_i \notin \ZZ_{-b_i}(\B,\A)$   
by definition, so $b_i \in 
S_i \setminus S_{i+1}$. 
Consequently, the induction
terminates after finitely many
steps in some
$S_{n+1}=\emptyset.$

In this way, we have
constructed $b_1,\ldots,b_n
\in \B$. As they are all taken
from $S_0$, we have $f(b_i)
\neq 0$ for all $i$.
Recall furthermore that 
$f \in I_d(\B,\A)$, so
$f(b)=0$ if $ b \in \ZZ_d(\B,A)$,
that is, we have
$$
	S_0 \cap \ZZ_d(\B,\A) = 
	\emptyset.
$$
Thus $b_i \notin \ZZ_d(\B,\A)$,
which means 
$b_i +d \in \A$ and we have shown
(1). 

Furthermore, the way we have
chosen $b_j$ for $j>i$ shows 
$b_j \in \ZZ_{-b_i}(\B,\A)$
and $b_j > b_i$. The latter
implies that we also have 
$ b_i \in \ZZ_{-b_j}(\B,\A)$
as $b_i-b_j < 0$. Thus we have
shown (2). 

To show (3), 
recall that 
$ I_d(\B,\A) = \langle \ell_d
\rangle$, so 
\[
	\sum_{i=1}^n 
	I_{-b_i} (\B,\A) = 
	\langle 
	\mathrm{gcd}
	(\ell_{-b_1},\ldots,
	\ell_{-b_n} ) 
	\rangle. 
\]
By the Definition~\ref{elldef} of
$\ell_d$, we have
\[
	\mathrm{gcd}
	(\ell_{-b_1},\ldots,
	\ell_{-b_n} ) 
	= 
	\prod_{b \in T} 
	(x-b),\quad 
	T:= 	
	\bigcap_{i=1}^n 
	\ZZ_{-b_i} (\B,\A).
\]
In view of (1) which is already
shown, we have 
$$
	T \supseteq
	\bigcap_{b \in \B,
	d+b \in \A} \ZZ_{-b} 
	(\B,\A) = 
	\ZZ_d(\B,\A),
$$
where the last equality was
established in Lemma~\ref{lemmaz2}
(5). The fact that 
$S_{n+1}=\emptyset$ means
$S_0 \cap T = \emptyset$. 
Put differently, 
if $ b \in T$, 
then $f(b)=0$, so 
$ f$ is divisible by 
$ \prod_{b \in T} (x-b)$ and (3) is
shown.  
\end{proof}

In particular, this shows:

\begin{proposition}\label{mayday}
For $ d \in \mathbb{Z} $
and 
$ f \in I_d (\B,\A)$, let 
$b_1,\ldots,b_n \in \B$ be as
in the previous lemma. Then
there exist
\(f_{i} \in I_{-b_{i}}(\B,\A)\), \(i = 1,\ldots,n\),
such that
\begin{enumerate} 
\item \(f_{i}(b_{j}) = \delta_{ij}\),
\item \(f_{i}(b) = 0\) for all \(b \in \B\)
such that \(f(b) = 0\),
and
\item  \(
	f = \sum_{i=1}^{n} f(b_{i}) f_{i}.
\)
\end{enumerate}
\end{proposition}
\begin{proof}
Let \(h \in k[x]\) such that
\(f = \prod_{b \in T}(x-b)
h\), where $T$ is as in the
lemma. Using (3) from there, 
we express 
$\prod_{b \in T}(x-b)$ as 
$g_1+\cdots+g_n$ with 
$g_j \in I_{-b_j}(\B,\A)$,
so that \(f = g_{1}h + \cdots + g_{n}h\). 
If \(b \in \B\) such that
\(f(b) = 0\),
then either  \(b \in T\)
or  \(h(b) = 0\),
hence  \(g_{i}(b)h(b) = 0\)
for all \(i=1,\ldots,n\).
Now 
(2) from the lemma 
implies that $g_j(b_i) =0 $ 
for $i \neq j$, hence 
$g_i(b_i)h(b_{i}) = f(b_i) \neq 0$ 
for all
$i=1,\ldots,n$. 
Rescaling  
the $g_ih$ to 
$f_i:=\frac{1}{f(b_i)} g_i h$ 
therefore yields elements as wanted. 
\end{proof}

Let us introduce some notation
before proving the main result.
Let us abbreviate \(M \coloneqq \DK(B,A)\)
and  \(R \coloneqq R(B,A)\).
In order to prove that \(M\) is  \(R\)-locally
projective, we introduce a subset
\(S \subseteq \hhom_{A}^{R}(M,M)\) 
which satisfies the three properties in
\cref{stronglocidem}:
for \(b \in \B\), we define the functional
\[
	r_{b} \colon D \mapsto D(b) \in R(B,A)
\]
and consider the set \(S\) of all elements
in \(\hhom_{A}^{R}(M,M)\) of the form
\[
	\pi \coloneqq 
	\sum_{b \in \B}
	r_{b}(-)
	(f_{b} \hash t^{-b}),
\]
where all but finitely many \(f_{b}\)
are nonzero, and we assume
\begin{enumerate}
\item \(f_{b} \in I_{-b}(\B,\A)\) for all \(b \in \B\), and
\item  \(f_{b}(a) = \delta_{ab}\) for all \(a,b \in 
\supp(\pi) \coloneqq \{ b \in \B \mid
f_{b} \neq 0 \}\).
\end{enumerate}

\begin{lemma}
Let \(D \coloneqq g \hash t^{d}
\in \D_{K,d}(B,A).\)
Then
\[
\pi (D)
=
\sum_{b \in \B}
g(b) f_{b}
\hash t^{d}.
\]
\end{lemma}
\begin{proof}
We have
\begin{align*}
\pi(D)
 =
\sum_{b \in \B}
r_{b}(D)
\,
(f_{b} \hash t^{-b})
 =
\sum_{b \in \B}
g(b)t^{d+b}
\,
(f_{b} \hash t^{-b})
 =
\sum_{b \in \B}
g(b)f_{b} \hash t^{d}.
\end{align*}
We have
\(g(b)f_{b} \hash t^{d} \in \D_{K,d}(B,A)\)
 for all \(b \in \B\),
since \(I_{-b}(\B,\A) 
\subseteq I_{d}(\B,\A)\)
if \(d+b \in \B\) and 
\(g(b) = 0\) otherwise.
The sum on the right hand side is thus
well-defined.
\end{proof}

The following lists the
properties of $S$ that we have
to establish and will use:

\begin{lemma}
	\label{Sproperties}
	Let
\[
	\alpha \coloneqq 
	\sum_{b \in \B}
	r_{b}(-)
(f_{b} \hash t^{-b}),
	\quad
	\beta \coloneqq 
	\sum_{b \in \B}
	r_{b}(-)
	(g_{b} \hash t^{-b})
	\in S.
\]
We have
\begin{enumerate}
\item \(\alpha 
(f_{b} \hash t^{-b}) = f_{b} \hash t^{-b}\)
for all \(b \in \B\),
\item \(\alpha = \beta\) if and only if
 \(f_{b} = g_{b}\) for all \(b \in \B\),
\item we have
\[
\beta \circ \alpha
=
\sum_{b \in \B}
r_{b}(-)(h_{b} \hash t^{-b}),
\text{ where }
h_{b} =
\sum_{a \in \B}
f_{b}(a) g_{a}
\hash t^{-b},
\]
\item 
	\label{Spropidem}
	\(\alpha\) and  \(\beta\)
	are idempotent, 
\item \(\alpha \circ \beta = 0\)
if and only if
 \(g_{b}(a) = 0\)
for all \(a \in \supp(\alpha)\), and
\item \label{Sprop2}
	\(\alpha + \beta - \beta \circ \alpha \in S\)
	if  \(\alpha \circ \beta = 0\).
\end{enumerate}
\end{lemma}
\begin{proof}
\begin{enumerate}
\item We have
\[
\alpha (f_{b} \hash t^{-b})
=
\sum_{a \in \B}
f_{b}(a) f_{a}
\hash t^{-b}
=
\sum_{a \in \B}
\delta_{ab} f_{a}
\hash t^{-b}
=
f_{b}
\hash t^{-b}.
\]
\item Suppose \(\alpha = \beta\)
	and let \(\Lambda \coloneqq 
	\supp(\alpha) \cup \supp(\beta)\).
We choose \(d \in \Z\) large enough
so that  \(d+b \in \A\)
for all \(b \in \Lambda\).
We define 
for each \(b \in \Lambda\)
the polynomial
 \[
	h_{b} \coloneqq 
	\ell_{d}
	\prod_{i \in \Lambda,
	i \neq b}
	(x - i)
\]
We chose
\(d\) in such a way that \(\ell_{d}(b) \neq 0\),
so \(h_{b}(b) \neq 0\),
and we can define
\(\bar h_{b} \coloneqq 
\frac{1}{h_{b}(b)} h_{b}\).
Then we have \(\bar h_{b}(a) = \delta_{ab}\)
 for all \(a,b \in \Lambda\), and
it follows that for all \(b \in \B\)
\[
f_{b} \hash t^{d}
=	
\alpha(\bar h_{b} \hash t^{d})
=
\beta(\bar h_{b} \hash t^{d})
=
g_{b} \hash t^{d}.
\]
\item Let \(D \coloneqq h \hash t^{d} \in \D_{K,d}(B,A)\).
We have
\begin{align*}
	\beta \circ \alpha (D)
&	= 
	\beta \left(
		\sum_{b \in \B} 
		h(b) f_{b}
		\hash t^{d}
		\right)
= 
\sum_{b \in \B} 
h(b) 
\beta \left(
f_{b}
\hash t^{d}
\right) \\
& = 
\sum_{b \in \B} 
h(b)
\left(
\sum_{a \in \B} 
f_{b}(a) g_{a}
\hash t^{d}
\right) \\
& = 
\sum_{b \in \B} 
h(b)
\left(
\sum_{a \in \B} 
f_{b}(a) g_{a}
\right)
\hash t^{d},
\end{align*}
thus the claim follows.
\item We have
\begin{align*}
\alpha^{2}
& =
\sum_{b \in \B}
r_{b}(-)
\left(
\sum_{a \in \B}
f_{b}(a) f_{a}
\right) 
\hash t^{-b}
\\
&=
\sum_{b \in \B}
r_{b}(-)
\left(
\sum_{a \in \supp(\alpha)}
f_{b}(a) f_{a}
\right)
\hash t^{-b} \\
& = 
\sum_{b \in \B}
r_{b}(-)
\left(
f_{b}
\hash t^{-b}
\right)
=
\alpha.
\end{align*}
\item Suppose
\[
\alpha \circ \beta
=
\sum_{b \in \B}
r_{b}(-)
\left(
\sum_{a \in \B}
g_{b}(a) f_{a}
\right) 
\hash t^{-b}
= 0
.\]
Then 
\(\sum_{a \in \B}
g_{b}(a) f_{a} = 0\)
for all \(b \in \B\).
For all \(c \in \supp(\alpha)\),
we have
$f_a(c) = \delta _{ac}$,
hence we obtain
\[
g_{b}	(c) 
= \sum_{a \in \supp(\alpha)}
g_{b}(a) f_{a}(c)
= \sum_{a \in \B}
g_{b}(a) f_{a}(c) = 0.
\]
\item Let \(\pi \coloneqq
\alpha + \beta - \beta \circ \alpha\).
Then
\[
\pi 
=
\sum_{b \in \B}
r_{b}(-)
(h_{b}
\hash t^{-b}),
\text{ where }
h_{b}	\coloneqq 
f_{b} + g_{b}
- \sum_{a \in \B}
f_{b}(a) g_{a}.
\]
We have  \(h_{b} = 0\) if \(b \not \in
\supp(\alpha) \cup \supp(\beta)\).
Since \(\alpha \circ \beta = 0\),
 \(g_{b}(c) = 0\) if \(c \in \supp(\alpha)\)
so that
\[
	h_{b}(c)
	=
	\begin{cases}
		g_{b}(c) & \text{if } c \in \supp(\beta), \\
	f_{b}(c) & \text{if } c \in \supp(\alpha)
	\setminus \supp(\beta).
	\end{cases}
\]
It follows that
\[
	h_{b}(a) = \delta_{ab}
 \quad \forall a,b \in
 \supp(\pi) =
 \supp(\alpha) \cup \supp(\beta),
\]
so \(\pi \in S\).
\end{enumerate}
\end{proof}

\begin{lemma}
\label{Spropbeta}
Let \(D\in \D_{K,d}(B,A)\)
and \(\alpha \in S\).
Then, there exists \(\beta \in S\)
such that
 \[
\beta(D - \alpha(D))
=
D - \alpha(D)
\text{ and }
\alpha \circ \beta = 0.
\]
\end{lemma}
\begin{proof}
Suppose that \(D =: g \hash t^{d}\)
and  
\(\alpha =:
\sum_{b \in \B}
r_{b}(-)
(f_{b} \hash t^{-b}).\)
Let \(g' \in I_{d}(\B,\A)\) 
such that
\[
	g' \hash t^{d}
	=
	D - \alpha(D)
	=
\left(
	g - 
	\sum_{b \in \B}
	g(b) f_{b}
	\right)
	\hash t^{d}.
\]
In particular,
\(\alpha(g' \hash t^{d})
= 0\),
so we have
\(g'(a) = r_{a} \circ 
\alpha(g' \hash t^{d}) = 0\) 
for all  \(a \in \supp(\alpha)\).

By Lemma~\ref{refpolylocproj},
there exist \(b_{1},\ldots,b_{n} \in \B\)
and polynomials \(g_{i} \in I_{-b_{i}}(\B,\A)\),
\(i=1,\ldots,n\), such that
\(g_{i}(b_{j}) = \delta_{ij}\),
\(g_{i}(b) = 0\) for all \(b \in \B\)
for which \(g'(b) = 0\), and 
\[
	g' = \sum_{i=1}^{n} g'(b_{i}) g_{i},
\]
Define 
\[
\beta \coloneqq 
\sum_{b \in \B}
r_{b}(-)
(g_{b} \hash t^{-b}),
\text{ where }
g_{b} = 
\begin{cases}
	g_{i}	& \text{if } b = b_{i}, \\
	0 & \text{else.}
\end{cases}
\]
Then straightforward computations
show that \(\beta\) satisfies
the desired properties.
\end{proof}

\begin{theorem}
	\label{numsemslc}
Let \(\A,\B\) be numerical semigroups,
and \(A,B \subseteq
K=k[t,t^{-1}]\) be their semigroup algebras.
Then \(\DK(B,A)\) is strongly
 \(R(B,A)\)-locally projective.
\end{theorem}
\begin{proof}
The set \(S\) satisfies all
the properties 
in \cref{stronglocidem}.
To begin with, all the elements
in \(S\) are idempotent 
(\cref{Sproperties}\cref{Spropidem}).
Furthermore,
\(0 \in S\) which is (1) in 
Lemma~\ref{stronglocidem}, 
while property (2)
was shown
in
\cref{Sproperties}\cref{Sprop2}.
Finally, property (3) was
shown in  
\cref{Spropbeta}.
Strong local projectivity thus follows.
\end{proof}

\subsection{Main result}
Let us summarise our results:
\begin{theorem}
\label{thmsemigroupalg}
Let \(\A,\B\) be numerical semigroups,
and \(A,B \subseteq
K=k[t,t^{-1}]\) be their semigroup algebras.
Then, we have:
\begin{enumerate}
\item 
There is a canonical 
\coring{B}{A}-coring structure on  
\(\DK(B,A)\).
\item 
This turns
\(\DA=\DK(A,A)\) and 
\(\DB=\DK(B,B)\) into left Hopf algebroids over
\(A\) respectively \(B\).
\item 
	\label{symantipode}
	If $\A$ is symmetric, 
then $\DA$ admits an antipode. 
\item $\DK(B,A) \otimes _{\DB} -
\colon \DB\jmod \rightarrow
\DA\jmod$ is a monoidal 
equivalence.
\end{enumerate}
\end{theorem}
\begin{proof}
All assumptions of 
Theorem~\ref{descentdiffops}
are satisfied and we deduce
(1) and (2). Theorem~\ref{kunzthm} 
stated that for a symmetric
numerical semigroup, $A$ is
Gorenstein and hence satisfies all
assumptions of
Theorem~\ref{eamonthm}, so 
$\DA \cong \D^\op_A$. Thus (3) 
follows from
Proposition~\ref{antipodprop}. 

Smith and
Stafford and independently 
Muhasky have shown that 
the $\DK(B,A)$ are Morita
equivalences (they consider 
$B = \mathbb{N} $, the general
claim follows by
transitivity) 
\cite{smithDifferentialOperatorsAffine1988, 
muhaskyDifferentialOperatorRing1988}.
By
Corollary~\ref{descentdiffopscor},
this Morita equivalence is
monoidal.   
\end{proof}

\begin{remark}
By construction, the coproduct
on $\DK(B,A)$ is 
$ \mathbb{Z} $-graded, 
$ \Delta (\D_{K,d}(B,A))
\subseteq \bigoplus_{i+j=d} 
\D_{K,i} (B,A) \otimes _A 
\D_{K,j} (B,A)$. 
The filtration given by the
order of a differential
operator agrees with the
primitive filtration
(cf.~\cref{primfiltrrem}).
However, this is not a
coalgebra
filtration, as illustrated
in \cref{examplenotfilt}.
\end{remark}

\begin{corollary}
Let \(a_{1}, \ldots, a_{n}\) be a set
of generators of \(\A\).
Then \(\DA\) is generated as
an algebra by  \(t^{a_{1}},\ldots,t^{a_{n}}\),
\(D_{0}\),
and 
\(L_{-a_{1}}, \ldots, L_{-a_{n}}\).
\end{corollary}

\begin{remark}
Several authors have described 
sets of generators of 
$\DA$ when $A$ is a semigroup algebra of 
\(\A \subseteq \Z^{n}\),
see e.g.~\cite{perkinsCommutativeSubalgebrasRing1989,
eriksenDifferentialOperatorsMonomial2003a,
barucciDifferentialOperatorsNumerical2013,
saitoFiniteGenerationRings2004,
saitoNoetherianPropertiesRings2009}.
The motivation there was to
study $\gr(\DA)$ and hence the 
generators were chosen to be of minimal
possible order. The ones we use
are chosen in order to
establish the strong local
projectivity of $\DA$. 
\end{remark}

\begin{corollary}
If \(\A\) is symmetric,
then the antipode  
\(S \colon \DA \to \DA^{op}\) 
in \cref{thmsemigroupalg}\cref{symantipode}
is given by
\begin{align*}
	S(D_{0})  = F(\A) - D_{0}, 
	\quad
	S(L_{-a})  = (-1)^{a} L_{-a},
\quad a \in \A.
\end{align*}
More precisely,
\[
	S(f \hash t^{d}) = f(-x-d+F(\A)) \hash t^{d},
	\quad
	f \hash t^{d} \in \DA.
\]
\end{corollary}
\begin{proof}
We have seen in \cref{dopaugment}
	that \(S\)
	respects both the filtration
	and the $ \mathbb{Z} $-grading of  \(\DA\).
	Therefore, we have
	\begin{align*}
		S(D_{0})  \in \D_{A,0}^{1}
		&=
		\vspan_{k} \{ D_{0}, 1 \}, \\
		\quad
		S(L_{-a})  \in \D_{A,-a}^{a}
		&=
		\vspan_{k} \{ L_{-a} \},
		\quad a \in \A.
	\end{align*}
	In addition, \cref{sfilt} yields
for the associated graded
morphism
\[
	\gr (S)([D_{0}]) = -[D_{0}],
	\quad
	\gr (S)([L_{-a}]) = (-1)^{a}[L_{-a}].
\]
Let us describe in more detail
the right action of \(\DA\) on  \(A\).
It has been shown in \cref{kunzthm}
that 
the shift map
\(\bar{\nu} \colon 
 	t^{a} \mapsto t^{a-F(\A)}\) 
induces an isomorphism
of 
$ \mathbb{Z} $-graded \(A\)-modules
 \(A \cong \omega_{A}\),
where
\[
	\omega_{A}= [K/A, k]
\cong
k[t^{-n}  \colon n \in \Z \setminus \A].
\]
Since \(K\) and  \(A\) are 
$ \mathbb{Z} $-graded
left \(\DA\)-modules, so is their quotient
\(K/A\). It follows that
\(\omega_{A}\)
is a $ \mathbb{Z} $-graded right
\(\DA\)-module
via precomposition.
For \(D = f \hash t^{d} \in \DA\)
and \(t^{-n} \in \omega_{A}\),
we have
\[
	t^{-n} \circ D = 
	\begin{cases}
		f(n-d) t^{-n+d} & 
		\text{if } n-d \in \Z \setminus \A, \\
		0 & \text{if } n-d \in \A.
	\end{cases}
\]
The antipode 
of an element \(D = f \hash t^{d} \in \DA\)
is then
obtained as follows:
\begin{align*}
	S(D)(t^{a}) 
	& = \bar{\nu}^{-1}(\bar{\nu}(t^{a}) \circ D) \\
	& = \bar{\nu}^{-1}(t^{a-F(\A)} \circ D) \\
	& = \bar{\nu}^{-1}
	\left(f(F(\A)-a-d) t^{a-F(\A)+d} \right) \\
	& = f\big(F(\A)-a-d\big) t^{a+d},
\end{align*}
thus the claim follows.
\end{proof}

\subsection{Computing $
\Delta (D)$}
Finally, we describe a
technique for computing the
coproduct of a given
differential operator
explicitly. 

We identify \(\DK \otimes_{K} \DK\)
with \(k[x,y]\hash \Z\)
via the correspondence
 \[
	 (f \hash t^{d}) \otimes_{K}
	 (g \hash t^{e})
	 \mapsto
	 f(x)g(y) \hash t^{d+e}.
\]
This is an isomorphism of $k$-vector spaces,
where
\[
	(\DK \otimes_{K} \DK)_{d}
	\cong k[x,y] \hash t^{d},
	\quad
  d \in \Z.
\]
The order filtration on \(\DK\),
whose graded component is
 \[
	\D_{K,d}^{n} \cong \{ 
		f \hash t^{d}
		\mid
		\deg f \le n
	\}
\]
induces the canonical filtration
on \(\DK \otimes_{K} \DK\),
with graded component
\[
	(\DK \otimes_{K} \DK)^{n}_{d}
	\cong
	\{
  f(x)g(y) \hash t^{d}
	\mid
	\deg f + \deg g \le n
	\}.
\]
In this presentation, the comultiplication
is given as follows:
\begin{lemma}
For \(f \hash t^{d} \in \D_{K,d}\),
we have
\[
	\Delta(f \hash t^{d})
	=
	f(x+y) \hash t^{d}
	\in 
	(\DK \otimes_{K} \DK)^{n}_{d}.
\]
\end{lemma}
\begin{proof}
Recall, \(\DK \cong \UK\) is the 
universal enveloping algebra
of a Lie--Rinehart algebra
(see \cref{weylalgebra}).
We have
\[
	\Delta(D_{0})
	=
	1 \otimes_{K} D_{0}
	+
	D_{0} \otimes_{K} 1,
	\quad
	\Delta(t^{d})
	=
	t^{d} \otimes_{K} 1,
\]
where \(D_{0} = t \partial\)
was identified with 
\(x \hash 1 \in k[x] \hash \Z\)
in \cref{isoDKlem}.
The claim thus follows, since
\(\Delta\) is multiplicative.
\end{proof}

Let us now look at \(\DK(B,A)\).
We define the set
\begin{gather*}
I^{\otimes 2}_{d}(\B,\A)
	\coloneqq \\
	\{
		(f,g) \in k[x] \times k[y]
		\mid
		f \in I_{i}(\B,\A), \,
		g \in I_{j}(\B,\A), \,
		i,j \in \Z, \,
		i+j = d
	\}.
\end{gather*}
Then under our identification, we have 
\[
	\DK(B,A) \otimes_{A} \DK(B,A)
	\cong
	\vspan_{k}
	\{ 
		f(x)g(y) \hash t^{d}
		\mid
		(f,g) \in I^{\otimes 2}_{d}(\B,\A)
	\}.
\]

\begin{remark}
\label{remfiltration}
For a universal enveloping
algebra such as \(\U_K \cong
\DK\), the primitive
filtration is also a
$K$-coalgebra filtration. 
As explained in
\cref{polishedrem},
\(\DK(B,A) \otimes_{A} \DK(B,A)\)  
therefore carries two filtrations:
the canonical one with graded component
\[
	(\DK(B,A) \otimes_{A} \DK(B,A))^{n}_{d}
\]
which is the \(k\)-linear subspace
spanned  by the elements
\[
	\{ 
		f(x)g(y) \hash t^{d}
		\mid
		(f,g) \in I^{\otimes 2}_{d}(\B,\A), \,
		\deg f + \deg g \le n
	\},
\]
and the induced filtration,
with graded component
\[
	(\DK \otimes_{K} \DK)^{n}_{d}
	\cap
	(\DK(B,A) \otimes_{A} \DK(B,A)).
\]
The first is included in the second one,
but this inclusion is strict in general,
due to the fact that a polynomial of degree \(n\)
might be written as a linear combination
of polynomials of higher degree.
As mentioned before, the restriction
of \(\Delta\) to \(\DK(B,A)\)
is compatible with the induced filtration, 
but not necessarily with the canonical one.
We will illustrate this with a concrete
example in \cref{examplenotfilt}.
\end{remark}

The statement of
\cref{thmsemigroupalg}
translates into the following:

\begin{proposition}
For any \(f \hash t^{d} \in \DK(B,A)\),
\[
	f(x+y) \in 
	\langle g(x)h(y) \mid
	(g,h) \in I^{\otimes 2}_{d}
	\rangle
	\subseteq k[x,y].
\]
\end{proposition}

To find a presentation of
\(\Delta(f \hash t^{d})\)
 in 
\(\DK(B,A) \otimes_{A} \DK(B,A)\)
is thus equivalent to finding
a presentation for \(f(x+y)\)
inside this ideal,
which is not
unique in general.

In concrete computations, 
we can however take advantage of the
fact that \(\Delta\) is multiplicative
by realising \(\DK(B,A)\) as a module
over a suitably chose 
subring \(H \subseteq \DK\)
whose comultiplication is
known. Then we may use:

\begin{lemma}
Let \(\{g_{i}\}\) be a set of generators
of \(\DK(B,A)\) as a module over \(H\). 
Suppose we have two maps 
(not necessarily \(k\)-linear)
\[
	\Delta' \colon \{g_{i}\}
	\to \DK(B,A) \otimes \DK(B,A),
	\quad
	\Delta_{H} \colon H
	\to H \otimes H
\]
such that \(\Delta|_{\{g_{i}\}} 
= \iota' \circ \Delta'\)
and
\(\Delta|_{H} = \iota_{H} \circ \Delta_{H}\),
where
\[
	\iota' \colon \DK(B,A) \otimes \DK(B,A)
	\to \DK \otimes_{K} \DK,
	\quad
\iota_{H} \colon H \otimes H \to 
\DK \otimes_{K} \DK
\]
are the canonical maps.
Then the comultiplication 
of \(\DK(B,A)\)
is uniquely determined by
 \(\Delta'\) and  \(\Delta_{H}\).
\end{lemma}
\begin{proof}
Straightforward.
\end{proof}

So if we find a pair \((H,\Delta_{H})\),
then it suffices to determine
\(\Delta'\) on a set of generators.

\begin{example}
	We have seen in \cref{unstop}
	that \(\DK(B,A)\) is freely
	generated as a left (or right)
	\(k[D_{0}]\)-module  by the 
	elements \(L_{d}\), \(d \in \Z\).
	Since 
\(	\Delta(D_{0})
	=
	1 \otimes_{K} D_{0}
	+
	D_{0} \otimes_{K} 1\),
	\(H:=k[D_{0}]\)
	 has the desired property.
\end{example}

\begin{example}
\(\DK(B,A)\) is both a left  \(\DA\)-
and a right  \(\DB\)-module,
where  \(\DA\) and  \(\DB\)
are bialgebroids
contained in  \(\DK\).
Furthermore, since
by 
\cref{thmsemigroupalg}(4),
\(\DK(B,A)\)
is a Morita context, it is finitely
generated projective over
both rings.
\end{example}

\subsection{The example
$2 \mathbb{N}+3 \mathbb{N} $}
\label{sec:ex2N+3N}
We demonstrate the
use of the technique 
for the most basic example 
\[
	\mathcal{A} \coloneqq 2
\mathbb{N} + 3 \mathbb{N} = 
	\{i \in \mathbb{N} \mid 
	i \neq 1\},\quad 
	\B := \N.
\] 
Their semigroup algebras
\(A=k[t^2,t^3]\) and \(B=k[t]\)
are then the coordinate ring of the cusp
and of its normalisation,  
the affine line, respectively.

We mentioned briefly
in 
\cref{universalgpp,InAnottorsionless}
that \(I^{(n)}_{A,A}\) and
\(I^{n}_{A}\) are not equal
in general, and we show it
here concretely for \(A=k[t^{2},t^{3}]\):

\begin{lemma}
\label{cuspnottorsionless}
For any \(n \ge 3\),
the element
of lowest degree in 
\(I^{(n)}_{A,A}\) is
 \[
	\tau_{n} \coloneqq
	(1 \otimes t - t \otimes 1)^{n}
	(t^{2} \otimes 1
	+ n t \otimes t
	+ 1 \otimes t^{2})
	\in I^{(n)}_{A,A} 
	\setminus
	I^{n}_{A}.
\]
\end{lemma}
\begin{proof}
We have 
\(I^{(n)}_{A,A}
= (A \otimes A) \cap I_{K}^{n}\),
where \(I_{K}\) is
generated by 
\(1 \otimes t - t \otimes 1\).
Direct computations show
that \(\tau_{n} \in A \otimes A\),
and that it is
the element
of lowest degree in \(A \otimes A\)
which can be written as 
a multiple of 
\((1 \otimes t - t \otimes 1)^{n}\).
Furthermore,
the ideal
\(I_{A}\) of \(A \otimes A\)
is generated by
\(1 \otimes t^{2}
- t^{2} \otimes 1\) 
and \(1 \otimes t^{3}
- t^{3} \otimes 1\),
so the elements of \(I_{A}^{n}\)
are at least of degree \(2n\),
whereas \(\tau_{n}\) is
of degree \(n+2\),
hence \(\tau_{n} \not \in I^{n}_{A}\).
\end{proof}

The Apéry sets of $\A$ 
are as follows:
 \begin{lemma}
We have
\[
	\ZZ_{d}(\mathcal{A})
	= 
	\begin{cases}
		\emptyset & d \in \mathcal{A} \\
		{0} & d = 1 \\
		\{i \in \mathcal{A}
		\mid
	0 \le i \le -d+1\}
		\setminus \{-d\} & d < 0
	\end{cases}.
\]
\end{lemma}

The differential operators 
\(L_{1}, L_{0},L_{-1}, L_{-2}, L_{-3}\)
from Definition~\ref{elldef} 
are the generators of \(\DA\)
given by Smith and Stafford 
in \cite{smithDifferentialOperatorsAffine1988}.
We include in the table below their
presentation
in terms of  \(t,t^{-1}, \partial\),
as well as their corresponding
polynomials.
\[
\begin{array}{|c|l|l|}
	\hline
	D \in \DA &
t,t^{-1},\partial & \Phi^{-1}(D)  \\
	\hline
L_{1} & t^{2} \partial &  x \hash t \\
D_{0} & t \partial  & x \hash 1 \\
L_{0} & 1 & 1 \hash 1\\
L_{-1} & t \partial^{2} - \partial &
x(x-2) \hash t^{-1} \\
L_{-2} & \partial^{2} - 2 t^{-1} \partial & 
x(x-3) \hash t^{-2} \\
L_{-3} & \partial^{3} - 3 t^{-1} \partial^{2} + 3 t^{-2} \partial &
x(x-2)(x-4) \hash t^{-3} \\
	\hline
\end{array}
\]

Here are some example
relations that are obtained
from Proposition~\ref{mayday}
by taking $f \hash t^{d}$ to be 
$\Phi^{-1}
(D)$ for $D = D_0,
L_{-1}$, and $L_1$,
respectively; the first two
relations demonstrate that the
$g_i$ in that proposition are
not unique, nor is their
number $n$. 

\begin{lemma}
\label{relationsDA}
We have
\begin{align*}
	D_{0}  & = t^{2} (D_{0} - 1) L_{-2} 
	- t^{3} L_{-3} 
	\\
	D_{0}	 & =  t^{3} (D_{0} - 1) L_{-3} 
        - t^{4} L_{-4} -t^{2} L_{-2},
	\\
	L_{1} & =  t^{3} (D_{0} - 1) L_{-2} 
	- t^{4} L_{-3}, \\
	L_{-1} & =  t^{2} (D_{0} - 1) L_{-3} 
	- t^{3} L_{-4}
.\end{align*}
\end{lemma}
\begin{proof}
Apply both sides of the
equations to $t^i$.  
\end{proof}

\begin{remark}
We can use these relations to
inductively obtain the
following formal expression
for $D_0$: 
\[
	D_{0} = \sum_{a=2}^{\infty} t^{a} L_{-a}.
\]
This equation holds pointwise
when applying both sides to
some $t^i$. More conceptually,
$\DA$ is the continuous dual
of $A \otimesk A$, where the
topology on the latter is
given by the $I_A$-adic
filtration (its completion is
known as the \emph{jet
algebroid}), and the above
series converges in the
topology on the dual.  
\end{remark}

Here is an example which illustrates 
\cref{remfiltration} on
how
\(\Delta\) does not respect the order
filtration:

\begin{lemma}
\label{examplenotfilt}
\(L_{-2} \in \D^{2}_{A,-2}\), but
\(\Delta(L_{-2}) \not \in
(\DA \otimes_{A} \DA)^{2}_{-2}\).
\end{lemma}
\begin{proof}
We have \(L_{-2} = \ell_{-2}
\hash t^{-2}\),
where \(\ell_{-2} = x(x-3)\),
so that
\[
\Delta(L_{-2})
= \ell_{-2}(x+y) \hash t^{-2} \\
= [\ell_{-2}(x) + 2xy + \ell_{-2}(y)]
\hash t^{-2}.
\]
We have 
\(xy \hash t^{-2} \in 
(\DK \otimes_{K} \DK)^{2}_{-2}\)
and \(xy \hash t^{-2}
(\DA \otimes_{A} \DA)\),
since
\begin{align*}
xy \hash t^{-2} = 
x (y-3)\ell_{-2}(y) \hash t^{0-2}
- x\ell_{-3}(y) \hash t^{1-3}.
\end{align*}
However, \(xy \hash t^{-2} \not \in
(\DA \otimes_{A} \DA)^{2}_{-2}\).
Indeed,
\((\DA \otimes_{A} \DA)^{2}_{-2}\) 
is spanned over \(k\) by the elements 
 \(f(x)g(y)\hash t^{-2}\) where
\begin{equation}
\label{eq:DA2}
f \in I_{i}(\A), \,
g \in I_{j}(\A), \,
i+j = -2, \,
\deg f + \deg g \le 2.
\end{equation}
For \(f \hash t^{i} \in \DA\),
we have \(\deg f \ge \deg \ell_{i}\),
since \(\ell_{i}\) divides \(f\).
However,  \(\deg \ell_{i} \le 2\) 
if and only \(i \ge -2\),
and \(\deg \ell_{i} \le 1\)
 if and only if \(i \ge 0\).
It then follows that
the only possible 
\((i,j)\) in \cref{eq:DA2}
are \((-2,0)\) and  \((0,-2)\),
hence
\[
(\DA \otimes_{A} \DA)^{2}_{-2}
=
\vspan_{k}\{
\ell_{-2}(x) \hash t^{-2},
\ell_{-2}(y) \hash t^{-2}
\}.
\]
It is then easy to see that
\(xy \hash t^{-2}\) and thus
\(\Delta(L_{-2})\) are not
in this subspace.
\end{proof}

Before we move on, let us point to a property
of the polynomials \((\ell_{-a})_{a \in \A}\),
which is not true for other numerical semigroups.

\begin{lemma}
For all \(a \in \A = 2\N + 3\N\), 
the polynomial 
\(\bar{\ell}_{-a} \coloneqq 
	\frac{1}{\ell_{-a}(a)} \ell_{-a}\) 
is integer-valued.
\end{lemma}
\begin{proof}
Simple computations show that 
the claim holds for \(a = 0,2,3\),
and the general result is obtained
through an argument
by induction using the equality
\(\bar{\ell}_{-a-b} 
= \bar{\ell}_{-a}(a+b) 
\bar{\ell}_{-a}
\sigma^{a}(\bar{\ell}_{-b})\)
for all \(a,b \in \A\)
(\cref{L-asubalg}).
\end{proof}

\begin{remark}
As a result, the 
\(k\)-linear map
\[
	\bar{L}_{-a} :
	k[t,t^{-1}] \to k[t,t^{-1}],
	\quad
	t^{n} \mapsto 
	\bar{\ell}_{-a}(n) t^{n-a},
	\quad a \in \A
\]
restricts for a field $k$ of any
characteristic to a differential
operator on $A= k[t^2,t^3]$.  
However, if $ \mathrm{char}(k) >0$,
then the order of 
$\bar{L}_{-a}$ is in general lower
than $a$. 
For example,
when \(k\) is a field of characteristic 2,
we have
\(\bar{L}_{-3} = t^{-3} D_{0}
\in \DA^{1}\). 
In this case,
the structure of \(\DA\) has been
studied by Ludington 
in \cite{ludingtonCounterexampleTwoConjectures1977}:
\(\DA\) is generated as an \(A\)-ring
by  \(\bar{L}_{-3},\bar{L}_{-2^{e}},
e \ge 0\),
and the \(A\)-module \(\DA^{n}\) is free
for all \(n \in \N\),
with basis \(\bar{L}_{-2i},\bar{L}_{-2i-3}\),
\(i \le \frac{n}{2}\).
This implies in particular
that \(\DA\) is
strongly \(R(A)\)-locally projective.
We have \(\DK \cong K \hash k \langle x \rangle\),
where \(k\langle x \rangle\)
is the polynomial ring of divided
powers in one variable.
This has a canonical left Hopf
algebroid structure
that restricts to \(\DA\).
\end{remark}

Before we end, we also write out
explicitly the corings that
establish in the example of
the cusp $A=k[t^2,t^3]$ and
its normalisation $B=k[t]$ 
the monoidal Morita equivalence
between the closed monoidal
categories of $\DA$- and $\DB$-modules.

\begin{lemma}
\[
	\ZZ_{d}(\A,\B) =
	\begin{cases}
		\emptyset & d \in \B \\
		\{0\} & d = -1,-2 \\
		\{i \in \A \mid
		i < -d\} 
					& d < 0.
	\end{cases}
\]
\end{lemma}

Since $A \subseteq B$, we have 
\(1 \in \DK(A,B)\). Therefore
it is straightforward to
compute the comultiplication
in this coring as we can just
do so within $\DK$: 

\begin{lemma}
\(\DK(A,B) = \DB + L_{-2}\DB
= \DA + \DA L_{-2}\)
\begin{gather*}
	\Delta(1) = 1 \otimes_{B} 1, \\
	\Delta(L_{-2}) =
	L_{-2} \otimes_{B} 1
	+ 1 \otimes_{B} L_{-2}
,\end{gather*}
where \(L_{-2} \coloneqq L_{-2}(\A,\B)\).
\end{lemma}

For the inverse Morita
context, we get:

\begin{lemma}
\[
	\ZZ_{d}(\B,\A) =
	\begin{cases}
		\emptyset & d \in \A \\
		\{0\} & d = 1 \\
		\{i \in \B \mid
	i \le -d + 1\}
\setminus \{-d\} 
					& d \le 0.
	\end{cases}
\]
\end{lemma}

Here it is harder to determine
the comultiplication
since \(1 \not \in \DK(B,A)\), 
so when computing 
$ \Delta (D)$ for $D \in
\DK(B,A)$ the elementary
tensors occurring in $ \Delta
(D) \in \DK \otimes _K \DK$
are a priori not in 
$\DK(B,A) \otimes _A \DK(B,A)
\subseteq \DK \otimes _K \DK$. 
After rewriting them
appropriately, we obtain:

\begin{lemma}
\(\DK(B,A) = t^{2} \DB + L_{0} \DB
= \DA t^{2} + \DA L_{0}\).
\begin{align*}
	\Delta(t^{2}) &= 
	  L_{1} \otimes_{A} L_{1}
	- t^{2} \otimes_{A} L_{0}
	- L_{0} \otimes_{A} t^{2}D_{0},
	\\
	\Delta(L_{0}) &= 
	L_{0} \otimes_{A} L_{0}
	+ L_{-2}
	\otimes_{A} t^{2} D_{0}
	- (D_{0}-1)L_{-1}
	\otimes_{A} L_{1},
\end{align*}
where \(L_{a} \coloneqq L_{a}(\B,\A)\)
 for \(a = -1,0,1\).
\end{lemma}

\bibliographystyle{plainurl}

\end{document}